\documentclass[12pt]{amsart}
\usepackage{amssymb,amsmath,amsthm}
\usepackage[all,cmtip]{xy}
\usepackage{pdflscape}
\calclayout
\newtheorem{theorem}{Theorem}[section]
\newtheorem{proposition}[theorem]{Proposition}
\newtheorem{corollary}[theorem]{Corollary}
\newtheorem{lemma}[theorem]{Lemma}
\newtheorem{example}[theorem]{Example}
\newtheorem{definition}[theorem]{Definition}
\newtheorem{remark}[theorem]{Remark}
\numberwithin{theorem}{section} \numberwithin{equation}{section}
\newcommand\scalemath[2]{\scalebox{#1}{\mbox{\ensuremath{\displaystyle #2}}}}
\begin{document}
\title[Normal forms for Kummer surfaces]{Normal forms for Kummer surfaces}
\author{Adrian Clingher, Andreas Malmendier}
\address{Dept.~of Mathematics and Computer Science, University of Missouri -- St.~Louis, St.~Louis, MO 63121}
\email{clinghera@umsl.edu}
\address{Dept.~of Mathematics and Statistics, Utah State University, Logan, UT 84322}
\email{andreas.malmendier@usu.edu}
\begin{abstract}
We determine normal forms for the Kummer surfaces associated with abelian surfaces of polarization of type  $(1,1)$, $(1,2)$, $(2,2)$, $(2,4)$, and $(1,4)$. Explicit formulas for coordinates and moduli parameters in terms of Theta functions of genus two are also given.  The normal forms in question are closely connected to the generalized Riemann identities for Theta functions of Mumford's.
\end{abstract}
\maketitle
\section{Introduction}
Let $\mathbf{A}$ denote an abelian surface defined over the field of complex numbers and let $-\mathbb{I} $ be its minus identity involution. The quotient $\mathbf{A}/\langle -\mathbb{I} \rangle$ has sixteen ordinary double points and its minimum resolution, denoted $\operatorname{Kum}(\mathbf{A})$, is known as the \emph{Kummer surface} of  $\mathbf{A}$.  The rich geometry of these surfaces, as well as their strong connection with Theta functions have been the subject of multiple studies \cite{MR1097176,MR0419459,MR0296076,MR3366121,MR1182682,MR3712162,MR2369941} over the last century and a half.  
\par Polarizations on an abelian surface $\mathbf{A}\cong \mathbb{C}^2/\Lambda$ are known to correspond to positive definite hermitian forms $H$ on $\mathbb{C}^2$,  satisfying $E = \operatorname{Im} H(\Lambda,\Lambda) \subset \mathbb{Z}$.  In turn, such a hermitian form determines a line bundle $\mathcal{L}$ in the N\'eron-Severi group $\mathrm{NS}(\mathbf{A})$. One may always then choose a basis of $\Lambda$ such that $E$ is given  by a matrix $\bigl(\begin{smallmatrix} 0&D\\ -D&0 \end{smallmatrix} \bigr)$ with $D=\bigl(\begin{smallmatrix}d_1&0\\ 0&d_2 \end{smallmatrix} \bigr)$ where $d_1, d_2 \in \mathbb{N}$, $d_1, d_2 \ge 0 $, and $d_1$ divides $d_2$. The pair $(d_1, d_2)$ gives the {\it type} of the polarization. 
\par If $\mathbf{A}=\operatorname{Jac}(\mathcal{C})$ is the Jacobian of a smooth curve $\mathcal{C}$ of genus two, the hermitian  form associated to the divisor class $[\mathcal{C}]$ is a polarization of type $(1, 1)$ - a \emph{principal polarization}.  Conversely, a principally polarized abelian surface is either the Jacobian of a smooth curve of genus two or the product of two complex elliptic curves, with the product polarization. 
\par The present work focuses on Kummer surfaces $\operatorname{Kum}(\mathbf{A})$ associated with abelian surfaces $\mathbf{A}$ of principal polarization, as well as of polarizations of type $(1,2)$, $(2,2)$, $(2,4)$, and $(1,4)$. We present a detailed description of moduli parameters for these surfaces, as well as several {\it normal forms}, i.e., explicit projective equations describing Kummer surfaces $\operatorname{Kum}(\mathbf{A})$. The crucial ingredient in our considerations is given by the theory of classical Theta functions of genus two, as well as their associated generalized Riemann identities, as derived by Mumford in \cite{MR2352717}.
\section{Theta functions and Mumford identities}
An effective way to understand the geometry of Kummer surfaces is to use the Siegel modular forms and Theta functions of genus two. Let us enumerate here the main such forms that will be relevant to the present paper. For detailed references, we refer the reader to the classical papers of Igusa \cite{MR0141643, MR0168805}; for some further applications see also \cite{MR2935386,2018arXiv180607460C}.
\subsection{Theta functions and abelian surfaces}
\label{GeneralRemarks}
The Siegel three-fold is a quasi-projective variety of dimension $3$ obtained from the Siegel upper half-plane of degree two which by definition is the set of two-by-two symmetric matrices over $\mathbb{C}$ whose imaginary part is positive definite, i.e.,
\begin{equation}
\label{Siegel_tau}
 \mathbb{H}_2 = \left. \left\lbrace \tau = \left( \begin{array}{cc} \tau_{11} & \tau_{12}  \\ \tau_{12}  & \tau_{22} \end{array} \right) \right|
 \tau_{11}, \tau_{22}, \tau_{12} \in \mathbb{C}, \, \operatorname{Im}\!{(\tau_{11})}  \operatorname{Im}\!{(\tau_{22}}) > \operatorname{Im}\!{(\tau_{21})}^2, \, \operatorname{Im}\!{(\tau_{22})} > 0 \right\rbrace ,
\end{equation}
divided by the action of the modular transformations $\Gamma_2:=
\operatorname{Sp}_4(\mathbb{Z})$, i.e., 
\begin{equation}
 \mathcal{A}_2 =  \mathbb{H}_2 / \Gamma_2 \;.
\end{equation} 
Each $\tau \in \mathbb{H}_2$ determines the principally polarized complex abelian surface $\mathbf{A}_{\tau} = \mathbb{C}^2 / \langle \mathbb{Z}^2 \oplus \tau \, \mathbb{Z}^2\rangle$ with period matrix $(\mathbb{I}_2,\tau) \in \mathrm{Mat}(2, 4;\mathbb{C})$. The canonical principal polarization  $\mathcal{L}$ of $\mathbf{A}_{\tau}$ is determined by the Riemann form  $E( x_1 + x_2 \tau, y_1  + y_2 \tau)=x_1^t\cdot y_2 - y_1^t\cdot x_2$ on  $\mathbb{Z}^2 \oplus \tau \, \mathbb{Z}^2$. Two abelian surfaces $\mathbf{A}_{\tau}$  and $\mathbf{A}_{\tau'}$ are isomorphic if and only if there is a symplectic matrix 
\begin{equation}
M=  \left(\begin{array}{cc} A & B \\ C & D \end{array} \right) \in \Gamma_2
\end{equation}
such that $\tau' = M (\tau):=(A\tau+B)(C\tau+D)^{-1}$. Since $M$ preserves the Riemann form, it follows that the Siegel three-fold $\mathcal{A}_2$ is the set of isomorphism classes of principally polarized abelian surfaces. Similarly, we define the subgroup $\Gamma_2(2n) = \lbrace M \in \Gamma_2 | \, M \equiv \mathbb{I} \mod{2n}\rbrace$ and Igusa's congruence subgroups $\Gamma_2(2n, 4n) = \lbrace M \in \Gamma_2(2n) | \, \operatorname{diag}(B) =  \operatorname{diag}(C) \equiv \mathbb{I} \mod{4n}\rbrace$ with corresponding Siegel modular threefolds $\mathcal{A}_2(2)$, $\mathcal{A}_2(2,4)$, and $\mathcal{A}_2(4,8)$ such that
\begin{equation}
 \Gamma_2/\Gamma_2(2)\cong S_6, \quad  \Gamma_2(2)/\Gamma_2(2,4)\cong (\mathbb{Z}/2\mathbb{Z})^4, \quad \Gamma_2(2,4)/\Gamma_2(4,8)\cong (\mathbb{Z}/2\mathbb{Z})^9,
\end{equation}
where $S_6$ is the permutation group of six elements. $\mathcal{A}_2(2)$ is the three-dimensional moduli space of principally polarized abelian surfaces with level-two structure. The meaning of $\mathcal{A}_2(2,4)$ and $\mathcal{A}_2(4,8)$ will be discussed below.
\par For an elliptic variable $z\in \mathbb{C}^2$ and modular variable $\tau\in \mathbb{H}_2$, Riemann's Theta function is defined by setting
\[ 
 \vartheta \big( z, \tau \big) = \sum_{u\in \mathbb{Z}^2} e^{\pi i\, ( u^t \cdot \tau \cdot  u+ 2 u^t \cdot z)  } . 
 \]
The Theta function is holomorphic on $\mathbb{C}^2\times \mathbb{H}_2$. For rational-valued vector $\vec{a}, \vec{b} \in \mathbb{Q}^2$, a Theta function is defined by setting
\[ 
 \vartheta\!\begin{bmatrix} \vec{a} \\ \vec{b} \end{bmatrix}\!\!(z, \tau)
 =  \sum_{u\in \mathbb{Z}^2} e^{\pi i \, \big((u+\vec{a})^t \cdot \tau  \cdot  (u+\vec{a}) + 2 \, (u+\vec{a})^t \cdot (z+\vec{b})\big)  }.
 \]
 For a rational matrix $\left( \begin{smallmatrix} a_1&a_2\\ b_1&b_2 \end{smallmatrix} \right)$, which we call a \emph{theta characteristic}, we set $\vec{a}=\langle a_1, a_2\rangle/2$ and $\vec{b}=\langle b_1, b_2\rangle/2$ and define 
 \[ 
 \theta\!\begin{bmatrix} a_1 & a_2 \\ b_1 & b_2 \end{bmatrix}\!\!(z, \tau)
 = \vartheta\!\begin{bmatrix} \vec{a} \\ \vec{b} \end{bmatrix}\!\!(z, \tau).
 \]
In this way, all standard Theta functions have characteristics with coefficients in $\mathbb{F}_2$. This will make it easier to relate them to the description of the $16_6$ configuration in finite geometry in Section~\ref{ss:FiniteGeometry}. For $\left( \begin{smallmatrix} a_1&a_2\\ b_1&b_2 \end{smallmatrix} \right) \in \mathbb{F}_2^4$ -- where $\mathbb{F}_2$ denotes the finite field with two elements -- there are sixteen corresponding Theta functions; 10 are even and 6 are odd functions according to
\begin{equation}
 \theta\!\begin{bmatrix} a_1 & a_2 \\ b_1 & b_2 \end{bmatrix}\!\!(-z, \tau)
  = (-1)^{a^t\cdot b} \;   \theta\!\begin{bmatrix} a_1 & a_2 \\ b_1 & b_2 \end{bmatrix}\!\!(z, \tau).
\end{equation}
We denote the even Theta functions by 
\begin{small}
\[
\begin{split}
\theta_1 	= \theta\!\begin{bmatrix} 0 &  0 \\ 0 & 0 \end{bmatrix}\!\!, \,  
\theta_2	= \theta\!\begin{bmatrix} 0 &  0 \\ 1 & 1 \end{bmatrix}\!\!, \,  
\theta_3 &	= \theta\!\begin{bmatrix} 0 &  0 \\ 1 & 0 \end{bmatrix}\!\!, \,  
\theta_4 	= \theta\!\begin{bmatrix} 0 &  0 \\ 0 & 1 \end{bmatrix}\!\!, \,   \;
\theta_5 	= \theta\!\begin{bmatrix} 1 &  0 \\ 0 & 0 \end{bmatrix}\!\!, \\
\theta_6	= \theta\!\begin{bmatrix} 1 &  0 \\ 0 & 1 \end{bmatrix}\!\!,  \, 
\theta_7	= \theta\!\begin{bmatrix} 0 &  1 \\ 0 & 0 \end{bmatrix}\!\!, \,  
\theta_8 &	= \theta\!\begin{bmatrix} 1 &  1 \\ 0 & 0 \end{bmatrix}\!\!, \,  
\theta_9 	= \theta\!\begin{bmatrix} 0 &  1 \\ 1 & 0 \end{bmatrix}\!\!, \,  
\theta_{10}=\theta\!\begin{bmatrix} 1 &  1 \\ 1 & 1 \end{bmatrix}\!\!.
\end{split}
\]
\end{small}
We denote the odd Theta functions by 
\begin{small}
\[
\begin{split}
    \theta_{11}	=\theta\!\begin{bmatrix} 0 &  1 \\ 0 & 1 \end{bmatrix}\!\!, \,  
    \theta_{12}&	=\theta\!\begin{bmatrix} 0 &  1 \\ 1 & 1 \end{bmatrix}\!\!, \,  
    \theta_{13}	=\theta\!\begin{bmatrix} 1 &  1 \\ 0 & 1 \end{bmatrix}\!\!, \\
    \theta_{14}	=\theta\!\begin{bmatrix} 1 &  0 \\ 1 & 0 \end{bmatrix}\!\!, \,  
    \theta_{15}&	=\theta\!\begin{bmatrix} 1 &  0 \\ 1 & 1 \end{bmatrix}\!\!, \,  
    \theta_{16}	=\theta\!\begin{bmatrix} 1 &  1 \\ 1 & 0 \end{bmatrix}\!\!.
\end{split}
\]
\end{small}
A scalar obtained by evaluating a Theta function at $z=0$ is called a Theta constant.  The six odd Theta functions give trivial Theta constants.  We write
\begin{equation}
\label{Eqn:theta_short}
 \theta_i(z) \quad \text{instead of} \quad 
  \theta\!\begin{bmatrix} a^{(i)}_1 & a^{(i)}_2 \\ b^{(i)}_1 & b^{(i)}_2 \end{bmatrix}\!\!(z, \tau)
 \quad \text{where $i=1,\dots ,16$,}
\end{equation}
and $\theta_i =\theta_i(0)$ such that $\theta_i=0$ for $i=11,\dots,16$. 
\par According to \cite[Sec.~3]{MR2367218} we have the following \emph{Frobenius identities} relating Theta constants:
\begin{equation}
\label{Eq:FrobeniusIdentities}
\begin{array}{lllclll}
 \theta_5^2 \theta_6^2 & = & \theta_1^2 \theta_4^2 - \theta_2^2 \theta_3^2 \,, &\qquad
 \theta_5^4 + \theta_6^4 & =& \theta_1^4 - \theta_2^4 - \theta_3^4 + \theta_4^4 \,, \\[0.2em]
 \theta_7^2 \theta_9^2 & = & \theta_1^2 \theta_3^2 - \theta_2^2 \theta_4^2 \,, &\qquad
 \theta_7^4 + \theta_9^4 &= & \theta_1^4 - \theta_2^4 + \theta_3^4 - \theta_4^4 \, , \\[0.2em]
 \theta_8^2 \theta_{10}^2 & = & \theta_1^2 \theta_2^2 - \theta_3^2 \theta_4^2 \, , &\qquad
 \theta_8^4 + \theta_{10}^4 & = & \theta_1^4 + \theta_2^4 - \theta_3^4 - \theta_4^4 \,.
\end{array}
\end{equation}
\subsubsection{Doubling formulas for Theta constants}
We will also introduce the Theta functions that are evaluated at $2 \tau$. Under duplication of the modular variable, the Theta functions $\theta_1, \theta_5, \theta_7, \theta_8$ play a role dual to $\theta_1, \theta_2, \theta_3, \theta_4$. We renumber the former and use the symbol $\Theta$ to mark the fact that they are evaluated at the isogenous abelian variety.
That is, we will denote Theta functions  with doubled modular variable by 
\begin{equation}
\label{Eqn:Theta_short}
 \Theta_i(z)  \quad \text{instead of} \quad 
  \theta\!\begin{bmatrix} b^{(i)}_1 & b^{(i)}_2 \\ a^{(i)}_1 & a^{(i)}_2 \end{bmatrix}\!\!(z, 2 \tau)
 \quad \text{where $i=1,\dots ,16$,}
\end{equation}
and $\Theta_i =\Theta_i(0)$. The following identities are called the \emph{second principal transformations of degree two}~\cite{MR0141643, MR0168805} for Theta constants:
\begin{equation}
\label{Eq:degree2doubling}
\begin{array}{lllclll}
 \theta_1^2 & = & \Theta_1^2 + \Theta_2^2 + \Theta_3^2 + \Theta_4^2 \,, &\qquad
 \theta_2^2 & =&  \Theta_1^2 + \Theta_2^2 - \Theta_3^2 - \Theta_4^2 \,, \\[0.2em]
 \theta_3^2 & = &  \Theta_1^2 - \Theta_2^2 - \Theta_3^2 + \Theta_4^2 \,, &\qquad
 \theta_4^2 &= &  \Theta_1^2 - \Theta_2^2 + \Theta_3^2 - \Theta_4^2 \,.
\end{array}
\end{equation}
%or, equivalently, 
%\[
%\begin{array}{lllclll}
%
% 4 \, \Theta_1^2 & = & \theta_1^2 + \theta_2^2 + \theta_3^2 + \theta_4^2 \,, &\qquad
%
% 4 \, \Theta_2^2 & =&  \theta_1^2 + \theta_2^2 - \theta_3^2 - \theta_4^2 \,, \\[0.2em]
%
%4 \, \Theta_3^2 & = &  \theta_1^2 - \theta_2^2 - \theta_3^2 + \theta_4^2 \,, &\qquad
%
% 4\, \Theta_4^2 &= &  \theta_1^2 - \theta_2^2 + \theta_3^2 - \theta_4^2 \,.
%
%\end{array}
%\]
We also have the following identities:
\begin{equation}
\label{Eq:degree2doublingR}
\begin{array}{lllclll}
 \theta_5^2 & = & 2 \, \big( \Theta_1 \Theta_3 + \Theta_2  \Theta_4 \big) \,, &\qquad
 \theta_6^2 & =&  2 \, \big( \Theta_1 \Theta_3 - \Theta_2  \Theta_4 \big)\,, \\[0.4em]
 \theta_7^2 & = & 2 \, \big( \Theta_1 \Theta_4 + \Theta_2  \Theta_3 \big) \,, &\qquad
 \theta_8^2 &= &  2 \, \big( \Theta_1 \Theta_2 + \Theta_3  \Theta_4 \big)\,, \\[0.4em]
 \theta_9^2 & = & 2 \, \big( \Theta_1 \Theta_4 - \Theta_2  \Theta_3 \big)\,, &\qquad
 \theta_{10}^2 &= & 2 \, \big( \Theta_1 \Theta_2 - \Theta_3  \Theta_4 \big) \,.
\end{array}
\end{equation}
Similarly, we have identities for Theta functions with non-vanishing elliptic argument:
\begin{equation}
\label{Eq:degree2doubling_z}
\begin{split}
 \theta_1 \, \theta_1(z) & = \Theta_1(z)^2 + \Theta_2(z)^2 + \Theta_3(z)^2 
 +\Theta_4(z)^2 \,,\\[0.2em]
 \theta_2 \, \theta_2(z) & = \Theta_1(z)^2 + \Theta_2(z)^2 - \Theta_3(z)^2 
 -\Theta_4(z)^2 \,,\\[0.2em]
  \theta_3 \, \theta_3(z) & = \Theta_1(z)^2 - \Theta_2(z)^2 - \Theta_3(z)^2 
 +\Theta_4(z)^2 \,,\\[0.2em]
  \theta_4 \, \theta_4(z) & = \Theta_1(z)^2 - \Theta_2(z)^2 + \Theta_3(z)^2 
 -\Theta_4(z)^2 \,,
\end{split}
\end{equation}
and
\begin{equation}
\label{Eq:Degree2doubling_z}
\begin{split}
 4 \, \Theta_1 \, \Theta_1(2z) & = \theta_1(z)^2 + \theta_2(z)^2 + \theta_3(z)^2 
 +\theta_4(z)^2 \,,\\[0.2em]
 4 \, \Theta_2 \, \Theta_2(2z) & = \theta_1(z)^2 + \theta_2(z)^2 - \theta_3(z)^2 
 -\theta_4(z)^2 \,,\\[0.2em]
 4\, \Theta_3 \, \Theta_3(2z) & = \theta_1(z)^2 - \theta_2(z)^2 - \theta_3(z)^2 
 +\theta_4(z)^2 \,,\\[0.2em]
 4\, \Theta_4 \, \Theta_4(2z) & = \theta_1(z)^2 - \theta_2(z)^2 + \theta_3(z)^2 
 -\theta_4(z)^2 \,.
\end{split}
\end{equation}
The following is a well-known fact:
\begin{remark}
\label{fact:sections1}
For the principally polarized abelian surface $(\mathbf{A}_{\tau},\mathcal{L})$ defined above, a basis of sections for $\mathcal{L}^2$, called Theta functions of level $2$, is given by $\Theta_i(2z)$ or, alternatively, $\theta^2_i(z)$ for $1 \le i \le 4$,  and the point $[\Theta_1:\Theta_2:\Theta_3:\Theta_4] \in \mathbb{P}^3$ is called the Theta null  point of level 2 of $\mathbf{A}_{\tau}$. Similarly, a basis of sections for $\mathcal{L}^4$, called the Theta functions of level $(2,2)$,  is given by $\theta_i(z)$ for $11 \le i \le 16$, and $[\theta_1:\dots :\theta_{10}]  \in \mathbb{P}^9$ is called the Theta null point of level  $(2,2)$ of $\mathbf{A}_{\tau}$.
\end{remark}
\subsection{Mumford identities for Theta functions}
To obtain all quadratic relations between Theta functions, we apply the generalized Riemann identity for Theta functions derived by Mumford in \cite[p.214]{MR2352717}. His master equation~$(R_{CH})$ generating all Theta relations -- which we adjusted to reflect our convention for Theta functions -- is given by
\begin{gather}
 4 \prod_{\epsilon, \epsilon' \in \{ \pm 1\}}
 \theta\!\begin{bmatrix}  \frac{a+\epsilon b+\epsilon' c+\epsilon \epsilon' d}{2} \\[0.2em]  \frac{e+\epsilon f+\epsilon' g+\epsilon \epsilon' h}{2}  \end{bmatrix}\!\!\left(\frac{x+\epsilon y+\epsilon' u+\epsilon \epsilon' v}{2}\right)\\
 \nonumber
 =  \sum_{\alpha, \beta \in  \mathbb{Z}/2\mathbb{Z}} e^{-\frac{\pi i}{2} \beta^t (a+b+c+d)} \,
 \theta\!\begin{bmatrix} a+\alpha  \\[0.2em] e + \beta\end{bmatrix}\!\!(x) \;
 \theta\!\begin{bmatrix} b+\alpha  \\[0.2em] f + \beta\end{bmatrix}\!\!(y) \, 
 \theta\!\begin{bmatrix} c+\alpha  \\[0.2em] g + \beta\end{bmatrix}\!\!(u) \;
 \theta\!\begin{bmatrix} d+\alpha  \\[0.2em] h+ \beta\end{bmatrix}\!\!(v) .
\end{gather}
We first consider all four-term relations between squares of Theta functions, we set $\xi_i = \theta^2_i(z)$ for $1\le i \le 16$. We have the following:
\begin{proposition}
\label{prop:MF}
The ideal of linear equations relating squares of Theta functions are generated by 12 equations: three equations relating odd Theta functions
\begin{equation}
\label{eqn:Mumford_odd}
\begin{array}{rrrrr}
 \theta_6^2 \, \xi_{11}		& \, - \, \theta_4^2 \, \xi_{13}	& \, - \, \theta_9^2 \, \xi_{14}	& \, + \, \theta_{3}^2 \, \xi_{16}	& = 0,\\[0.2em]
 \theta_6^2 \, \xi_{12}	& \, - \, \theta_2^2 \, \xi_{13}	& \, - \, \theta_7^2 \, \xi_{14}	& \, + \, \theta_{1}^2 \, \xi_{16}	& = 0,\\[0.2em]
 \theta_{10}^2 \, \xi_{13}	& \,+ \, \theta_5^2 \, \xi_{14}	& \, - \, \theta_6^2 \, \xi_{15}	& \, - \, \theta_{8}^2 \, \xi_{16}	& = 0,\\[0.2em]
\end{array}
\end{equation}
and nine equations relating even and odd Theta functions
\begin{equation}
\label{eqn:Mumford_mix}
\begin{array}{rrrrr}
 \theta_6^2 \, \xi_1 		& \, - \, \theta_1^2 \, \xi_6 		& \, + \, \theta_7^2 \, \xi_{13}	& \, - \, \theta_2^2 \, \xi_{14}	& = 0,\\[0.2em]
 \theta_6^2 \, \xi_2 		& \, - \, \theta_2^2 \, \xi_6 		& \, - \, \theta_1^2 \, \xi_{14}	& \, + \, \theta_7^2 \, \xi_{16}	& = 0,\\[0.2em]
 \theta_6^2 \, \xi_3 		& \, - \, \theta_3^2 \, \xi_6 		& \, + \, \theta_9^2 \, \xi_{13}	& \, - \, \theta_{4}^2 \, \xi_{14}	& = 0,\\[0.2em]
 \theta_6^2 \, \xi_4 		& \, - \, \theta_4^2 \, \xi_6 		& \, - \, \theta_3^2 \, \xi_{14}	& \, + \, \theta_9^2 \, \xi_{16}	& = 0,\\[0.2em]
 \theta_6^2 \, \xi_5 		& \, - \, \theta_5^2 \, \xi_6 		& \, + \, \theta_8^2 \, \xi_{13}	& \, - \, \theta_{10}^2 \, \xi_{16}	& = 0,\\[0.2em]
 \theta_7^2 \, \xi_6 		& \, - \, \theta_6^2 \, \xi_7 		& \, - \, \theta_1^2 \, \xi_{13}	& \, + \, \theta_2^2 \, \xi_{16}	& = 0,\\[0.2em]
 \theta_8^2 \, \xi_6 		& \, - \, \theta_6^2 \, \xi_8 		& \, - \, \theta_5^2 \, \xi_{13}	& \, + \, \theta_{10}^2 \, \xi_{14}	& = 0,\\[0.2em]
 \theta_9^2 \, \xi_6 		& \, - \, \theta_6^2 \, \xi_9 		& \, - \, \theta_3^2 \, \xi_{13}	& \, + \, \theta_{4}^2 \, \xi_{16}	& = 0,\\[0.2em]
 \theta_{10}^2 \, \xi_6 	& \, - \, \theta_6^2 \, \xi_{10} 	& \, + \, \theta_8^2 \, \xi_{14}	& \, - \, \theta_{5}^2 \, \xi_{16}	& = 0.
\end{array}
\end{equation}
\end{proposition}
\begin{proof}
We follow the strategy outlined in \cite[Sec.~3.2]{MR3238326} where fifteen quadrics of rank four involving only odd Theta functions were derived.  Using \cite[p.214, Eq.~$(R_{CH})$]{MR2352717}, one generates all rank-four quadrics relating squares of Theta function. We obtain 142 equations. One then uses Equations~(\ref{Eq:degree2doubling}) and Equations~(\ref{Eq:degree2doublingR}) to write all coefficients in terms of  $\lbrace \Theta_1 , \Theta_2, \Theta_3, \Theta_4\rbrace$ and determines a generating set of quadrics.
\end{proof}
Next, we consider all three-term relations between bi-monomial combinations of Theta functions. We set $\xi_{i,j} = \theta_i(z)\theta_j(z)$ for $1\le i < j \le 16$. We have the following:
\begin{proposition}
\label{prop:MMF}
The ideal of linear equations relating $\xi_{i,j}$ for $1\le i < j \le 16$ is generated by the following 60 equations:
\begin{equation}
\label{Eqn:MumfordBimonomial_eg}
\begin{split}
 \theta_1 \theta_2 \xi_{1,2} - \theta_3\theta_4 \xi_{3,4} - \theta_8\theta_{10} \xi_{8,10} =0, \;
 \theta_3 \theta_4 \xi_{1,2} - \theta_1\theta_2 \xi_{3,4} - \theta_8\theta_{10} \xi_{13,16} =0, \; \dots\\
  \text{(A complete generating set of 60 equations is given in Equation~(\ref{Eqn:MumfordBimonomial}).)}
\end{split}
\end{equation}
\end{proposition}
\begin{proof}
Using \cite[p.214, Eq.~$(R_{CH})$]{MR2352717}, one generates all three-term bi-monomial combinations of Theta function. 
\end{proof}
\begin{remark}
\label{rem:embedding}
The map $z \mapsto [\theta_1(z): \dots: \theta_{16}(z)]$ given by a all Theta functions provides a high-dimensional embedding of the the abelian variety $\mathbf{A}_{\tau}$ into $\mathbb{P}^{15}$ \cite[Sec.~3]{MR2514037} whose defining equations were determined explicitly in \cite{MR1041476}. The image in $\mathbb{P}^{15}$ is given as the intersection of the 72 conics given by Equations~(\ref{eqn:Mumford_odd}),~(\ref{eqn:Mumford_mix}), and~(\ref{Eqn:MumfordBimonomial}).
\end{remark}
\subsection{Theta functions and genus-two curves}
Let $\mathcal{C}$ be an irreducible, smooth, projective curve of genus two, defined over the complex field $\mathbb{C}$. Let $\mathcal{M}_2$, be the coarse moduli space of smooth curves of genus two. We denote by $[\mathcal{C}]$ the isomorphism class of $\mathcal{C}$, i.e.,  the corresponding point in $\mathcal{M}_2$. For a genus-two curve $\mathcal{C}$ given as sextic $Y^2 = f_6(X,Z)$ in weighted projective space $\mathbb{P}(1,3,1)$, we send three roots $\lambda_4, \lambda_5, \lambda_6$ to $0, 1, \infty$ to get an isomorphic curve in Rosenhain normal form, i.e., 
\begin{equation}
\label{Eq:Rosenhain}
 \mathcal{C}: \quad Y^2 = X\,Z \, \big(X-Z\big) \, \big( X- \lambda_1 Z\big) \,  \big( X- \lambda_2 Z\big) \,  \big( X- \lambda_3 Z\big) \;.
\end{equation} 
The ordered tuple $(\lambda_1, \lambda_2, \lambda_3)$ where the $\lambda_i$ are all distinct and different from $0, 1, \infty$ determines a point in $\mathcal{M}_2(2)$, the moduli space of genus-two curves with level-two structure. 
\par Torelli's theorem states that the map sending a curve $\mathcal{C}$ to its Jacobian  $\mathrm{Jac}(\mathcal{C})$ is injective and defines a birational map $\mathcal{M}_2 \dashrightarrow \mathcal{A}_2$. In fact, if the point $\tau$ is not equivalent with respect to $\Gamma_2$ to a point with $\tau_{12}=0$, the $\mathsf{\Theta}$-divisor is a non-singular curve $\mathcal{C}$ of genus-two and $\mathbf{A}_{\tau}  = \mathrm{Jac}(\mathcal{C})$ is its Jacobian.  
\par Thomae's formula is a formula introduced by Thomae relating Theta constants to the branch points. The three $\lambda$ parameters appearing in the Rosenhain normal form of a genus-two curve $\mathcal{C}$ in Equation~(\ref{Eq:Rosenhain})  are ratios of even Theta constants. There are 720 choices for such expressions since the forgetful map $\mathcal{M}_2(2) \to \mathcal{M}_2$ is a Galois covering of degree $720 = |S_6|$ where $S_6$ acts on the  roots of $\mathcal{C}$ by permutations. Any of the 720 choices may be used, we choose the one also used in \cite{MR2367218,MR2372155}:
\begin{lemma} \label{ThomaeLemma}
For any period point $\tau \in\mathcal{A}_2(2)$, there is a genus-two curve $\mathcal{C} \in \mathcal{M}_2(2)$ with level-two structure and Rosenhain roots $\lambda_1, \lambda_2, \lambda_3$  such that
\begin{equation}\label{Picard}
\lambda_1 = \frac{\theta_1^2\theta_3^2}{\theta_2^2\theta_4^2} \,, \quad \lambda_2 = \frac{\theta_3^2\theta_8^2}{\theta_4^2\theta_{10}^2}\,, \quad \lambda_3 =
\frac{\theta_1^2\theta_8^2}{\theta_2^2\theta_{10}^2}\,.
\end{equation}
Similarly, the following expressions are perfect squares of Theta constants:
\begin{equation}
\label{Picard2}
\begin{split}
\lambda_1 -1 = \frac{\theta_7^2 \theta_9^2}{\theta_2^2 \theta_4^2} , \qquad \lambda_2 -1 = \frac{\theta_5^2 \theta_9^2}{\theta_4^2 \theta_{10}^2} , \qquad 
\lambda_3 -1 = \frac{\theta_5^2 \theta_7^2}{\theta_2^2 \theta_{10}^2} ,\qquad \qquad\\
\lambda_2 -\lambda_1 = \frac{\theta_3^2 \theta_6^2 \theta_9^2}{\theta_2^2 \theta_4^2\theta_{10}^2} , \qquad \lambda_3 -\lambda_1 = \frac{\theta_1^2 \theta_6^2 \theta_7^2}{\theta_2^2 \theta_4^2\theta_{10}^2} , \qquad
 \lambda_3 -\lambda_2 = \frac{\theta_5^2 \theta_6^2 \theta_8^2}{\theta_2^2 \theta_4^2\theta_{10}^2} .
\end{split}
\end{equation} 
Conversely, given a smooth genus-two curve $\mathcal{C} \in \mathcal{M}_2(2)$ with three distinct complex numbers $(\lambda_1, \lambda_2, \lambda_3)$ different from $0, 1, \infty$, there is complex abelian surface $\mathbf{A}_{\tau}$ with period matrix $(\mathbb{I}_2,\tau)$ and $\tau \in\mathcal{A}_2(2)$ such that $\mathbf{A}_{\tau}=\mathrm{Jac} \, \mathcal{C}$  and the fourth powers of the even Theta constants are given by
\begin{equation}\label{Thomaeg=2}
\begin{array}{ll}
\theta_1^4 = R \, \lambda_3 \lambda_1 (\lambda_2 -1) (\lambda_3 - \lambda_1) \,,  &
\theta_2^4  = R \, \lambda_2 (\lambda_2 -1) ( \lambda_3 - \lambda_1) \,,  \\[0.2em]
\theta_3^4  = R\, \lambda_2  \lambda_1 (\lambda_2 - \lambda_1) (\lambda_3 - \lambda_1) \,,  &
\theta_4^4  = R\, \lambda_3 (\lambda_3 - 1) (\lambda_2 - \lambda_1) \,,  \\[0.2em]
\theta_5^4  = R\, \lambda_1 (\lambda_2 -1) (\lambda_3 - 1) ( \lambda_3 - \lambda_2) \,, &
\theta_6^4  = R\, (\lambda_3 - \lambda_2) (\lambda_3 -\lambda_1) ( \lambda_2 -\lambda_1) \,, \\[0.2em]
\theta_7^4  =  R \, \lambda_2 (\lambda_3 -1) ( \lambda_1 -1) (\lambda_3 - \lambda_1)  \,, &
\theta_8^4  = R \, \lambda_2 \lambda_3 (\lambda_3 - \lambda_2) (\lambda_1 -1)  \,, \\[0.2em]
\theta_9^4  = R\, \lambda_3 ( \lambda_2 -1) (\lambda_1 - 1) (\lambda_2 - \lambda_1)  \,, &
\theta_{10}^4  = R \, \lambda_1 ( \lambda_1 - 1) (\lambda_3 - \lambda_2) \,, 
\end{array}
\end{equation}
where $R\in \mathbb{C}^{*}$ is a non-zero constant.
\end{lemma}
The characterization of the Siegel modular threefolds $\mathcal{A}_2(2,4)$, and $\mathcal{A}_2(4,8)$ as projective varieties and their Satake compactifications was given in \cite[Prop.~2.2]{MR3238326}:
\begin{lemma}
\label{compactifications}
\begin{enumerate}
\item[]
\item The holomorphic map $\Xi_{2,4}: \mathbb{H}_2 \to \mathbb{P}^3$ given by $\tau \mapsto [\Theta_1 : \Theta_2: \Theta_3 : \Theta_4]$ induces an isomorphism
between the Satake compactification $\overline{\mathcal{A}_2(2,4)}$ and $\mathbb{P}^3$.
\item The holomorphic map $\Xi_{4,8}: \mathbb{H}_2 \to \mathbb{P}^9$ given by $\tau \mapsto [\theta_1 : \dots : \theta_{10}]$ induces an isomorphism
between the Satake compactification $\overline{\mathcal{A}_2(4,8)}$ and the closure of $\Xi_{4,8}$ in $\mathbb{P}^9$.
\item We have the following commutative diagram:

\centerline{
\xymatrix{
\overline{\mathcal{A}_2(4,8)} \ar[rr]^{\Xi_{4,8}} \ar[d]_{\pi} 							&  									& \mathbb{P}^9 \ar[d]^{\operatorname{Sq}} \\
\overline{\mathcal{A}_2(2,4)} \ar[r]^{\phantom{aa} \Xi_{2,4}}  \ar[d]_{\operatorname{Ros}} 	&  \mathbb{P}^3 \ar[r]^{\operatorname{Ver}} 	& \mathbb{P}^9\\
\overline{\mathcal{A}_2(2)}
}}
Here, $\pi$ is the covering map with deck transformations $\Gamma_2(2,4)/\Gamma_2(4,8)\cong (\mathbb{Z}/2\mathbb{Z})^9$, the map $\operatorname{Sq}$ is the  square map $[\theta_1: \cdots : \theta_{10}] \mapsto [\theta_1^2: \cdots : \theta_{10}^2]$, the map $\operatorname{Ver}$ is the Veronese type map defined by the quadratic relations~(\ref{Eq:degree2doubling}) and~(\ref{Eq:degree2doublingR}), and the map $\operatorname{Ros}$ is the covering map with the deck transformations $\Gamma_2(2)/\Gamma_2(2,4)\cong (\mathbb{Z}/2\mathbb{Z})^3$ given by plugging the quadratic relations~(\ref{Eq:degree2doubling}) and~(\ref{Eq:degree2doublingR}) into Equations~(\ref{Picard}).
\end{enumerate}
\end{lemma}
\begin{remark}
\label{fact:sections2}
Lemma~\ref{ThomaeLemma} shows that $\Gamma_2(2)$ is the group of isomorphisms which fix the 4th power of the Theta constants $\theta_i$  for $1\le i \le 10$,  $\Gamma_2(2,4)$ fixes their 2nd power, and $\Gamma_2(4,8)$ the Theta constants of level $(2,2)$ themselves.
\end{remark}
 \section{Jacobians and two-isogenies}
\subsection{$16_6$ configuration on the Jacobian}
\label{sec:16_6}
 On the Jacobian $\mathbf{A}=\operatorname{Jac}(\mathcal{C})$ the sixteen Theta divisors together with the sixteen order-two points form a $16_6$ configuration in the following way; see \cite{MR2457735}. For the genus-two curve  $\mathcal{C}$ in Equation~(\ref{Eq:Rosenhain}), we denote the Weierstrass points by $\mathfrak{p}_i: [X:Y:Z]=[\lambda_i:0:1]$ for $1 \le i \le 3$, $\mathfrak{p}_4: [X:Y:Z]=[0:0:1]$, $\mathfrak{p}_5: [X:Y:Z]=[1:0:1]$, and $\mathfrak{p}_6:  [X:Y:Z]=[1:0:0]$; we also denote the space of two torsion points on an abelian variety $\mathbf{A}$ by $\mathbf{A}[2]$. In the case of the Jacobian of a genus-two curve, every nontrivial two-torsion point can be expressed using differences of Weierstrass points of $\mathcal{C}$. Concretely, the sixteen order-two points of $\mathbf{A}[2]$ are obtained using the embedding of the curve into the connected component of the identity in the Picard group, i.e., $\mathcal{C} \hookrightarrow \mathbf{A} \cong \operatorname{Pic}^0(\mathcal{C})$ with $\mathfrak{p} \mapsto [\mathfrak{p} -\mathfrak{p}_6]$. In particular, we obtain all elements of $\mathbf{A}[2]$ as
 \begin{equation}
 \label{oder2points}
  \mathfrak{p}_{i6} = [ \mathfrak{p}_i - \mathfrak{p}_6] \; \text{for $1 \le i < 6$}, \qquad 
  \mathfrak{p}_{ij}=[ \mathfrak{p}_i + \mathfrak{p}_j - 2 \, \mathfrak{p}_6]  \; \text{for $1 \le i < j \le 5$}, 
 \end{equation}
where we set $\mathfrak{p}_{0} = \mathfrak{p}_{66} =[0]$. For $\{i,j,k,l,m,n\}=\{1, \dots6\}$,  the group law on $\mathbf{A}[2]$ is given by the relations
 \begin{equation}
 \label{group_law}
    \mathfrak{p}_0 +  \mathfrak{p}_{ij} =  \mathfrak{p}_{ij}, \quad  \mathfrak{p}_{ij} +  \mathfrak{p}_{ij} =  \mathfrak{p}_{0}, \quad 
    \mathfrak{p}_{ij} +  \mathfrak{p}_{kl} =  \mathfrak{p}_{mn}, \quad  \mathfrak{p}_{ij} +
      \mathfrak{p}_{jk} =  \mathfrak{p}_{ik}.
 \end{equation}
\par The  space $\mathbf{A}[2]$ of two torsion points on an abelian variety $\mathbf{A}$ admits a symplectic bilinear form, called the \emph{Weil pairing}. The Weil pairing is induced by the pairing
\[
 \langle [ \mathfrak{p}_i - \mathfrak{p}_j  ] ,[ \mathfrak{p}_k - \mathfrak{p}_l] \rangle =\#\{  \mathfrak{p}_{i}, \mathfrak{p}_{j}\}\cap \{ \mathfrak{p}_{k}, \mathfrak{p}_{l}\} \mod{2},
\]
such that the following table gives all possible pairings:
\begin{equation}
\begin{array}{c|ccccccc}
\langle \bullet, \bullet \rangle 	& \mathfrak{p}_{0}	& \mathfrak{p}_{i6}	& \mathfrak{p}_{j6}	&  \mathfrak{p}_{ij}	& \mathfrak{p}_{il}	& \mathfrak{p}_{kl} \\
\hline
\mathfrak{p}_{0}			& 0				& 0				& 0				& 0				& 0				& 0\\
\mathfrak{p}_{i6}			& 0				& 0				& 1				& 1				& 1				& 0\\
\mathfrak{p}_{j6}			& 0				& 1				& 0				& 1				& 0				& 0\\
\mathfrak{p}_{ij}			& 0				& 1				& 1				& 0				& 1				& 0\\
\mathfrak{p}_{il}			& 0				& 1				& 0				& 1				& 0				& 1\\
\mathfrak{p}_{kl}			& 0				& 0				& 0				& 0				& 1				& 0
\end{array}
\end{equation}
\par We call a two-dimensional, maximal isotropic subspace of $\mathbf{A}[2]$ with respect to the Weil pairing, i.e., a subspace such that the symplectic form vanishes on it, a \emph{G\"opel group} in $\mathbf{A}[2]$. Such a maximal subgroup is isomorphic to $( \mathbb{Z}/2\mathbb{Z})^2$. We have the following:
 \begin{lemma}
 \label{lem:Goepel}
 There are 15 distinct G\"opel groups in $\mathbf{A}[2]$ given by
 \begin{small}
 \begin{equation*}
 \begin{split}
 \{ \mathfrak{p}_0,  \mathfrak{p}_{12}, \mathfrak{p}_{34}, \mathfrak{p}_{56} \}, \,  \{ \mathfrak{p}_0,  \mathfrak{p}_{12}, \mathfrak{p}_{35}, \mathfrak{p}_{46} \}, \, 
 \{ \mathfrak{p}_0,  \mathfrak{p}_{12}, \mathfrak{p}_{36}, \mathfrak{p}_{45} \}, \,  \{ \mathfrak{p}_0,  \mathfrak{p}_{13}, \mathfrak{p}_{24}, \mathfrak{p}_{56} \}, \,
 \{ \mathfrak{p}_0,  \mathfrak{p}_{13}, \mathfrak{p}_{25}, \mathfrak{p}_{46} \}, \\  \{ \mathfrak{p}_0,  \mathfrak{p}_{13}, \mathfrak{p}_{26}, \mathfrak{p}_{45} \}, \,
 \{ \mathfrak{p}_0,  \mathfrak{p}_{14}, \mathfrak{p}_{23}, \mathfrak{p}_{56} \}, \,  \{ \mathfrak{p}_0,  \mathfrak{p}_{14}, \mathfrak{p}_{25}, \mathfrak{p}_{36} \}, \,
 \{ \mathfrak{p}_0,  \mathfrak{p}_{14}, \mathfrak{p}_{26}, \mathfrak{p}_{35} \}, \,  \{ \mathfrak{p}_0,  \mathfrak{p}_{15}, \mathfrak{p}_{23}, \mathfrak{p}_{46} \}, \\
 \{ \mathfrak{p}_0,  \mathfrak{p}_{15}, \mathfrak{p}_{24}, \mathfrak{p}_{36} \}, \,  \{ \mathfrak{p}_0,  \mathfrak{p}_{15}, \mathfrak{p}_{26}, \mathfrak{p}_{34} \}, \,
 \{ \mathfrak{p}_0,  \mathfrak{p}_{16}, \mathfrak{p}_{23}, \mathfrak{p}_{45} \}, \,  \{ \mathfrak{p}_0,  \mathfrak{p}_{16}, \mathfrak{p}_{24}, \mathfrak{p}_{35} \}, \,
 \{ \mathfrak{p}_0,  \mathfrak{p}_{16}, \mathfrak{p}_{25}, \mathfrak{p}_{34} \}.
 \end{split}
 \end{equation*}
 \end{small}
 Moreover, there are 60 distinct affine spaces in $\mathbf{A}[2]$ obtained from the four translates of each G\"opel group.
 \end{lemma}
 \qed
\par We define a \emph{Rosenhain group} to be a subgroup in $\mathbf{A}[2]$ isomorphic to $( \mathbb{Z}/2\mathbb{Z})^2$ of the from $ \{ \mathfrak{p}_0,  \mathfrak{p}_{ij}, \mathfrak{p}_{ik}, \mathfrak{p}_{jk} \}$ with $1\le i < j, k \le 6$. Note that a Rosenhain group is not an isotropic subspace of $\mathbf{A}[2]$. We have the following:
 \begin{lemma}
 \label{lem:Rosenhain}
 There are 20 different Rosenhain groups in $\mathbf{A}[2]$ given by
 \begin{small}
 \begin{equation*}
 \begin{split}
 \{ \mathfrak{p}_0,  \mathfrak{p}_{12}, \mathfrak{p}_{13}, \mathfrak{p}_{23} \}, \,  \{ \mathfrak{p}_0,  \mathfrak{p}_{12}, \mathfrak{p}_{14}, \mathfrak{p}_{24} \}, \, 
 \{ \mathfrak{p}_0,  \mathfrak{p}_{12}, \mathfrak{p}_{15}, \mathfrak{p}_{25} \}, \,  \{ \mathfrak{p}_0,  \mathfrak{p}_{12}, \mathfrak{p}_{16}, \mathfrak{p}_{26} \}, \,
 \{ \mathfrak{p}_0,  \mathfrak{p}_{13}, \mathfrak{p}_{14}, \mathfrak{p}_{34} \}, \\  \{ \mathfrak{p}_0,  \mathfrak{p}_{13}, \mathfrak{p}_{15}, \mathfrak{p}_{35} \}, \,
 \{ \mathfrak{p}_0,  \mathfrak{p}_{13}, \mathfrak{p}_{16}, \mathfrak{p}_{36} \}, \,  \{ \mathfrak{p}_0,  \mathfrak{p}_{14}, \mathfrak{p}_{15}, \mathfrak{p}_{45} \}, \,
 \{ \mathfrak{p}_0,  \mathfrak{p}_{14}, \mathfrak{p}_{16}, \mathfrak{p}_{46} \}, \,  \{ \mathfrak{p}_0,  \mathfrak{p}_{15}, \mathfrak{p}_{16}, \mathfrak{p}_{56} \}, \\
 \{ \mathfrak{p}_0,  \mathfrak{p}_{23}, \mathfrak{p}_{24}, \mathfrak{p}_{34} \}, \,  \{ \mathfrak{p}_0,  \mathfrak{p}_{23}, \mathfrak{p}_{25}, \mathfrak{p}_{35} \}, \,
 \{ \mathfrak{p}_0,  \mathfrak{p}_{23}, \mathfrak{p}_{26}, \mathfrak{p}_{36} \}, \,  \{ \mathfrak{p}_0,  \mathfrak{p}_{24}, \mathfrak{p}_{25}, \mathfrak{p}_{45} \}, \,
 \{ \mathfrak{p}_0,  \mathfrak{p}_{24}, \mathfrak{p}_{26}, \mathfrak{p}_{46} \}, \\  \{ \mathfrak{p}_0,  \mathfrak{p}_{25}, \mathfrak{p}_{26}, \mathfrak{p}_{56} \}, \,
 \{ \mathfrak{p}_0,  \mathfrak{p}_{34}, \mathfrak{p}_{35}, \mathfrak{p}_{45} \}, \,  \{ \mathfrak{p}_0,  \mathfrak{p}_{34}, \mathfrak{p}_{36}, \mathfrak{p}_{46} \}, \,
 \{ \mathfrak{p}_0,  \mathfrak{p}_{35}, \mathfrak{p}_{36}, \mathfrak{p}_{56} \}, \,  \{ \mathfrak{p}_0,  \mathfrak{p}_{45}, \mathfrak{p}_{46}, \mathfrak{p}_{56} \}.
 \end{split}
 \end{equation*}
 \end{small}
 Moreover, there are 80 distinct affine spaces in $\mathbf{A}[2]$ comprised of four translates of each Rosenhain group.
 \end{lemma}
 \qed
 \par  In general, for a principally polarized abelian variety $\mathbf{A}$ the line bundle $\mathcal{L}$ defining its  principal polarization is ample and satisfies $h^0(\mathcal{L}) = 1$. Then, there exists an effective divisor $\mathsf{\Theta}$ such that $\mathcal{L}= \mathcal{O}_{\mathbf{A}}(\mathsf{\Theta})$, uniquely defined only up to a translation. The divisor $\mathsf{\Theta}$ is called a \emph{Theta divisor} associated with the polarization.  It is known that the abelian surface $\mathbf{A}$ is not the product of two elliptic curves if and only if $\mathsf{\Theta}$ is an irreducible divisor. In this case, $\mathsf{\Theta}$ is a smooth curve of genus two and $\mathbf{A}=\operatorname{Jac}(\mathcal{C})$. The standard Theta divisor $\mathsf{\Theta} = \mathsf{\Theta}_6 = \{ [\mathfrak{p} - \mathfrak{p}_0] \, \mid \mathfrak{p} \in \mathcal{C}\}$  contains the six order-two points $\mathfrak{p}_0, \mathfrak{p}_{i6}$ for $1 \le i \le 5$.  Likewise for $1 \le i \le 5$, the five translates $\mathsf{\Theta}_i = \mathfrak{p}_{i6} + \mathsf{\Theta} $ contain $\mathfrak{p}_0, \mathfrak{p}_{i6}, \mathfrak{p}_{ij}$ with $j \not = i, 6$, and the ten translates $\mathsf{\Theta}_{ij6} = \mathfrak{p}_{ij} + \mathsf{\Theta}$ for $1 \le i < j \le 5$ contain $\mathfrak{p}_{ij}, \mathfrak{p}_{i6}, \mathfrak{p}_{j6}, \mathfrak{p}_{kl}$ with $k,l \not = i,j,6$ and $k<l$. Conversely, each order-two point lies on exactly six of the divisors, namely
 \begin{align}
   \mathfrak{p}_{0} & \in \mathsf{\Theta}_i \phantom{ ,  \mathsf{\Theta}_6,  \mathsf{\Theta}_{ij6} } \quad  \text{for $i=1, \dots, 6$,}\\
   \mathfrak{p}_{i6} & \in \mathsf{\Theta}_i,  \mathsf{\Theta}_6,  \mathsf{\Theta}_{ij6} \quad  \text{for $i=1, \dots, 5$ with $j \not = i,6$,}\\
   \mathfrak{p}_{ij} & \in  \mathsf{\Theta}_i,  \mathsf{\Theta}_j,  \mathsf{\Theta}_{kl6} \quad  \text{for $1 \le i < j \le 5$ with $k,l \not = i,j,6$ and $k<l$.}
 \end{align}
We call the divisors $\{ \mathsf{\Theta}_i \}$ and $\{ \mathsf{\Theta}_{jk6} \}$ with $1\le i \le6$ and $1\le j < k<6$, the six \emph{odd} and the ten \emph{even} Theta divisors, respectively. The odd Theta divisors can be identified with the six translates of the curve $\mathcal{C}$ by $\mathfrak{p}_{i6}$ with $1 \le i \le 6$, and thus with the six Weierstrass points $\mathfrak{p}_i$ on the curve $\mathcal{C}$. Furthermore, there is a one-to-one correspondence between pairs of odd symmetric Theta divisors and two-torsion points on $\mathbf{A}[2]$ since $\mathsf{\Theta}_i \cap \mathsf{\Theta}_j =\{ \mathfrak{p}_0, \mathfrak{p}_{ij} \}$, and, in turn, unordered pairs $\{\mathfrak{p}_i, \mathfrak{p}_j\}$ of Weierstrass points since $\mathfrak{p}_{ij} = \mathfrak{p}_{i6}  +  \mathfrak{p}_{j6}$.
 \subsubsection{Relation to finite geometry}
 \label{ss:FiniteGeometry}
 The $16_6$ configuration on $\mathbf{A}[2]$ is effectively described in the context of finite geometry; see~\cite{MR2112585}.  We denote elements of the vector space $\mathbb{F}_2^4$ as matrices $A=\left( \begin{smallmatrix} a_1&a_2\\ a_3& a_4 \end{smallmatrix} \right)$.   It is easy to show that these matrices form a group with $16=2^{2g}$ elements for $g=2$, and a group structure defined by addition modulo two. A symplectic form  on $\mathbb{F}_2^4$ is defined by $(A,A')\mapsto \operatorname{Tr}(A^t \cdot \left( \begin{smallmatrix} 0&1\\ -1&0 \end{smallmatrix} \right)  \cdot A')$.  We say that two matrices $A, A'$ are \emph{syzygetic} if $\operatorname{Tr}(A^t \cdot \left( \begin{smallmatrix} 0&1\\ -1&0 \end{smallmatrix} \right)  \cdot A') \equiv 0 \mod 2$. A \emph{G\"opel group} in $\mathbb{F}_2^4$ is a subgroup of four elements such that every two elements are syzygetic.  It is well known~\cite{MR2367218} that (i) each G\"opel group  in $\mathbb{F}_2^4$ is maximal and isomorphic to $( \mathbb{Z}/2\mathbb{Z})^2$, (ii) the number of different G\"opel groups  in $\mathbb{F}_2^4$ is $15$, (iii) each G\"opel group  in $\mathbb{F}_2^4$ has $2^{2g-2}=4$ cosets which are called G\"opel systems.  Moreover, singular planes in finite geometry are indexed by points in $\mathbb{F}_2^4$ as follows: a plane indexed by $\left( \begin{smallmatrix} b_1&b_2\\ b_3& b_4 \end{smallmatrix} \right) \in\mathbb{F}_2^4$ contains the points $\left( \begin{smallmatrix} a_1&a_2\\ a_3& a_4 \end{smallmatrix} \right)\in\mathbb{F}_2^4$ that satisfy either $\left( \begin{smallmatrix} a_1\\ a_3 \end{smallmatrix} \right)=\left( \begin{smallmatrix} b_1\\ b_3 \end{smallmatrix} \right) \; \text{and} \; \left( \begin{smallmatrix} a_2\\ a_4 \end{smallmatrix} \right)\not =\left( \begin{smallmatrix} b_2\\ b_4 \end{smallmatrix} \right)$ or $\left( \begin{smallmatrix} a_1\\ a_3 \end{smallmatrix} \right)\not =\left( \begin{smallmatrix} b_1\\ b_3 \end{smallmatrix} \right) \; \text{and} \; \left( \begin{smallmatrix} a_2\\ a_4 \end{smallmatrix} \right) =\left( \begin{smallmatrix} b_2\\ b_4 \end{smallmatrix} \right)$. The following was proved in \cite{MR2062673}:
 \begin{lemma}
 \label{lem:16_6}
 The points $\left( \begin{smallmatrix} a_1&a_2\\ a_3& a_4 \end{smallmatrix} \right)\in\mathbb{F}_2^4$ and the singular planes indexed by $\left( \begin{smallmatrix} b_1&b_2\\ b_3& b_4 \end{smallmatrix} \right)\in\mathbb{F}_2^4$ form a $16_6$ configuration on $\mathbb{F}_2^4$: 
 \begin{enumerate}
 \item Any point is contained in exactly six singular planes.
 \item Any singular plane indexed contains exactly six points.
 \end{enumerate}
 The automorphism group of the $16_6$ configuration is $\mathbb{F}_2^4 \rtimes \operatorname{Sp}(4, \mathbb{F}_2)$ and is given by translations by order-two points and rotations preserving the symplectic form. 
 \end{lemma}
 \qed
\par We construct an isomorphism between $\mathbf{A}[2]$ and $\mathbb{F}_2^4$ such that the point $\mathfrak{p}_0$ is mapped to $\left( \begin{smallmatrix} 0&0\\ 0&0 \end{smallmatrix} \right)$. Following Lemma~\ref{lem:16_6}, each Theta divisor $\mathsf{\Theta}_i$ or $\mathsf{\Theta}_{ijk}$, respectively, can also be identified with a singular plane given by points $\mathbb{F}_2^4$. We have the following:
\begin{proposition}
\label{prop:16_6config}
Table~\ref{tab:16_6_configuration} provides an isomorphism between points and planes of the $16_6$ configurations on $\mathbf{A}[2]$ and $\mathbb{F}_2^4$ such that G\"opel groups and their translates in $\mathbf{A}[2]$ -- given in Lemma~\ref{lem:Goepel} --  are mapped bijectively to G\"opel systems in $\mathbb{F}_2^4$.
\begin{table}
\scalebox{0.85}{
\begin{tabular}{c|c||r|c|}
$\mathfrak{p}_{\bullet} \in \mathbf{A}[2]$ & $A\in\mathbb{F}_2^4$ & $\mathsf{\Theta}_{\bullet} \in \operatorname{NS}(\mathbf{A})$  & $\mathfrak{p}_{\bullet}  \in \mathsf{\Theta}_{\bullet}$\\
\hline
$\mathfrak{p}_0$ 	& $\begin{bmatrix}  0 & 0 \\[0.2em] 0 & 0 \end{bmatrix}$  & $\mathsf{\Theta}_{236} = \mathfrak{p}_{23} + \mathsf{\Theta}$ 
				& $\mathfrak{p}_{14}, \mathfrak{p}_{15}, \mathfrak{p}_{23}, \mathfrak{p}_{26}, \mathfrak{p}_{36}, \mathfrak{p}_{45}$\\
$\mathfrak{p}_{45}$ 	& $\begin{bmatrix}  0 & 1 \\[0.2em] 0 & 0 \end{bmatrix}$  & $\mathsf{\Theta}_{1} =  \mathfrak{p}_{16} + \mathsf{\Theta}$ 
				& $\mathfrak{p}_{0\phantom{6}}, \mathfrak{p}_{12}, \mathfrak{p}_{13}, \mathfrak{p}_{14}, \mathfrak{p}_{15}, \mathfrak{p}_{16}$\\
$\mathfrak{p}_{36}$ 	& $\begin{bmatrix}  1 & 0 \\[0.2em] 0 & 0 \end{bmatrix}$  & $\mathsf{\Theta}_{2} =  \mathfrak{p}_{26} + \mathsf{\Theta}$ 
				& $\mathfrak{p}_{0\phantom{6}}, \mathfrak{p}_{12}, \mathfrak{p}_{23}, \mathfrak{p}_{24}, \mathfrak{p}_{25}, \mathfrak{p}_{26}$\\
$\mathfrak{p}_{26}$ 	& $\begin{bmatrix}  1 & 0 \\[0.2em] 1 & 0 \end{bmatrix}$  & $\mathsf{\Theta}_{3} =  \mathfrak{p}_{36} + \mathsf{\Theta}$ 
				& $\mathfrak{p}_{0\phantom{6}}, \mathfrak{p}_{13}, \mathfrak{p}_{23}, \mathfrak{p}_{34}, \mathfrak{p}_{35}, \mathfrak{p}_{36}$\\
$\mathfrak{p}_{15}$ 	& $\begin{bmatrix}  0 & 0 \\[0.2em] 0 & 1 \end{bmatrix}$  & $\mathsf{\Theta}_{4} = \mathfrak{p}_{46} + \mathsf{\Theta}$ 	
				& $\mathfrak{p}_{0\phantom{6}}, \mathfrak{p}_{14}, \mathfrak{p}_{24}, \mathfrak{p}_{34}, \mathfrak{p}_{45}, \mathfrak{p}_{46}$\\
$\mathfrak{p}_{14}$ 	& $\begin{bmatrix}  0 & 1 \\[0.2em] 0 & 1 \end{bmatrix}$  & $\mathsf{\Theta}_{5} = \mathfrak{p}_{56} + \mathsf{\Theta}$ 	
				& $\mathfrak{p}_{0\phantom{6}}, \mathfrak{p}_{15}, \mathfrak{p}_{25}, \mathfrak{p}_{35}, \mathfrak{p}_{45}, \mathfrak{p}_{56}$\\
$\mathfrak{p}_{23}$ 	& $\begin{bmatrix}  0 & 0 \\[0.2em] 1 & 0 \end{bmatrix}$  & $\mathsf{\Theta}_{6} =  \mathfrak{p}_{0\phantom{6}} + \mathsf{\Theta}$ 
				& $\mathfrak{p}_{0\phantom{6}}, \mathfrak{p}_{16}, \mathfrak{p}_{26}, \mathfrak{p}_{36}, \mathfrak{p}_{46}, \mathfrak{p}_{56}$\\
$\mathfrak{p}_{16}$ 	& $\begin{bmatrix}  0 & 1 \\[0.2em] 1 & 0 \end{bmatrix}$  & $\mathsf{\Theta}_{456}= \mathfrak{p}_{45} + \mathsf{\Theta}$ 	
				& $\mathfrak{p}_{12}, \mathfrak{p}_{13}, \mathfrak{p}_{23}, \mathfrak{p}_{45}, \mathfrak{p}_{46}, \mathfrak{p}_{56}$\\
$\mathfrak{p}_{13}$ 	& $\begin{bmatrix}  1 & 1 \\[0.2em] 1 & 0 \end{bmatrix}$  & $\mathsf{\Theta}_{126} = \mathfrak{p}_{12} + \mathsf{\Theta}$ 	
				& $\mathfrak{p}_{12}, \mathfrak{p}_{16}, \mathfrak{p}_{26}, \mathfrak{p}_{34}, \mathfrak{p}_{35}, \mathfrak{p}_{45}$\\
$\mathfrak{p}_{12}$ 	& $\begin{bmatrix}  1 & 1 \\[0.2em] 0 & 0 \end{bmatrix}$  & $\mathsf{\Theta}_{136} = \mathfrak{p}_{13} + \mathsf{\Theta}$ 
				& $\mathfrak{p}_{13}, \mathfrak{p}_{16}, \mathfrak{p}_{24}, \mathfrak{p}_{25}, \mathfrak{p}_{36}, \mathfrak{p}_{45}$\\
$\mathfrak{p}_{24}$ 	& $\begin{bmatrix}  1 & 0 \\[0.2em] 0 & 1 \end{bmatrix}$  & $\mathsf{\Theta}_{346} = \mathfrak{p}_{34} + \mathsf{\Theta}$ 
				& $\mathfrak{p}_{12}, \mathfrak{p}_{15}, \mathfrak{p}_{25}, \mathfrak{p}_{34}, \mathfrak{p}_{36}, \mathfrak{p}_{46}$\\
$\mathfrak{p}_{34}$ 	& $\begin{bmatrix}  1 & 0 \\[0.2em] 1 & 1 \end{bmatrix}$  & $\mathsf{\Theta}_{246} = \mathfrak{p}_{24} + \mathsf{\Theta}$ 
				& $\mathfrak{p}_{13}, \mathfrak{p}_{15}, \mathfrak{p}_{24}, \mathfrak{p}_{26}, \mathfrak{p}_{35}, \mathfrak{p}_{46}$\\
$\mathfrak{p}_{56}$ 	& $\begin{bmatrix}  0 & 1 \\[0.2em] 1 & 1 \end{bmatrix}$  &  $\mathsf{\Theta}_{146} = \mathfrak{p}_{14} + \mathsf{\Theta}$ 
				& $\mathfrak{p}_{14}, \mathfrak{p}_{16}, \mathfrak{p}_{23}, \mathfrak{p}_{25}, \mathfrak{p}_{35}, \mathfrak{p}_{46}$\\
$\mathfrak{p}_{25}$ 	& $\begin{bmatrix}  1 & 1 \\[0.2em] 0 & 1 \end{bmatrix}$  & $\mathsf{\Theta}_{356} = \mathfrak{p}_{35} + \mathsf{\Theta}$ 
				& $\mathfrak{p}_{12}, \mathfrak{p}_{14}, \mathfrak{p}_{24}, \mathfrak{p}_{35}, \mathfrak{p}_{36}, \mathfrak{p}_{56}$\\
$\mathfrak{p}_{35}$ 	& $\begin{bmatrix}  1 & 1 \\[0.2em] 1 & 1 \end{bmatrix}$  & $\mathsf{\Theta}_{256} = \mathfrak{p}_{25} + \mathsf{\Theta}$ 
				& $\mathfrak{p}_{13}, \mathfrak{p}_{14}, \mathfrak{p}_{25}, \mathfrak{p}_{26}, \mathfrak{p}_{34}, \mathfrak{p}_{56}$\\
$\mathfrak{p}_{46}$ 	& $\begin{bmatrix}  0 & 0 \\[0.2em] 1 & 1 \end{bmatrix}$  & $\mathsf{\Theta}_{156} = \mathfrak{p}_{15} + \mathsf{\Theta}$ 
				& $\mathfrak{p}_{15}, \mathfrak{p}_{16}, \mathfrak{p}_{23}, \mathfrak{p}_{24}, \mathfrak{p}_{34}, \mathfrak{p}_{56}$
\end{tabular}}
\bigskip
\caption{Isomorphism between $16_6$ configurations}
\label{tab:16_6_configuration}
\end{table}
\end{proposition}
\begin{proof}
Since $\mathfrak{p}_0$ is mapped to the matrix $\left( \begin{smallmatrix} 0&0\\ 0&0 \end{smallmatrix} \right)$, the divisors $\mathsf{\Theta}_i$ for $i=1, \dots, 6$ must be mapped to the six matrices $\left( \begin{smallmatrix} 0& b_2\\ 0 & b_4 \end{smallmatrix} \right)$ or $\left( \begin{smallmatrix} b_1& 0 \\ b_3 & 0 \end{smallmatrix} \right)$. Making a choice (cf.~Remark~\ref{rem:isos}) for these, we obtain the images of all points $\mathfrak{p}_{ij}$ with $1 \le i < j \le 6$, since we have $\mathsf{\Theta}_i \cap \mathsf{\Theta}_j = \{ \mathfrak{p}_0, \mathfrak{p}_{ij} \}$. Using the properties of the $16_6$ configuration, we obtain the matrices indexing the meaning divisors $\mathsf{\Theta}_{ij6}$. Finally, one checks by an explicit computation that the G\"opel groups and their translates in $\mathbf{A}[2]$ given in Lemma~\ref{lem:Goepel} coincide with the G\"opel systems in $\mathbb{F}_2^4$.
\end{proof}
\begin{remark}
\label{rem:isos}
Following the proof of Proposition~\ref{prop:16_6config} we can say even more. Table~\ref{tab:16_6_configuration} is the unique isomorphism such that the odd Theta divisors are mapped to translates -- namely translates by the fixed element $\left( \begin{smallmatrix} 1&1\\ 1&1 \end{smallmatrix} \right)$ -- of the characteristics of the odd Theta functions introduced in Section~\ref{GeneralRemarks}. We will prove in Lemma~\ref{lem:bijection_tropes_thetas} that this is precisely the property required to make the identification compatible with the Mumford identities for Theta functions while at the same time also mapping odd  Theta divisor to odd Theta functions. All Theta divisors and Theta characteristics are then paired up according to the following table:
\begin{center}
\begin{tabular}{c|c|c||c|c|c}
$\mathsf{\Theta}_{\bullet}$ & 
$\left( \begin{smallmatrix} a_1&a_2\\ a_3&a_4 \end{smallmatrix} \right) + \left( \begin{smallmatrix} 1&1\\ 1&1 \end{smallmatrix} \right)$ & 
$\theta\left[ \begin{smallmatrix} a_1&a_2\\ a_3&a_4 \end{smallmatrix} \right]$ &
$\mathsf{\Theta}_{\bullet}$ & 
$\left( \begin{smallmatrix} a_1&a_2\\ a_3&a_4 \end{smallmatrix} \right) + \left( \begin{smallmatrix} 1&1\\ 1&1 \end{smallmatrix} \right)$ & 
$\theta\left[ \begin{smallmatrix} a_1&a_2\\ a_3&a_4 \end{smallmatrix} \right]$\\
\hline
$\mathsf{\Theta}_{1}$ & $\left( \begin{smallmatrix} 1&0\\ 1&1 \end{smallmatrix} \right) + \left( \begin{smallmatrix} 1&1\\ 1&1 \end{smallmatrix} \right)$ & $\theta_{15}(z)$&
$\mathsf{\Theta}_{2}$ & $\left( \begin{smallmatrix} 0&1\\ 1&1 \end{smallmatrix} \right) + \left( \begin{smallmatrix} 1&1\\ 1&1 \end{smallmatrix} \right)$ & $\theta_{12}(z)$\\
$\mathsf{\Theta}_{3}$ & $\left( \begin{smallmatrix} 0&1\\ 0&1 \end{smallmatrix} \right) + \left( \begin{smallmatrix} 1&1\\ 1&1 \end{smallmatrix} \right)$ & $\theta_{11}(z)$&
$\mathsf{\Theta}_{4}$ & $\left( \begin{smallmatrix} 1&1\\ 1&0 \end{smallmatrix} \right) + \left( \begin{smallmatrix} 1&1\\ 1&1 \end{smallmatrix} \right)$ & $\theta_{16}(z)$\\
$\mathsf{\Theta}_{5}$ & $\left( \begin{smallmatrix} 1&0\\ 1&0 \end{smallmatrix} \right) + \left( \begin{smallmatrix} 1&1\\ 1&1 \end{smallmatrix} \right)$ & $\theta_{14}(z)$&
$\mathsf{\Theta}_{6}$ & $\left( \begin{smallmatrix} 1&1\\ 0&1 \end{smallmatrix} \right) + \left( \begin{smallmatrix} 1&1\\ 1&1 \end{smallmatrix} \right)$ & $\theta_{13}(z)$\\
$\mathsf{\Theta}_{126}$ & $\left( \begin{smallmatrix} 0&0\\ 0&1 \end{smallmatrix} \right) + \left( \begin{smallmatrix} 1&1\\ 1&1 \end{smallmatrix} \right)$ & $\theta_{4}(z)$&
$\mathsf{\Theta}_{136}$ & $\left( \begin{smallmatrix} 0&0\\ 1&1 \end{smallmatrix} \right) + \left( \begin{smallmatrix} 1&1\\ 1&1 \end{smallmatrix} \right)$ & $\theta_{2}(z)$\\
$\mathsf{\Theta}_{146}$ & $\left( \begin{smallmatrix} 1&0\\ 0&0 \end{smallmatrix} \right) + \left( \begin{smallmatrix} 1&1\\ 1&1 \end{smallmatrix} \right)$ & $\theta_{5}(z)$&
$\mathsf{\Theta}_{156}$ & $\left( \begin{smallmatrix} 1&1\\ 0&0 \end{smallmatrix} \right) + \left( \begin{smallmatrix} 1&1\\ 1&1 \end{smallmatrix} \right)$ & $\theta_{8}(z)$\\
$\mathsf{\Theta}_{236}$ & $\left( \begin{smallmatrix} 1&1\\ 1&1 \end{smallmatrix} \right) + \left( \begin{smallmatrix} 1&1\\ 1&1 \end{smallmatrix} \right)$ & $\theta_{10}(z)$&
$\mathsf{\Theta}_{246}$ & $\left( \begin{smallmatrix} 0&1\\ 0&0 \end{smallmatrix} \right) + \left( \begin{smallmatrix} 1&1\\ 1&1 \end{smallmatrix} \right)$ & $\theta_{7}(z)$\\
$\mathsf{\Theta}_{256}$ & $\left( \begin{smallmatrix} 0&0\\ 0&0 \end{smallmatrix} \right) + \left( \begin{smallmatrix} 1&1\\ 1&1 \end{smallmatrix} \right)$ & $\theta_{1}(z)$&
$\mathsf{\Theta}_{346}$ & $\left( \begin{smallmatrix} 0&1\\ 1&0 \end{smallmatrix} \right) + \left( \begin{smallmatrix} 1&1\\ 1&1 \end{smallmatrix} \right)$ & $\theta_{9}(z)$\\
$\mathsf{\Theta}_{356}$ & $\left( \begin{smallmatrix} 0&0\\ 1&0 \end{smallmatrix} \right) + \left( \begin{smallmatrix} 1&1\\ 1&1 \end{smallmatrix} \right)$ & $\theta_{3}(z)$&
$\mathsf{\Theta}_{456}$ & $\left( \begin{smallmatrix} 1&0\\ 0&1 \end{smallmatrix} \right) + \left( \begin{smallmatrix} 1&1\\ 1&1 \end{smallmatrix} \right)$ & $\theta_{6}(z)$
\end{tabular}
\end{center}
\end{remark}
\subsection{$(2,2)$-isogenies of abelian surfaces}
\label{ssec:2isog}
The translation of the Jacobian $\mathbf{A}=\operatorname{Jac}(\mathcal{C})$ by a two-torsion point is an isomorphism of the Jacobian and maps the set of two-torsion points to itself. For any isotropic two-dimensional subspace $\mathsf{K}$ of $\mathbf{A}[2]$, i.e., a G\"opel group in $\mathbf{A}[2]$, it is well-known that $\hat{\mathbf{A}}=\mathbf{A}/\mathsf{K}$ is again a principally polarized abelian surface~\cite[Sec.~23]{MR2514037}. Therefore,  the isogeny $\psi: \mathbf{A} \to \hat{\mathbf{A}}$ between principally polarized abelian surfaces has as its kernel the two-dimensional isotropic subspace $\mathsf{K}$ of $\mathbf{A}[2]$. We call such an isogeny $\psi$ a \emph{$(2,2)$-isogeny}. Concretely, given any choice of $\mathsf{K}$ the $(2,2)$-isogeny is analytically given by the map
\begin{equation}
\begin{split}
\psi: \; \mathbf{A} = \mathbb{C}^2 / \langle \mathbb{Z}^2 \oplus \tau\mathbb{Z}^2\rangle & \to  \hat{\mathbf{A}} =\mathbb{C}^2 / \langle \mathbb{Z}^2 \oplus 2\tau \mathbb{Z}^2\rangle\\
(z,\tau) \mapsto (z,2\tau).
\end{split}
\end{equation}
We have the following lemma:
\begin{lemma}
\label{lem:decomp}
Let $\mathsf{K}$ and $\mathsf{K}'$ be two maximal isotropic subgroup of $\mathbf{A}[2]$ such that  $\mathsf{K}+\mathsf{K}'=\mathbf{A}[2]$, $\mathsf{K} \cap \mathsf{K}'=\{ \mathfrak{p}_0\}$.  Set $\hat{\mathbf{A}}=\mathbf{A}/\mathsf{K}$, and denote the image of $\mathsf{K}'$ in $\hat{\mathbf{A}}$ by $\hat{\mathsf{K}}$. Then it follows that $\hat{\mathbf{A}} /\hat{\mathsf{K}} \cong \mathbf{A}$,  and the composition of $(2,2)$-isogenies $\hat{\psi} \circ\psi$ is multiplication by two on $\mathbf{A}$, i.e., $(z, \tau) \mapsto (2 z, \tau)$. 
\end{lemma}
\begin{proof}
 By construction $\mathsf{K}$ is a finite subgroup of $\mathbf{A}$, $\hat{\mathbf{A}}=\mathbf{A}/\mathsf{K}$ a complex torus,  and the natural projection $\psi: \mathbf{A} \to \hat{\mathbf{A}}\cong \mathbf{A}/\mathsf{K}$ and isogeny. The order of the kernel is two, hence it is a degree-four  isogeny.  The same applies to the map $\hat{\psi}: \hat{\mathbf{A}} \to \hat{\mathbf{A}}/\hat{\mathsf{K}}$. Therefore, the composition $\hat{\psi}\circ \psi$ is an isogeny with kernel $\mathsf{K}+\mathsf{K}'=\mathbf{A}[2]$.Thus, $\hat{\mathbf{A}}/\hat{\mathsf{K}}\cong \mathbf{A}$ and the map $\hat{\psi}\circ \psi$ is the group homomorphism $z\mapsto 2z$ whose kernel are the two-torsion points.
 \end{proof} 
In the case $\mathbf{A}=\operatorname{Jac}(\mathcal{C})$ of the Jacobian of a smooth genus-two curve, one may ask whether $\hat{\mathbf{A}} =\operatorname{Jac}(\hat{\mathcal{C}})$ for some other curve $\hat{\mathcal{C}}$ of genus two, and what the precise relationship between the moduli of $\mathcal{C}$ and $\hat{\mathcal{C}}$ is. The geometric moduli relationship between the two curves of genus two was found by Richelot \cite{MR1578135}; see \cite{MR970659}.
\par Because of the isomorphism $S_6 \cong \operatorname{Sp}(4,\mathbb{F}_2)$, the $(2,2)$-isogenies induce an action of the permutation group of the set of six Theta divisors containing a fixed two-torsion point.  There is a classical way, called \emph{Richelot isogeny}, to describe the 15 inequivalent $(2,2)$-isogenies on the Jacobian $\operatorname{Jac}(\mathcal{C})$ of a generic curve $\mathcal{C}$ of genus-two. If we choose for $\mathcal{C}$ a sextic equation $Y^2 = f_6(X,Z)$, then any factorization $f_6 = A\cdot B\cdot C$ into three degree-two polynomials $A, B, C$ defines a genus-two  curve $\hat{\mathcal{C}}$ given by
\begin{equation}
\label{Richelot}
 \Delta_{ABC} \cdot Y^2 = [A,B] \, [A,C] \, [B,C] 
\end{equation}
where we have set $[A,B] = A'B - AB'$ with $A'$ the derivative of $A$ with respect to $X$ and $\Delta_{ABC}$ is the determinant of $A, B, C$ with respect to the basis $X^2, XZ, Z^2$. It was proved in~\cite{MR970659} that $\operatorname{Jac}(\mathcal{C})$ and $\operatorname{Jac}(\hat{\mathcal{C}})$ are $(2,2)$-isogenous, and that there are exactly  $15$ different curves $\hat{\mathcal{C}}$ that are obtained this way. It follows that this construction yields all principally polarized abelian surfaces $(2,2)$-isogenous to $\mathbf{A} = \operatorname{Jac}(\mathcal{C})$.
\subsubsection{An explicit model using Theta functions}
\label{computation}
We provide an explicit model for a pair of $(2,2)$-isogenies in terms of Theta functions: For the G\"opel groups $\mathsf{K}=\{ \mathfrak{p}_0,  \mathfrak{p}_{15}, \mathfrak{p}_{23}, \mathfrak{p}_{46} \}$ and $\mathsf{K}'=\{ \mathfrak{p}_0,  \mathfrak{p}_{12}, \mathfrak{p}_{34}, \mathfrak{p}_{56} \}$, we have $\mathsf{K}+\mathsf{K}'=\mathbf{A}[2]$, $\mathsf{K} \cap \mathsf{K}'=\{ \mathfrak{p}_0\}$. We set $\hat{\mathbf{A}}=\mathbf{A}/\mathsf{K}$. We will use Theta functions to determine explicit formulas relating the Rosenhain roots of $\mathcal{C}$ in Equation~(\ref{Eq:Rosenhain}) -- given by $\lambda_1, \lambda_2, \lambda_3$ and $\lambda_4=0$, $\lambda_5=1$, and $\lambda_6=\infty$ -- to the roots of a curve $\hat{\mathcal{C}}$ given by the sextic curve 
\begin{equation}
\label{Eq:Rosenhain2}
 \hat{\mathcal{C}}: \quad y^2 = x\,z\,\big(x-z\big) \, \big( x- \Lambda_1z\big) \,  \big( x- \Lambda_2 z\big) \,  \big( x- \Lambda_3 z\big) \;.
\end{equation} 
Since the Theta functions $\Theta_i$ play a role dual to $\theta_i$ for the isogenous abelian variety, the Rosenhain roots of $\hat{\mathcal{C}}$ are given by
\begin{equation}\label{Picard_sq}
\Lambda_1 = \frac{\Theta_1^2\Theta_3^2}{\Theta_2^2\Theta_4^2} \,, \quad \Lambda_2 = \frac{\Theta_3^2\Theta_8^2}{\Theta_4^2\Theta_{10}^2}\,, \quad \Lambda_3 =
\frac{\Theta_1^2\Theta_8^2}{\Theta_2^2\Theta_{10}^2}\,.
\end{equation}
Rosenhain roots can be expressed in terms of just four Theta constants whose characteristics form a G\"opel group in $\mathbb{F}_2^4$. We will write the roots $\lambda_1, \lambda_2, \lambda_3$ and $\Lambda_1, \Lambda_2, \Lambda_3$ in terms of  the Theta constants $\lbrace \Theta_1^2 , \Theta_2^2, \Theta_3^2, \Theta_4^2\rbrace$ and $\lbrace \theta_1^2 , \theta_2^2, \theta_3^2, \theta_4^2\rbrace$, respectively. Their Theta characteristics are form G\"opel groups, namely
\begin{equation*} 
\left( \begin{smallmatrix} 0&0\\ 0&0 \end{smallmatrix} \right), \left( \begin{smallmatrix} 1&0\\ 0&0 \end{smallmatrix} \right), 
\left( \begin{smallmatrix} 1&1\\ 0&0 \end{smallmatrix} \right), \left( \begin{smallmatrix} 0&1\\ 0&0 \end{smallmatrix} \right), \; \text{and} \;
\left( \begin{smallmatrix} 0&0\\ 0&0 \end{smallmatrix} \right), \left( \begin{smallmatrix} 0&0\\ 1&1 \end{smallmatrix} \right), 
\left( \begin{smallmatrix} 0&0\\ 1&0 \end{smallmatrix} \right), \left( \begin{smallmatrix} 0&0\\ 0&1 \end{smallmatrix} \right).
\end{equation*}
We have the following:
\begin{lemma}
The Rosenhain roots $\lambda_1, \lambda_2, \lambda_3$ and $\Lambda_1, \Lambda_2, \Lambda_3$ are rational functions of the Theta functions $\lbrace \Theta_1^2 , \Theta_2^2, \Theta_3^2, \Theta_4^2\rbrace$ and $\lbrace \theta_1^2 , \theta_2^2, \theta_3^2, \theta_4^2\rbrace$, respectively.
Over $\mathcal{A}_2(2,4)$ the rational function
\begin{equation}
\label{def_l}
 l = \dfrac{(\Theta_1 \Theta_2 - \Theta_3 \Theta_4) (\Theta_1^2 + \Theta_2^2 + \Theta_3^2 + \Theta_4^2)(\Theta_1^2 - \Theta_2^2 - \Theta_3^2 + \Theta_4^2)}
 {(\Theta_1 \Theta_2 + \Theta_2 \Theta_4) (\Theta_1^2 - \Theta_2^2 + \Theta_3^2 - \Theta_4^2)(\Theta_1^2 + \Theta_2^2 - \Theta_3^2 - \Theta_4^2)},
\end{equation}
satisfies $l^2= \lambda_1 \lambda_2 \lambda_3$. A similar statement applies to $L$ such that $L^2= \Lambda_1 \Lambda_2 \Lambda_3$.
\end{lemma}
\par We claim that there is a Richelot isogeny realizing the $(2,2)$-isogeny $\psi: \mathbf{A}=\operatorname{Jac}(\mathcal{C}) \to \hat{\mathbf{A}}=\operatorname{Jac}(\hat{\mathcal{C}})$,  for the maximal isotropic subgroup $\mathsf{K}$.  We have the following:
\begin{lemma}
\label{lem:R2isog}
Taking the quotient by the G\"opel group $\mathsf{K}=\{ \mathfrak{p}_0,  \mathfrak{p}_{15}, \mathfrak{p}_{23}, \mathfrak{p}_{46} \}$ corresponds to the Richelot isogeny acting on the genus-two curve $\mathcal{C}$ in Equation~(\ref{Eq:Rosenhain}) by pairing the linear factors according $A=(x-\lambda_1)(x-\lambda_5)$, $B=(x-\lambda_2)(x-\lambda_3)$, $C=(x-\lambda_4)(x-\lambda_6)$
in Equation~(\ref{Richelot}).
\end{lemma}
\begin{proof}
We compute the Richelot-isogeny in Equation~(\ref{Richelot}) obtained from pairing the roots according to $(\lambda_1,\lambda_5=1)$, $(\lambda_2,\lambda_3)$, $(\lambda_4=0,\lambda_6=\infty)$. For this new curve we compute its Igusa-invariants which are in fact rational functions of the Theta functions $[\theta_1: \theta_2: \theta_3: \theta_4]$. We then compute the Igusa invariants for the quadratic twist $\hat{\mathcal{C}}^{(\mu)}$  of the curve in Equation~(\ref{Eq:Rosenhain2}) with
\[
\mu = \frac{(\theta_1\theta_2-\theta_3\theta_4)^2(\theta_1^2 + \theta_2^2 - \theta_3^2 - \theta_4^2)(\theta_1^2 - \theta_2^2 + \theta_3^2 - \theta_4^2)}{4 \, \theta_1\theta_2\theta_3\theta_4(\theta_1^2 + \theta_2^2 + \theta_3^2 + \theta_4^2)(\theta_1^2 - \theta_2^2 - \theta_3^2 + \theta_4^2)} \;.
\]
They again are  rational functions of the Theta functions $[\theta_1: \theta_2: \theta_3: \theta_4]$ since the Rosenhain roots of $\hat{\mathcal{C}}$ are determined by the equations
\begin{equation}\label{Picard_sq_b}
\begin{split}
\Lambda_1 & = \frac{(\theta_1^2 + \theta_2^2 + \theta_3^2 + \theta_4^2)(\theta_1^2 - \theta_2^2 - \theta_3^2 + \theta_4^2)}{(\theta_1^2 + \theta_2^2 - \theta_3^2 - \theta_4^2)(\theta_1^2 - \theta_2^2 + \theta_3^2 - \theta_4^2)}, \\
\Lambda_2 & =  \frac{(\theta_1^2 - \theta_2^2 - \theta_3^2 + \theta_4^2)(\theta_1^2\theta_2^2+\theta_3^2\theta_4^2+2 \theta_1\theta_2\theta_3\theta_4)}{(\theta_1^2 - \theta_2^2 + \theta_3^2 - \theta_4^2)(\theta_1^2\theta_2^2-\theta_3^2\theta_4^2)} ,\\
\Lambda_3 & =  \frac{(\theta_1^2 + \theta_2^2 + \theta_3^2 + \theta_4^2)(\theta_1^2\theta_2^2+\theta_3^2\theta_4^2+2 \theta_1\theta_2\theta_3\theta_4)}{(\theta_1^2 + \theta_2^2 - \theta_3^2 - \theta_4^2)(\theta_1^2\theta_2^2-\theta_3^2\theta_4^2)} \;.
\end{split}
\end{equation}
The two sets of Igusa invariants are identical.
\end{proof}
Lemma~\ref{lem:R2isog} proves that the genus-two curve $\hat{\mathcal{C}}$ is isomorphic to the curve obtained by Richelot isogeny using Equation~(\ref{Richelot}) with
\begin{gather*}
[B,C]=x^2- \lambda_1,  [A,C]=x^2- \lambda_2\lambda_3, \\
[A,B]=(1+\lambda_1-\lambda_2-\lambda_3)x^2-2(\lambda_1-\lambda_2\lambda_3)x +\lambda_1\lambda_2+\lambda_1\lambda_3-\lambda_2\lambda_3-\lambda_1\lambda_2\lambda_3.
\end{gather*}
Recall that for the G\"opel group $\mathsf{K}'=\{ \mathfrak{p}_0,  \mathfrak{p}_{12}, \mathfrak{p}_{34}, \mathfrak{p}_{56} \}$ we have $\mathsf{K}+\mathsf{K}'=\mathbf{A}[2]$, $\mathsf{K} \cap \mathsf{K}'=\{ \mathfrak{p}_0\}$. We denote the image of $\mathsf{K}'$ in $\hat{\mathbf{A}}$ by $\hat{\mathsf{K}}$. We have the following:
\begin{lemma}
\label{lem:R2isog_dual}
Taking the quotient by the G\"opel group $\hat{\mathsf{K}}=\{ \hat{\mathfrak{p}}_0,  \hat{\mathfrak{p}}_{12}, \hat{\mathfrak{p}}_{34}, \hat{\mathfrak{p}}_{56} \}$ corresponds to the Richelot isogeny acting on the curve $\Delta_{ABC} Y^2 = \hat{A}\cdot \hat{B}\cdot \hat{C}$ isomorphic to $\hat{\mathcal{C}}$ by pairing linear factors according to $\hat{A}=[B,C]$, $\hat{B}=[A,C]$, $\hat{C}=[A,B]$.
\end{lemma}
\begin{proof} 
The proof is analogous to the proof of Lemma~\ref{lem:R2isog}.
\end{proof}
\par To see the symmetric relation between the moduli of the isogenous curves $\mathcal{C}$ and $\hat{\mathcal{C}}$ directly, we introduce new moduli $\lambda_1' = ( \lambda_1 + \lambda_2 \lambda_3)/l$,  $\lambda_2' = (\lambda_2 + \lambda_1 \lambda_3)/l$,  and  $\lambda_3' = (\lambda_3 + \lambda_1 \lambda_2)/l$, and similarly $\Lambda_1', \Lambda_2', \Lambda_3'$ with $l^2 = \lambda_1 \lambda_2 \lambda_3$ and $L^2=  \Lambda_1 \Lambda_2 \Lambda_3$. One checks by explicit computation the following:
\begin{lemma}
The parameters $\{\lambda_i'\}$ and $\{\Lambda_i'\}$ are rational functions 
in the squares of Theta constants $\lbrace \theta_1^2 , \theta_2^2, \theta_3^2, \theta_4^2\rbrace$ 
and $\lbrace \Theta_1^2 , \Theta_2^2, \Theta_3^2, \Theta_4^2\rbrace$, respectively.
\end{lemma}
\begin{proof}
The proof follows by direct computation.
\end{proof}
Moreover, we have the following relations:
\begin{proposition}
\label{lem:2isog_curve}
The moduli of the genus-two curve $\mathcal{C}$ in Equation~(\ref{Eq:Rosenhain}) and the $(2,2)$-isogenous genus-two curve $\hat{\mathcal{C}}$ in Equation~(\ref{Eq:Rosenhain2})
are related by
\begin{equation}
\label{relations_RosRoots}
 \begin{split}
 \begin{array}{rl}
   \Lambda_1' & = 2 \, \frac{2 \lambda_1' - \lambda_2'-\lambda_3'}{\lambda_2'-\lambda_3'} \,,\\[0.6em]
   \Lambda_2' - \Lambda_1' & = - \frac{4(\lambda_1'-\lambda_2')(\lambda_1'-\lambda_3')}{(\lambda_1'+2)(\lambda_2'-\lambda_3')} \,,\\[0.6em]
   \Lambda_3' - \Lambda_1' & = - \frac{4(\lambda_1'-\lambda_2')(\lambda_1'-\lambda_3')}{(\lambda_1'-2)(\lambda_2'-\lambda_3')} \,,
  \end{array}
  & \qquad
 \begin{array}{rl}
   \lambda_1' & = 2 \, \frac{2 \Lambda_1' - \Lambda_2'-\Lambda_3'}{\Lambda_2'-\Lambda_3'} \,,\\[0.6em]
   \lambda_2' - \lambda_1' & = - \frac{4(\Lambda_1'-\Lambda_2')(\Lambda_1'-\Lambda_3')}{(\Lambda_1'+2)(\Lambda_2'-\Lambda_3')} \,,\\[0.6em]
   \lambda_3' - \lambda_1' & = - \frac{4(\Lambda_1'-\Lambda_2')(\Lambda_1'-\Lambda_3')}{(\Lambda_1'-2)(\Lambda_2'-\Lambda_3')} \,.
  \end{array}   
  \end{split}
\end{equation}
\end{proposition}
\begin{proof}
The proof follows by direct computation.
\end{proof}
In summary, we proved that there is a Richelot isogeny realizing the $(2,2)$-isogeny $\psi: \mathbf{A}=\operatorname{Jac}(\mathcal{C}) \to \hat{\mathbf{A}}=\operatorname{Jac}(\hat{\mathcal{C}})=\mathbf{A}/\mathsf{K}$,  for the maximal isotropic subgroup $\mathsf{K}$. For the isogenous abelian $\hat{\mathbf{A}}$ variety with period matrix $(\mathbb{I}_2,2\tau)$, the moduli are rational functions in the roots of the $(2,2)$-isogenous curve $\hat{\mathcal{C}}$ in Equation~(\ref{Eq:Rosenhain2}). We consider these Rosenhain roots $\Lambda_1, \Lambda_2, \Lambda_3$ the coordinates on the moduli space of $(2,2)$-Isogenous abelian varieties which we denote by $\hat{\mathcal{A}}_2(2)$. Following Lemma~\ref{compactifications},  the holomorphic map $\xi_{2,4}: \mathbb{H}_2 \to \mathbb{P}^3$ given by $\tau \mapsto [\theta_1 : \theta_2: \theta_3 : \theta_4]$ induces an isomorphism between the Satake compactification of $\hat{\mathcal{A}}_2(2,4)$ and $\mathbb{P}^3$.
\section{Principally polarized Kummer surfaces}
In this section we discuss various normal forms for Kummer surfaces with principal polarization and their rich geometry. One can always choose such a Theta divisor to satisfy $(-\mathbb{I})^* \mathsf{\Theta}=\mathsf{\Theta}$, that is, to be a \emph{symmetric Theta divisor}. The abelian surface $\mathbf{A}$ then maps to the complete linear system $|2\mathsf{\Theta}|$, and the rational map $\varphi_{\mathcal{L}^2}: \mathbf{A} \to \mathbb{P}^3$ associated with the line bundle $\mathcal{L}^2$ factors via an embedding through the projection $\mathbf{A} \to \mathbf{A}/\langle -\mathbb{I} \rangle$; see \cite{MR2062673}. In this way, we can identify $\mathbf{A}/\langle -\mathbb{I} \rangle$ with its image in $\mathbb{P}^3$, a singular quartic surface with sixteen ordinary double points, called a \emph{singular Kummer variety}.  
\subsection{The Shioda normal form}
We start with two copies of a smooth genus-two curve $\mathcal{C}$ in Rosenhain normal form in Equation~(\ref{Eq:Rosenhain}).  The ordered tuple $(\lambda_1, \lambda_2, \lambda_3)$ -- where the $\lambda_i$ are pairwise distinct and different from $(\lambda_4,\lambda_5,\lambda_6)=(0, 1, \infty)$ -- determines a point in $\mathcal{M}_2(2)$, the moduli space of genus-two curves with level-two structure.  The symmetric product of the curve $\mathcal{C}$ is given by $\mathcal{C}^{(2)} = (\mathcal{C}\times\mathcal{C})/\langle \sigma_{\mathcal{C}^{(2)} } \rangle$ where $\sigma_{\mathcal{C}^{(2)}}$ interchanges the two copies of $\mathcal{C}$. We denote the hyperelliptic involution on $\mathcal{C}$ by $\imath_\mathcal{C}$. The variety $\mathcal{C}^{(2)}/\langle \imath_\mathcal{C} \times  \imath_\mathcal{C} \rangle$ is given in terms of the variables $z_1=Z^{(1)}Z^{(2)}$, $z_2=X^{(1)}Z^{(2)}+X^{(2)}Z^{(1)}$, $z_3=X^{(1)}X^{(2)}$, and $\tilde{z}_4=Y^{(1)}Y^{(2)}$  with $[z_1:z_2:z_3:\tilde{z}_4] \in \mathbb{P}(1,1,1,3)$ by the equation
\begin{equation}
\label{kummer_middle}
  \tilde{z}_4^2 = z_1 z_3 \big(  z_1  - z_2 +  z_3 \big)  \prod_{i=1}^3 \big( \lambda_i^2 \, z_1  -  \lambda_i \, z_2 +  z_3 \big) \;.
\end{equation}
\begin{definition}
The hypersurface in $\mathbb{P}(1,1,1,3)$ given by Equation~(\ref{kummer_middle}) is called Shioda sextic and was described in \cite{MR2296439}.
\end{definition}
\begin{lemma}
\label{lem:Shioda}
The Shioda sextic in Equation~(\ref{kummer_middle}) is birational to the Kummer surface $\operatorname{Kum}(\operatorname{Jac}\mathcal{C})$ associated with the Jacobian $\operatorname{Jac}(\mathcal{C})$ of a genus-two curve~$\mathcal{C}$ in Rosenhain normal form~(\ref{Eq:Rosenhain}).
\end{lemma}
\begin{proof}
$\operatorname{Kum}(\operatorname{Jac}\mathcal{C})$ is birational to the quotient  of the Jacobian by the involution $-\mathbb{I}$. For a smooth genus-two curve $\mathcal{C}$, we identify the Jacobian $\operatorname{Jac}(\mathcal{C})$ with $\operatorname{Pic}^2(\mathcal{C})$ under the map $x \mapsto x + K_{\mathcal{C}}$. Since $\mathcal{C}$ is a hyperelliptic curve with involution $\imath_{\mathcal{C}}$ a map from the symmetric product $\mathcal{C}^{(2)}$ to $\operatorname{Pic}^2(\mathcal{C})$ given by $(p,q) \mapsto p+q$ is the blow down of the graph of the hyperelliptic involution to the canonical divisor class. Thus, the Jacobian $\operatorname{Jac}(\mathcal{C})$ is birational to the symmetric square $\mathcal{C}^{(2)}$. The involution $-\mathbb{I}$ restricts to the hyperelliptic involution on each factor of $\mathcal{C}$ in $\mathcal{C}^{(2)}$.
\end{proof}
\begin{remark}
\label{rem:6lines}
Equation~(\ref{kummer_middle}) defines a double cover of $\mathbb{P}^2 \ni [z_1:z_2:z_3]$ branched along six lines given by
\begin{equation}
\label{eqn:6lines}
 z_1=0, \  z_3=0,  \   z_1  - z_2 +  z_3=0, \ \lambda_i^2 \, z_1  -  \lambda_i \, z_2 +  z_3 =0 \; \text{with $1\le i \le 3$}.
 \end{equation}
The six lines are tangent to the common conic $z_2^2 - 4 \, z_1 z_3=0$. Conversely, any six lines tangent to a common conic can always be brought into the form of Equations~\ref{eqn:6lines}. A picture is provided in Figure~\ref{fig:6Lines}.
\end{remark}
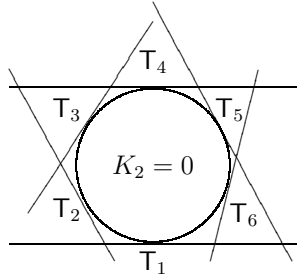
\begin{figure}[ht]
\scalemath{0.85}{
$$
  \begin{xy}
   <0cm,0cm>;<1.5cm,0cm>:
    (2,-.3)*++!D\hbox{$\mathsf{T}_1$},
    (1.1,0.25)*++!D\hbox{$\mathsf{T}_2$},
    (1.1,1.3)*++!D\hbox{$\mathsf{T}_3$},
    (2,1.8)*++!D\hbox{$\mathsf{T}_4$},
    (2.8,1.3)*++!D\hbox{$\mathsf{T}_5$},
    (2.95,0.2)*++!D\hbox{$\mathsf{T}_6$},
    (2,0.7)*++!D\hbox{$K_2=0$},
    (.5,0.18);(3.5,0.18)**@{-},
    (0.5,2);(1.6,0)**@{-},
    (3.1,2);(2.6,0)**@{-},
    (0.7,.5);(2,2.5)**@{-},
    (0.5,1.82);(3.5,1.82)**@{-},
    (2,2.7);(3.45,0)**@{-},
    (2,1)*\xycircle(.8,.8){},
  \end{xy}
  $$}
\caption{Double cover branched along reducible sextic}
\label{fig:6Lines}
\end{figure}
\subsection{The Cassels-Flynn normal form}
We call a surface in complex projective space a \emph{nodal surface} if its only singularities are nodes. For a quartic surface in $\mathbb{P}^3$, it is known that the maximal number of simple nodes is sixteen. 
\par In $\mathbb{P}^3$ we use the projective coordinates $[z_1:z_2:z_3:z_4] \in \mathbb{P}^3$. We consider the morphism $\pi: \mathbb{P}^3 \to \mathbb{P}(1,1,1,3)$ defined by $z_4 \mapsto \tilde{z}_4=(2K_2 \, z_4 +K_1)/4$ with coefficients $K_j=K_j(z_1,z_2,z_3)$ homogeneous of degree $4-j$ with $j=0,1,2$ and given by
\begin{align*}
  K_2 & = z_2^2 - 4 \, z_1 \, z_3 \;,
\end{align*}
\begin{align*}
  K_1 & =\left(  \left( 4\,\lambda_1+4\,\lambda_2+4\,\lambda_3+4 \right) z_1-2\,z_2 \right) z_3^2\\
  & + \Big( \left( 4\,\lambda_1\lambda_2\lambda_3+4\,\lambda_1
\lambda_2+4\,\lambda_1\lambda_3+4\,\lambda_2\lambda_3 \right) z_1^2\\
& + \left( -2\,\lambda_1\lambda_2-2\,
\lambda_1\lambda_3-2\,\lambda_2\lambda_3-2\,\lambda_1-2\,\lambda_2-2\,\lambda_3 \right) z_2z_1 \Big) \, z_3\\
&-2\,\lambda_1\lambda_2\lambda_3z_1^2z_2 \;,
\stepcounter{equation}\tag{\theequation}\label{coeffs_Kummer}
\end{align*}
\begin{align*}
K_0 & =  z_3^4 -2\, \left( \lambda_1\lambda_2+\lambda_1 \lambda_3+\lambda_2\lambda_3+\lambda_1
+\lambda_2+\lambda_3 \right) z_1z_3^3\\
& +\Big( \big(\lambda_1^2 \lambda_2^2 -2 \lambda_1^2 \lambda_2 \lambda_3 + \lambda_1^2 \lambda_3^2 - 2 \lambda_1 \lambda_2^2 \lambda_3-2 \lambda_1 \lambda_2 \lambda_3^2 +\lambda_2^2 \lambda_3^2 
-2 \lambda_1^2 \lambda_2 -2 \lambda_1^2 \lambda_3\\
&-2 \lambda_1 \lambda_2^2-8\lambda_1\lambda_2\lambda_3-2\lambda_1\lambda_3^2-2\lambda_2^2\lambda_3-2\lambda_2\lambda_3^2+\lambda_1^2-2\lambda_1\lambda_2-2\lambda_1\lambda_3+\lambda_2^2\\
&-2\lambda_2\lambda_3+\lambda_3^2\big) \, z_1^2+\left(4\lambda_1 \lambda_2 \lambda_3+4 \lambda_1 \lambda_2+4 \lambda_1 \lambda_3+4 \lambda_2 \lambda_3\Big) \, z_1 z_2\right) \, z_3^2\\
& +\Big(-4 \lambda_1 \lambda_2 \lambda_3 z_1 z_2^2+4 \lambda_1 \lambda_2 \lambda_3 \left(\lambda_1+\lambda_2+\lambda_3+1\right) z_1^2 z_2\\
& -2 \lambda_1 \lambda_2 \lambda_3 \left(\lambda_1 \lambda_2+\lambda_1 \lambda_3+\lambda_2 \lambda_3+\lambda_1+\lambda_2+\lambda_3\right) z_1^3\Big) \, z_3+\lambda_1^2 \lambda_2^2 \lambda_3^2 z_1^4\;.
\end{align*}
\par We have the following result:
\begin{lemma}
\label{lem:CFquartic}
The map $\pi: \mathbb{P}^3 \to \mathbb{P}(1,1,1,3)$ blows down the double cover of the special conic $K_2=0$, and is an isomorphism elsewhere. In particular, the proper transform of $\pi^{-1}\mathcal{C}^{(2)}/\langle \imath_\mathcal{C} \times  \imath_\mathcal{C} \rangle$  is a nodal quartic surface in $\mathbb{P}^3 \ni [z_1:z_2:z_3:z_4]$ with sixteen nodes given by
\begin{equation}
\label{kummer}
 K_2(z_1,z_2,z_3) \; z_4^2 \; + \; K_1(z_1,z_2,z_3)\; z_4 \; + \; K_0(z_1,z_2,z_3) = 0 \;.
\end{equation}
\end{lemma}
\begin{proof}
We observe that
\[
 \frac{1}{16} \big(2\, K_2 \, z_4 +K_1\big)^2 =
 \frac{1}{16} \left(K_1^2 - 4 \, K_0 K_2\right) = z_1  z_3 \,  \prod_{i=1}^4 \big( \lambda_i^2 \, z_1  -  \lambda_i \, z_2 +  z_3 \big) \;,
 \]
and Equation~(\ref{kummer_middle}) is equivalent to
\[
  0= K_2  \left( K_2 z_4^2  + K_1 z_4  +  K_0 \right) \;.
\]
The preimage contains the additional point $[z_1:z_2:z_3:z_4]=[0:0:0:1]$ for which $K_2=K_1=K_0=0$. We then obtain sixteen nodes on the quartic surface given by Equation~(\ref{kummer}). All singular points are listed in Table~\ref{tab:nodes}.
\end{proof}
\begin{definition}
The quartic surface in $\mathbb{P}^3$ given by Equation~(\ref{kummer}) appeared in Cassels and Flynn \cite[Sec.~3]{MR1406090} and is called the Cassels-Flynn quartic.
\end{definition}
\par We have the immediate:
\begin{corollary}
The Cassels-Flynn quartic in Equation~(\ref{kummer}) is isomorphic to the singular Kummer variety $\operatorname{Jac}(\mathcal{C})/\langle -\mathbb{I} \rangle$ associated with the Jacobian $\operatorname{Jac}(\mathcal{C})$ of the genus-two curve~$\mathcal{C}$ in Rosenhain form~(\ref{Eq:Rosenhain}).
\end{corollary}
\qed
\par   The map from the Jacobian $\operatorname{Jac}(\mathcal{C})$ onto its image in $\mathbb{P}^3$ is two-to-one, except on sixteen points of order two where it is injective. This can be seen as follows:  We label the sixteen singular points $p_0, \, p_{ij}$ with $1\le i < j \le 6$ where $p_0$ is located at $[0:0:0:1]$.  Fifteen nodes for Equation~(\ref{kummer_middle}), namely $p_{ij}$ with $1\le i < j \le 6$, are given by
$$
 [z_1: z_2 : z_3 : \tilde{z}_4] = [1: \lambda_i+\lambda_j : \lambda_i \lambda_j: 0] \;,
$$
where we have used $\lambda_4 =0$, $\lambda_5=1$. These points are obtained by combining the Weierstrass points $\mathfrak{p}_{i}$ and $\mathfrak{p}_{j}$  on $\mathcal{C}$ given by
\begin{equation}
\label{pencilC_DP}
[X^{(1)}:Y^{(1)}:Z^{(1)}]=[\lambda_i: 0 : 1] \; \text{and} \; [X^{(2)}:Y^{(2)}:Z^{(2)}]=[\lambda_j: 0 : 1] .
\end{equation}
Similarly, five nodes  $p_{i6}$ for $1\le i \le 5$ are given by
$$
 [z_1: z_2 : z_3 : \tilde{z}_4] = [0: 1 : \lambda_i : 0]
$$
and obtained by combining Weierstrass points $\mathfrak{p}_{i}$ and $\mathfrak{p}_{6}$ on $\mathcal{C}$ given by
\begin{equation}
\label{pencilC_DP_b}
[X^{(1)}:Y^{(1)}:Z^{(1)}]=[\lambda_i: 0 : 1] \; \text{and} \; [X^{(2)}:Y^{(2)}:Z^{(2)}]=[1: 0 : 0].
\end{equation}
Therefore, the singular points $p_{ij}$ are the images of the two-torsion points $\mathfrak{p}_{ij} \in \mathbf{A}[2]$ in Equation~(\ref{oder2points}).   Table~\ref{tab:nodes} lists all points $p_{ij}$ corresponding to points $\mathfrak{p}_{ij} \in \mathbf{A}[2]$ and relate our notation to the notation $n_\bullet$ used by Kumar in~\cite{MR3263663} and $e_\bullet$ used by Mehran in ~\cite{MR2804549}.
\par Conversely, we can start with a singular Kummer surface in $\mathbb{P}^3$ and reconstruct the configuration of six lines; see \cite{MR2964027}.  For a fixed singular point $p_0$, we identify the lines in $\mathbb{P}^3$ through the point $p_0$ with $\mathbb{P}^2$  and map any line in the tangent cone of $p_0$ to itself. Any projective line through $p_0$ meets the quartic surface generically in two other points and  with multiplicity two at the other nodes. In this way we obtain a double cover of $\mathbb{P}^2$ branched along  a plane curve of degree six  where all nodes of the quartic surface different from $p_0$ map to nodes of the sextic. By the genus-degree formula, the maximal number of nodes on a sextic curve is attained when the curve is a union of six lines, in which case we obtain the fifteen remaining nodes apart from $p_0$. Since $p_0$ is a node, the tangent cone to this point is mapped to a conic, and this conic is tangent to the six lines. In summary, the branch locus of the double cover to $\mathbb{P}^2$ is a reducible plane sextic curve, namely the union of six lines tangent to a special conic. By Remark~\ref{rem:6lines} this proves the following:
\begin{corollary}
Every nodal quartic surface with sixteen nodes is isomorphic to the singular Kummer variety of the Jacobian of a genus-two curve.
\end{corollary}
\qed
\begin{table}
\scalebox{0.85}{
\begin{tabular}{c|c|c|l}
$p$ & $n$ & $e$  & $[z_1: z_2 : z_3 : z_4]$\\
\hline
$p_0$ 	& $n_0$ 		& $e_0$ 	& $[0:0:0:1]$\\
$p_{16}$	& $n_3$ 		& $e_{14}$ &  $[0: 1 : \lambda_1 : \lambda_1^2]$\\
$p_{26}$	& $n_4$ 		& $e_{15}$ &  $[0: 1 : \lambda_2 : \lambda_2^2]$\\
$p_{36}$	& $n_5$ 		& $e_{16}$ &  $[0: 1 : \lambda_3 : \lambda_3^2]$\\
$p_{46}$	& $n_1$ 		& $e_{12}$ &  $[0: 1 : 0 : 0]$\\
$p_{56}$	& $n_2$ 		& $e_{13}$ &  $[0: 1 : 1 : 1]$\\
$p_{14}$	& $n_{13}$ 	& $e_{24}$ &  $[1: \lambda_1 : 0 : \lambda_2\lambda_3]$\\
$p_{24}$	& $n_{14}$ 	& $e_{25}$ &  $[1: \lambda_2 : 0 : \lambda_1\lambda_3]$\\
$p_{34}$	& $n_{15}$ 	& $e_{26}$ &  $[1: \lambda_3 : 0 : \lambda_1\lambda_2]$\\
$p_{45}$	& $n_{12}$ 	& $e_{23}$ &  $[1: 1 : 0 :  \lambda_1\lambda_2 \lambda_3]$\\
$p_{15}$	& $n_{23}$ 	& $e_{34}$ &  $[1: \lambda_1+1 : \lambda_1 : \lambda_1(\lambda_2+\lambda_3)]$\\
$p_{25}$	& $n_{24}$ 	& $e_{35}$ &  $[1: \lambda_2+1 : \lambda_2 : \lambda_2(\lambda_1+\lambda_3)]$\\
$p_{35}$	& $n_{25}$ 	& $e_{36}$ &  $[1: \lambda_3+1 : \lambda_3 : \lambda_3(\lambda_1+\lambda_2)]$\\
$p_{13}$	& $n_{35}$ 	& $e_{46}$ &  $[1: \lambda_1+\lambda_3 : \lambda_1\lambda_3 : (\lambda_2+1) \, \lambda_1  \lambda_3]$\\
$p_{23}$	& $n_{45}$ 	& $e_{56}$ &  $[1: \lambda_2+\lambda_3 : \lambda_2\lambda_3 : (\lambda_1+1) \, \lambda_1  \lambda_2]$\\
$p_{12}$	& $n_{34}$ 	& $e_{45}$ &  $[1: \lambda_1+\lambda_2 : \lambda_1\lambda_2 : (\lambda_3+1) \, \lambda_1  \lambda_2]$
\end{tabular}}
\bigskip
\caption{Nodes on a generic Jacobian Kummer surface}
\label{tab:nodes}
\end{table}
\subsection{Intersection model for a Kummer surface}
\label{ss:TropesAsThetas}
We consider the principally polarized abelian surface $\mathbf{A}=\operatorname{Jac}(\mathcal{C})$ with the standard Theta divisor $\mathsf{\Theta} \cong [\mathcal{C}]$. The image of each two-torsion point is a singular point on the Kummer surface, called a \emph{node}. The nodes $p_{ij}$ are the images of the two-torsion points $\mathfrak{p}_{ij} \in \mathbf{A}[2]$ in Equation~(\ref{oder2points}). Any Theta divisor is mapped to the intersection of the Kummer quartic with a plane in $\mathbb{P}^3$.  We call such a singular plane a \emph{trope}. Hence, we have a configuration of sixteen nodes and sixteen tropes in $\mathbb{P}^3$, where each contains six nodes, and such that the intersection of each two is along two nodes. This configuration is called the $16_{6}$ configuration on the Kummer surface.
\par In the complete linear system $|2\mathsf{\Theta}|$ on $\operatorname{Jac}(\mathcal{C})$, the odd symmetric Theta divisors $\mathsf{\Theta}_i$ on $\operatorname{Jac}(\mathcal{C})$ introduced in Section~\ref{sec:16_6} are mapped to six tropes $\mathsf{T}_i$. Equations for the six tropes are now easily found by inspection: for a given integer  $i$ with $1 \le i \le 5$, the nodes $p_0$ and $p_{ij}$ or $p_{ji}$ all lie on the plane
\begin{equation}
 \mathsf{T}_i: \quad \lambda_i^2 z_1 - \lambda_i z_2 + z_3 = 0 \;,
\end{equation}
where we have set $\mathsf{T}_6: z_1=0$. Thus, we obtain
\[
 \tilde{z}_4^2 = \mathsf{T}_1 \mathsf{T}_2 \mathsf{T}_3 \mathsf{T}_4 \mathsf{T}_5 \mathsf{T}_6 \;,
\]
and the Kummer surface $\operatorname{Kum}(\operatorname{Jac} \mathcal{C})$ is the minimal resolution of the double cover of $\mathbb{P}^2$ branched along a reducible plane sextic curve -- the union of six lines all tangent to a conic. In fact, the trope $\mathsf{T}_i$ is tangent to the conic $K_2=0$ at $[1: 2\lambda_i : \lambda_i^2 :0]$ for $i=1,\dots,5$, and $\mathsf{T}_6$ is tangent to $K_2=0$ at $[0:0:1:0]$. 
\par The remaining $10$ tropes $\mathsf{T}_{ijk}$ with $1\le i < j < k \le 6$  correspond to partitions of $\{1,\dots,6\}$ into two unordered sets of three, say $\lbrace i, j, k\rbrace$ $\lbrace l,m, n\rbrace$. We use the formulas for  $\mathsf{T}_{ijk}$ from  \cite[Sec.~3.7]{MR1406090} paying careful attention to the fact that  we have moved the root $\lambda_6$ to infinity. 
\begin{table}
\scalebox{0.85}{
\begin{tabular}{c|c|l|l}
$\mathsf{T}$ & $T$  & $[z_1: z_2 : z_3 : z_4]$ & contained nodes\\
\hline
$\mathsf{T}_1$ 	& $T_3$ 	& $[\lambda_1^2:-\lambda_1:1:0]$ & $p_{0\phantom{6}}, p_{12}, p_{13}, p_{14}, p_{15}, p_{16}$\\
$\mathsf{T}_2$ 	& $T_4$ 	& $[\lambda_2^2:-\lambda_2:1:0]$ & $p_{0\phantom{6}}, p_{12}, p_{23}, p_{24}, p_{25}, p_{26}$\\
$\mathsf{T}_3$ 	& $T_5$ 	& $[\lambda_3^2:-\lambda_3:1:0]$ & $p_{0\phantom{6}}, p_{13}, p_{23}, p_{34}, p_{35}, p_{36}$\\
$\mathsf{T}_4$ 	& $T_1$ 	& $[0:0:1:0]$ & $p_{0\phantom{6}}, p_{14}, p_{24}, p_{34}, p_{45}, p_{46}$\\
$\mathsf{T}_5$ 	& $T_2$ 	& $[1:-1:1:0]$ &$p_{0\phantom{6}}, p_{15}, p_{25}, p_{35}, p_{45}, p_{56}$\\
$\mathsf{T}_6$ 	& $T_0$ 	& $[1:0:0:0]$ & $p_{0\phantom{6}}, p_{16}, p_{26}, p_{36}, p_{46}, p_{56}$\\
$\mathsf{T}_{146}$ 	& $T_{13}$ & $[-\lambda_2\lambda_3:0:-\lambda_1:1]$ & $p_{14}, p_{16}, p_{23}, p_{25}, p_{35}, p_{46}$\\
$\mathsf{T}_{246}$ 	& $T_{14}$ & $[-\lambda_1\lambda_3:0:-\lambda_2:1]$ & $p_{13}, p_{15}, p_{24}, p_{26}, p_{35}, p_{46}$\\
$\mathsf{T}_{346}$ 	& $T_{15}$ & $[-\lambda_1\lambda_2:0:-\lambda_3:1]$ & $p_{12}, p_{15}, p_{25}, p_{34}, p_{36}, p_{46}$\\
$\mathsf{T}_{456}$ 	& $T_{12}$ & $[-\lambda_1\lambda_2\lambda_3:0:-1:1]$&$p_{12}, p_{13}, p_{23}, p_{46}, p_{45}, p_{56}$\\
$\mathsf{T}_{156}$ 	& $T_{23}$ & $[-\lambda_1(\lambda_2+\lambda_3):\lambda_1:-(\lambda_1+1):1]$ & $p_{15}, p_{16}, p_{23}, p_{24}, p_{34}, p_{56}$\\
$\mathsf{T}_{256}$ 	& $T_{24}$ & $[-\lambda_2(\lambda_1+\lambda_3):\lambda_2:-(\lambda_2+1):1]$ & $p_{13}, p_{14}, p_{25}, p_{26}, p_{34}, p_{56}$\\
$\mathsf{T}_{356}$ 	& $T_{25}$ & $[-\lambda_3(\lambda_1+\lambda_2):\lambda_3:-(\lambda_3+1):1]$ & $p_{12}, p_{14}, p_{24}, p_{35}, p_{36}, p_{56}$\\
$\mathsf{T}_{136}$ 	& $T_{35}$ & $[- (\lambda_2+1) \, \lambda_1  \lambda_3: \lambda_1  \lambda_3: - (\lambda_1 + \lambda_3) :1]$ 
	& $p_{13}, p_{16}, p_{24}, p_{25}, p_{36}, p_{45}$\\
$\mathsf{T}_{236}$ 	& $T_{45}$ & $[- (\lambda_1+1) \, \lambda_2  \lambda_3: \lambda_2  \lambda_3: - (\lambda_2 + \lambda_3) :1]$ 
	& $p_{14}, p_{15}, p_{23}, p_{26}, p_{36}, p_{45}$\\
$\mathsf{T}_{126}$ 	& $T_{34}$ & $[- (\lambda_3+1) \, \lambda_1  \lambda_2: \lambda_1  \lambda_2: - (\lambda_1 + \lambda_2) :1]$ 
	& $p_{12}, p_{16}, p_{26}, p_{34}, p_{35}, p_{45}$\\
\end{tabular}}
\bigskip
\caption{Tropes on a generic Jacobian Kummer surface}
\label{tab:tropes}
\end{table}
For example, we have
\begin{equation}
 \begin{split}
  \mathsf{T}_{246}: \quad & - \lambda_1 \lambda_3 \, z_1 - \lambda_2 \, z_3 + z_4 = 0\;,\\
  \mathsf{T}_{346}: \quad & - \lambda_1 \lambda_2 \, z_1 - \lambda_3 \, z_3 + z_4 = 0\;,\\
  \mathsf{T}_{236}: \quad & - (\lambda_1+1) \, \lambda_2  \lambda_3 \, z_1 + \lambda_2  \lambda_3 \, z_2 - (\lambda_2 + \lambda_3) \, z_3 + z_4 = 0\;.
 \end{split}
\end{equation}
We also have the following:
\begin{lemma}
\label{lem:LinearRelation}
The ideal of linear relations between the 16 tropes is generated by 12 equations given by
\begin{equation}
\label{Kummer:topesLRa}
\begin{array}{rrrrr}
 \mathsf{T}_1=		&(1-\lambda_1) \mathsf{T}_4 				&+\lambda_1 \mathsf{T}_5			&+ \lambda_1 (\lambda_1-1) \mathsf{T}_6
 				&,\\[0.5em]
 \mathsf{T}_2 =		&(1-\lambda_2) \mathsf{T}_4 				&+\lambda_2 \mathsf{T}_5 			&+ \lambda_2 (\lambda_2-1) \mathsf{T}_6
 				&,\\[0.5em]
 \mathsf{T}_3 =		&(1-\lambda_3) \mathsf{T}_4 				&+\lambda_3 \mathsf{T}_4 			&+ \lambda_3 (\lambda_3-1) \mathsf{T}_6
 				&,\\[0.5em]
\end{array}
\end{equation}
and 
\begin{equation}
\label{Kummer:topesLRb}
\begin{array}{rrrrr}
 \mathsf{T}_{126} =	&(\lambda_1-1)(\lambda_2-1)\mathsf{T}_4	&-\lambda_1\lambda_2 \mathsf{T}_5		&
 				&+\mathsf{T}_{456},\\[0.5em]
 \mathsf{T}_{136} =	&(\lambda_1-1)(\lambda_3-1)\mathsf{T}_4	&-\lambda_1 \lambda_3 \mathsf{T}_5	&
 				&+\mathsf{T}_{456},\\[0.5em]
 \mathsf{T}_{146} =	&(1-\lambda_1)\mathsf{T}_4				&								&+\lambda_2 \lambda_3 (\lambda_1-1)\mathsf{T}_6
 				&+\mathsf{T}_{456},\\[0.5em]
 \mathsf{T}_{156} =	&									&-\lambda_1 \mathsf{T}_5				&+\lambda_1 (\lambda_2-1) (\lambda_-1)\mathsf{T}_6
 				&+\mathsf{T}_{456},\\[0.5em]
 \mathsf{T}_{236} =	&(1-\lambda_2)(1-\lambda_3)\mathsf{T}_4	&-\lambda_2\lambda_3 \mathsf{T}_5		&
 				&+\mathsf{T}_{456},\\[0.5em]
 \mathsf{T}_{246} =	&(1-\lambda_2)\mathsf{T}_4				&								&+\lambda_1\lambda_3(\lambda_2-1)\mathsf{T}_6
 				&+\mathsf{T}_{456},\\[0.5em]
 \mathsf{T}_{256} =	&									&-\lambda_2 \mathsf{T}_5				&+\lambda_2(\lambda_1-1)(\lambda_3-1)\mathsf{T}_6
 				&+\mathsf{T}_{456},\\[0.5em]
 \mathsf{T}_{346} =	&(1-\lambda_3) \mathsf{T}_4				&								&+\lambda_1\lambda_2(\lambda_3-1)\mathsf{T}_6
 				&+\mathsf{T}_{456},\\[0.5em]
 \mathsf{T}_{356} =	&									&-\lambda_3\mathsf{T}_5				&+\lambda_3(\lambda_1-1)(\lambda_2-1)\mathsf{T}_6
 				&+\mathsf{T}_{456}.
\end{array}
\end{equation}
\end{lemma}
\begin{proof}
Using the explicit form of the tropes given in Table~\ref{tab:tropes}, one checks that there are 240 tuples of four tropes such that the four tropes satisfy a linear relation. Solving these equations in terms of $\{\mathsf{T}_4, \mathsf{T}_5, \mathsf{T}_6, \mathsf{T}_{456}\}$ yields the result.
\end{proof}
There are fifteen linear four-term relations that involve only the tropes $\mathsf{T}_1, \dots, \mathsf{T}_6$, i.e., the images of the odd symmetric Theta divisors $\mathsf{\Theta}_i$ on $\operatorname{Jac}(\mathcal{C})$:
\begin{corollary}
There are fifteen linear four-terms relations that involve only the tropes $\mathsf{T}_1, \dots, \mathsf{T}_6$. They are given by
\begin{equation}
\begin{split}
 -&(\lambda_j-\lambda_k)(\lambda_j-\lambda_l)(\lambda_k-\lambda_l)\mathsf{T}_i 
+ (\lambda_i-\lambda_k)(\lambda_i-\lambda_l)(\lambda_k-\lambda_l)\mathsf{T}_j\\
-&  (\lambda_i-\lambda_j)(\lambda_i-\lambda_l)(\lambda_j-\lambda_l)\mathsf{T}_k
+ (\lambda_i-\lambda_j)(\lambda_i-\lambda_k)(\lambda_j-\lambda_k)\mathsf{T}_l=0,
\end{split}
\end{equation}
with $1\le i < j < k <l \le 5$ and
\begin{equation}
\begin{split}
 -&(\lambda_j-\lambda_k)\mathsf{T}_i + (\lambda_i-\lambda_k)\mathsf{T}_j-  (\lambda_i-\lambda_j)\mathsf{T}_k
+ (\lambda_i-\lambda_j)(\lambda_i-\lambda_k)(\lambda_j-\lambda_k)\mathsf{T}_6=0.
\end{split}
\end{equation}
Moreover, these fifteen equations have rank three, and a generating set is given by
\begin{equation}
\label{eqn:tropes_QR}
\begin{split}
 \mathsf{T}_1 & = (1-\lambda_1) \mathsf{T}_4 +\lambda_1 \mathsf{T}_5 + \lambda_1 (\lambda_1-1) \mathsf{T}_6,\\
 \mathsf{T}_2 & = (1-\lambda_2) \mathsf{T}_4 +\lambda_2 \mathsf{T}_5 +  \lambda_2 (\lambda_2-1) \mathsf{T}_6,\\
 \mathsf{T}_3 & = (1-\lambda_3) \mathsf{T}_4 +\lambda_3 \mathsf{T}_5 +  \lambda_3 (\lambda_3-1) \mathsf{T}_6.
\end{split}
\end{equation}
\end{corollary}
We consider the blow up $p: \tilde{\mathbf{A}} \to \mathbf{A}$ of the sixteen two-torsion points with the exceptional curves $E_1, \dots, E_{16}$. The linear system $|4 p^* \mathsf{\Theta} - \sum E_i|$ determines a morphism of degree two from $\tilde{\mathbf{A}}$ to a complete intersection of three quadrics in $\mathbb{P}^5$. In particular, the image is the Kummer surface $\operatorname{Kum}(\mathbf{A})=\tilde{\mathbf{A}}/\langle -\mathbb{I} \rangle$. The net spanned by these quadrics is isomorphic to $\mathbb{P}^2$ with a discriminant locus corresponding to the union of six lines.  We have the following:
\begin{proposition}
The Kummer surface $\operatorname{Kum}(\operatorname{Jac}\mathcal{C})$ associated with the Jacobian $\operatorname{Jac}(\mathcal{C})$ of a genus-two curve~$\mathcal{C}$ in Rosenhain normal form~(\ref{Eq:Rosenhain}) is the complete intersection of three quadrics in $\mathbb{P}^5$. For $[\mathsf{t}_1:\mathsf{t}_2:\mathsf{t}_3:\mathsf{t}_4:\mathsf{t}_5:\mathsf{t}_6] \in \mathbb{P}^5$ with $\mathsf{t}_i^2=\mathsf{T}_i$ and $1 \le i \le 6$, the Kummer surface $\operatorname{Kum}(\operatorname{Jac}\mathcal{C})$ is given by the intersection of the three quadrics
\begin{equation}
\label{Kum:quadrics}
\begin{split}
 \mathsf{t}^2_1 & = (1-\lambda_1) \mathsf{t}^2_4 +\lambda_1 \mathsf{t}^2_5 +  \lambda_1 (\lambda_1-1) \mathsf{t}^2_6,\\
 \mathsf{t}^2_2 & = (1-\lambda_2) \mathsf{t}^2_4 +\lambda_2 \mathsf{t}^2_5 + \lambda_2 (\lambda_2-1) \mathsf{t}^2_6,\\
 \mathsf{t}^2_3 & = (1-\lambda_3) \mathsf{t}^2_4 +\lambda_3 \mathsf{t}^2_5 +  \lambda_3 (\lambda_3-1) \mathsf{t}^2_6.
\end{split}
\end{equation}
\end{proposition}
\begin{proof}
We proved that the Shioda sextic in Lemma~\ref{lem:Shioda} is given by the double cover of $\mathbb{P}^2$ branched along the reducible sextic that is the union of the six lines given by the tropes $\mathsf{T}_1, \dots, \mathsf{T}_6$ , i.e., 
\[
 \tilde{z}_4^2 = \mathsf{T}_1 \mathsf{T}_2 \mathsf{T}_3 \mathsf{T}_4 \mathsf{T}_5 \mathsf{T}_6 \;.
\]
The tropes $\mathsf{T}_1, \dots, \mathsf{T}_6$ satisfy the fifteen linear relations in Lemma~\ref{lem:LinearRelation} of rank three equivalent to Equations~(\ref{Kummer:topesLRa}). Introducing $\mathsf{T}_i=\mathsf{t}_i^2$ such that $\tilde{z}_4=\mathsf{t}_1 \mathsf{t}_2 \mathsf{t}_3 \mathsf{t}_4 \mathsf{t}_5 \mathsf{t}_6$, the Kummer surface $\operatorname{Kum}(\mathbf{A})=\tilde{\mathbf{A}}/\langle -\mathbb{I} \rangle$ is the complete intersection of three quadrics in $\mathbb{P}^5 \ni [\mathsf{t}_1:\mathsf{t}_2:\mathsf{t}_3:\mathsf{t}_4:\mathsf{t}_5:\mathsf{t}_6]$ given by Equations~(\ref{Kum:quadrics}).
\end{proof}
The coordinates $[\mathsf{t}_1:\mathsf{t}_2:\mathsf{t}_3:\mathsf{t}_4:\mathsf{t}_5:\mathsf{t}_6] \in \mathbb{P}^5$ are related to the odd Theta functions.  We have the following:
\begin{lemma}
\label{lem:bijection_tropes_thetas}
Given Thomae's formula~(\ref{Picard}), there is a unique bijection between the six tropes $\mathsf{T}_1, \dots, \mathsf{T}_6$, i.e., the images of the odd symmetric Theta divisors $\mathsf{\Theta}_i$ on $\operatorname{Jac}(\mathcal{C})$, and squares of the odd Theta functions $\theta^2_{11}(z), \dots,\theta^2_{16}(z)$ such that Equations~(\ref{eqn:tropes_QR}) coincide with the Mumford relations in Equations~(\ref{eqn:Mumford_odd}). It is given by setting
\begin{equation}
\label{eqn:sol_MatchTropes}
\begin{array}{lcllcl}
 \mathsf{T}_1	&=& R^2 \theta_1^2 \theta_3^2 \theta_6^2 \theta_7^2 \theta_9^2 \theta_{10}^2 \, \theta_{15}^2(z),&
 \mathsf{T}_2	&=& R^2  \theta_2^2 \theta_3^2 \theta_5^2 \theta_6^2 \theta_8^2 \theta_9^2 \, \theta_{12}^2(z),\\
 \mathsf{T}_3	&=& R^2 \theta_1^2 \theta_4^2 \theta_5^2 \theta_6^2 \theta_7^2 \theta_8^2 \, \theta_{11}^2(z),&
 \mathsf{T}_4	&=& R^2 \theta_1^2 \theta_2^2 \theta_3^2 \theta_4^2 \theta_8^2 \theta_{10}^2 \, \theta_{16}^2(z),\\
 \mathsf{T}_5	&=& R^2 \theta_2^2 \theta_4^2 \theta_5^2 \theta_7^2 \theta_9^2 \theta_{10}^2 \, \theta_{14}^2(z),&
 \mathsf{T}_6	&=& R^2 \theta_2^4 \theta_4^4 \theta_{10}^4 \, \theta_{13}^2(z),
\end{array}
\end{equation}
where $R\in \mathbb{C}^{*}$ is a non-zero constant.
\end{lemma}
\begin{proof}
Using a computer algebra system, we showed that Equations~(\ref{eqn:sol_MatchTropes}) is the only solution such that Equations~(\ref{Kummer:topesLRa}) coincide with Equations~(\ref{eqn:Mumford_odd}).
\end{proof}
\par We introduce expressions $\mathsf{t}_i$, that is sections of appropriate lines bundles over $\mathbf{A}_{\tau}$ (see Remarks~\ref{fact:sections1} and ~\ref{fact:sections2}), such that $\mathsf{t}_i^2=\mathsf{T}_i$ for $1 \le i \le 6$. We set
\begin{equation}
\label{def:t_i}
  \mathsf{t}_1	= R \, \theta_1 \theta_3 \theta_6 \theta_7 \theta_9 \theta_{10} \, \theta_{15}(z), \quad
 \mathsf{t}_2	= R \,  \theta_2 \theta_3 \theta_5 \theta_6 \theta_8 \theta_9 \, \theta_{12}(z), \quad \text{etc.}
\end{equation}
We have the following:
\begin{theorem}
For $[\tau] \in \mathcal{A}_2(2,4)$ the Kummer surface $\operatorname{Kum}(\mathbf{A}_{\tau})$ associated with the principally polarized abelian surfaces $\mathbf{A}_{\tau}$ with period matrix $(\mathbb{I}_2,\tau)$ is isomorphic to the image of the odd Theta functions $[\theta_{11}(z): \dots: \theta_{16}(z)]$ in $\mathbb{P}^5$ which satisfy the Mumford relations~(\ref{eqn:Mumford_odd}).
\end{theorem}
\begin{proof}
Using Lemma~\ref{lem:bijection_tropes_thetas} we obtain the linear map of $\mathbb{P}^5$ given by $[\mathsf{T}_1: \dots : \mathsf{T}_6] \mapsto [\theta_{11}(z): \dots: \theta_{16}(z)]$ which is well defined over $\mathcal{A}_2(2,4)$. Equations~(\ref{Kummer:topesLRa}) then coincide with Equations~(\ref{eqn:Mumford_odd}).
\end{proof}
\par In Lemma~\ref{lem:bijection_tropes_thetas} we determined the unique bijection between the tropes $\mathsf{T}_1, \dots, \mathsf{T}_6$, i.e., the images of the odd symmetric Theta divisors $\mathsf{\Theta}_i$ on $\operatorname{Jac}(\mathcal{C})$, and squares of the odd Theta functions $\theta^2_{11}(z), \dots,\theta^2_{16}(z)$ such that Equations~(\ref{eqn:tropes_QR}) coincide with the Mumford identities in Equations~(\ref{eqn:Mumford_odd}). We now extend the bijection to all tropes. We have the following:
\begin{proposition}
\label{lem:bijection_tropes_thetas2}
Given Thomae's formula~(\ref{Picard}), there is a unique bijection between the tropes $\mathsf{T}_1, \dots, \mathsf{T}_6$ and $\mathsf{T}_{126}, \dots, \mathsf{T}_{456}$ and the squares of the Theta functions $\theta^2_{1}(z), \dots,\theta^2_{16}(z)$ such that Equations~(\ref{eqn:tropes_QR}) coincide with the Mumford identities in Equations~(\ref{eqn:Mumford_odd}) and Equations~(\ref{eqn:Mumford_mix}). It is given by Equations~(\ref{eqn:sol_MatchTropes}) and by
\begin{equation}
\label{eqn:sol_MatchTropes2}
\begin{array}{lcllcl}
 \mathsf{T}_{126}	&=& k \, \theta_1^2 \theta_2^2 \theta_3^2 \theta_5^2 \theta_6^2 \theta_7^2 \theta_8^2 \theta_9^2\theta_{10}^2 \, \theta_{4}^2(z),&
 \mathsf{T}_{136}	&=& k \, \theta_1^2 \theta_3^2 \theta_4^2 \theta_5^2 \theta_6^2 \theta_7^2 \theta_8^2 \theta_9^2\theta_{10}^2 \, \theta_{2}^2(z),\\
 \mathsf{T}_{146}	&=& k \, \theta_1^2 \theta_2^2 \theta_3^2 \theta_4^2 \theta_6^2 \theta_7^2 \theta_8^2 \theta_9^2\theta_{10}^2 \, \theta_{5}^2(z),&
 \mathsf{T}_{156}	&=& k \, \theta_1^2 \theta_2^2 \theta_3^2 \theta_4^2 \theta_5^2 \theta_6^2 \theta_7^2 \theta_9^2\theta_{10}^2 \, \theta_{8}^2(z),\\
 \mathsf{T}_{236}	&=& k \, \theta_1^2 \theta_2^2 \theta_3^2 \theta_4^2 \theta_5^2 \theta_6^2 \theta_7^2 \theta_8^2\theta_9^2 \, \theta_{10}^2(z),&
 \mathsf{T}_{246}	&=& k \, \theta_1^2 \theta_2^2 \theta_3^2 \theta_4^2 \theta_5^2 \theta_6^2 \theta_8^2 \theta_9^2\theta_{10}^2 \, \theta_{7}^2(z),\\
 \mathsf{T}_{256}	&=& k \, \theta_2^2 \theta_3^2 \theta_4^2 \theta_5^2 \theta_6^2 \theta_7^2 \theta_8^2 \theta_9^2\theta_{10}^2 \, \theta_{1}^2(z),&
 \mathsf{T}_{346}	&=& k \, \theta_1^2 \theta_2^2 \theta_3^2 \theta_4^2 \theta_5^2 \theta_6^2 \theta_7^2 \theta_8^2\theta_{10}^2 \, \theta_{9}^2(z),\\
 \mathsf{T}_{356}	&=& k \, \theta_1^2 \theta_2^2 \theta_4^2 \theta_5^2 \theta_6^2 \theta_7^2 \theta_8^2\theta_9^2 \theta_{10}^2 \, \theta_{3}^2(z),&
 \mathsf{T}_{456}	&=& k \, \theta_1^2 \theta_2^2 \theta_3^2 \theta_4^2 \theta_5^2 \theta_7^2 \theta_8^2 \theta_9^2\theta_{10}^2 \, \theta_{6}^2(z),
\end{array}
\end{equation}
where $k\in \mathbb{C}^{*}$ is a non-zero constant, and $R^2=-\theta_2^2 \theta_4^2 \theta_{10}^2 \, k$ in Equations~(\ref{eqn:sol_MatchTropes2}).
\end{proposition}
\begin{proof}
Using a computer algebra system, we showed that Equations~(\ref{eqn:sol_MatchTropes2}) constitute the only solution such that Equations~(\ref{Kummer:topesLRa}) and Equations~(\ref{Kummer:topesLRb}) coincide with Equations~(\ref{eqn:Mumford_odd}) and Equations~(\ref{eqn:Mumford_mix}), respectively.
\end{proof}
\begin{remark}
The isomorphism between $\mathbf{A}[2]$ and $\mathbb{F}_2^4$ in Table~\ref{tab:16_6_configuration} in Lemma~\ref{lem:bijection_tropes_thetas} was determined. It was the unique isomorphism that mapped the Theta divisors to the translates of the characteristics of the Theta functions such that the symmetric Theta divisors are the tropes on the Kummer surface given by the squares of Theta functions.
\end{remark}
We also introduce sections $\mathsf{t}_{a,b}$ which are bi-monomial expressions in terms of Theta functions, such that $\mathsf{t}_{a,b}^2 = \mathsf{T}_a \mathsf{T}_b$ for $a, b \in \{k, ij6\}$ with $1\le i < j <6$ and $1\le k <6$ and the following consistent choice for the sign of the 120 square roots:
\begin{small}
\begin{equation}
\label{def:t_ijk}
\begin{array}{lclclcl}
 \mathsf{t}_{1,2} & = & k \,  \theta_1 \theta_2^3 \theta_3^2 \theta_4^2 \theta_5 \theta_6^2 \theta_7 \theta_8 \theta_9^2 \theta_{10}^3 \, \theta_{12}(z)\theta_{15}(z),  
 	&\dots, &\mathsf{t}_{5,6}  &=& k\dots, \\ 
 \mathsf{t}_{1,126} & = & ik \,  \theta_1^2 \theta_2^2 \theta_3^2 \theta_4 \theta_5 \theta_6^2 \theta_7^2 \theta_8 \theta_9^2 \theta_{10}^3 \, \theta_{4}(z)\theta_{15}(z),  
 	&\dots, &\mathsf{t}_{5,456}  &=& ik\dots, \\ 
 \mathsf{t}_{6,126} & = & -ik \,  \theta_1 \theta_2^4 \theta_3 \theta_4^3 \theta_5 \theta_6 \theta_7 \theta_8 \theta_9 \theta_{10}^4 \, \theta_{4}(z)\theta_{13}(z),  
 	&\dots, &\mathsf{t}_{6,456}  &=& -ik\dots, \\ 
 \mathsf{t}_{126,136} & = & k \,  \theta_1^2 \theta_2 \theta_3^2 \theta_4 \theta_5^2 \theta_6^2 \theta_7^2 \theta_8^2 \theta_9^2 \theta_{10}^2 \, \theta_{2}(z)\theta_{4}(z),
 	&\dots, &\mathsf{t}_{356,456}  &=& k\dots.
 \end{array}
\end{equation}
\end{small}
We have the following:
\begin{proposition}
\label{prop:MMFT}
In terms of the sections $\mathsf{t}_{a,b}$ with $\mathsf{t}_{a,b}^2 = \mathsf{T}_a \mathsf{T}_b$ introduced above, the 72 Equations~(\ref{eqn:Mumford_odd}),~(\ref{eqn:Mumford_mix}), and~(\ref{Eqn:MumfordBimonomial}) are given by Equations~(\ref{Kummer:topesLRa}),~(\ref{Kummer:topesLRb}), and
\begin{gather}
\label{Kummer:topesLRc_eg}
  (\lambda_2-\lambda_3)(\lambda_1-1) \mathsf{t}_{4,6}+\mathsf{t}_{136, 256}-\mathsf{t}_{126, 356}=0, \; \dots \\
  \nonumber
(\lambda_1-\lambda_2)(\lambda_3-1) \mathsf{t}_{126,356}- (\lambda_1-\lambda_3)(\lambda_2-1) \mathsf{t}_{136,256}+ (\lambda_2-\lambda_3)(\lambda_1-1) \mathsf{t}_{156,236}=0,\\
\nonumber
 \text{(A complete generating set of 60 equations is given in Equation~(\ref{Kummer:topesLRc}).)}
\end{gather}
In particular, all coefficients of the 72 conics are polynomials in $\mathbb{Z}[\lambda_1,\lambda_2,\lambda_3]$.
\end{proposition}
We have the following:
\begin{corollary}
The image of the embedding of $\operatorname{Jac}(\mathcal{C}) \hookrightarrow \mathbb{P}^{15}$ of the Jacobian $\operatorname{Jac}(\mathcal{C})$ of a genus-two curve~$\mathcal{C}$ in Equation~(\ref{Eq:Rosenhain}) given by $z \mapsto [\mathsf{t}_1: \dots : \mathsf{t}_{16}]$ is the intersection of the 72 conics given by Equations~(\ref{Kummer:topesLRa}),~(\ref{Kummer:topesLRb}), and~(\ref{Kummer:topesLRc}).  Similarly, the image of the image of the embedding of $\operatorname{Kum}(\operatorname{Jac} \mathcal{C}) \hookrightarrow \mathbb{P}^{5}$ given by $z \mapsto [\mathsf{t}_1: \dots : \mathsf{t}_{6}]$ is the intersection of the 3 conics given by Equations~(\ref{Kummer:topesLRa}). In particular, the images are defined over $\mathcal{A}_2(2)$, i.e., in terms of the Rosenhain parameters of $\mathcal{C}$
in Equation~(\ref{Eq:Rosenhain}).
\end{corollary}
\qed
\subsection{Tetrahedra and even eights}
We now use the Cassels-Flynn quartic in Equation~(\ref{kummer})  as model for a singular Kummer variety to derive special relations between the tropes.  The $16_6$ configuration on the Jacobian descends to the Kummer surface given by Equation~(\ref{kummer}) as follows: any trope contains exactly six nodes. Any node is contained in exactly six tropes. Any two different tropes have exactly two nodes in common. This is easily verified using Table~\ref{tab:tropes}. Tetrahedra in $\mathbb{P}^3$ whose faces are tropes are given by tuples $\{\mathsf{T}_a, \mathsf{T}_b, \mathsf{T}_c, \mathsf{T}_d\}$. They are called \emph{Rosenhain tetrahedra} if all vertices are nodes. They are called \emph{G\"opel tetrahedra} if none of the vertices are nodes. We also remind the reader that we call the tropes $\{ \mathsf{T}_i \}$ and $\{ \mathsf{T}_{jk6} \}$, with $1\le i \le6$ and $1\le j < k<6$, the six \emph{odd} and the ten \emph{even} tropes, respectively. We have the following:
\begin{lemma}
There are 60 G\"opel tetrahedra on the quartic surface given by Equation~(\ref{kummer}) which are obtained for $\lbrace i, j, k, l, m  \rbrace= \lbrace 1, \dots, 5\rbrace$ as follows:
\begin{enumerate}
\item $30$ tetrahedra are of the form $\lbrace  \mathsf{T}_i, \mathsf{T}_j, \mathsf{T}_{ik6}, \mathsf{T}_{jk6} \rbrace$;
\item $15$ tetrahedra are of the form $\lbrace  \mathsf{T}_m, \mathsf{T}_{6}, \mathsf{T}_{ij6}, \mathsf{T}_{kl6} \rbrace$;
\item $15$ tetrahedra are of the form $\lbrace  \mathsf{T}_{ij6}, \mathsf{T}_{kl6},  \mathsf{T}_{ik6}, \mathsf{T}_{jl6}  \rbrace$.
\end{enumerate}
In particular, 15 G\"opel tetrahedra contain only even tropes, and 45 contain two even and two odd tropes.
\end{lemma}
\begin{proof}
One checks that the given list contains all sets of four tropes $\{\mathsf{T}_a, \mathsf{T}_b, \mathsf{T}_c, \mathsf{T}_d\}$ whose vertices are not nodes.
\end{proof}
 \begin{lemma}
There are 80 Rosenhain tetrahedra $\{\mathsf{T}_a, \mathsf{T}_b, \mathsf{T}_c, \mathsf{T}_d\}$ on the surface given by Equation~(\ref{kummer}). The Rosenhain tetrahedra fall into five subsets $\mathsf{R}^{(1)}, \dots, \mathsf{R}^{(5)}$ such that all Rosenhain tetrahedra are obtained for $\lbrace i, j, k, l, m  \rbrace= \lbrace 1, \dots, 5\rbrace$ as follows:
\begin{equation}
\begin{array}{c|c|llll}
 \text{notation} & \# & \mathsf{T}_a & \mathsf{T}_b & \mathsf{T}_c &\mathsf{T}_d \\
\hline
\mathsf{R}^{(1)}%(i,j,k;l,m) 
& 10 &\mathsf{T}_{i} & \mathsf{T}_{j} & \mathsf{T}_{k} & \mathsf{T}_{lm6} \\
\mathsf{R}^{(2)}%(i,j;k) 
&10 &\mathsf{T}_{jk6} & \mathsf{T}_{ik6} & \mathsf{T}_{ij6} & \mathsf{T}_{6} \\
\mathsf{R}^{(3)}%(i,j) 
&10 &\mathsf{T}_{i} & \mathsf{T}_{j} & \mathsf{T}_{ij6} & \mathsf{T}_{6} \\
\mathsf{R}^{(4)}%(i;j,k;l,m) 
& 30 &\mathsf{T}_{i} & \mathsf{T}_{ij6} & \mathsf{T}_{ik6} & \mathsf{T}_{lm6} \\
\mathsf{R}^{(5)}%(i;j;k,l,m) 
& 20 &\mathsf{T}_{i} & \mathsf{T}_{jk6} & \mathsf{T}_{jl6} & \mathsf{T}_{jm6} \\
\end{array}
\end{equation}
In particular, there are 60 Rosenhain tetrahedra with one odd and three even tropes, and 20 that contain one even and three odd tropes.
\end{lemma}
\begin{proof}
One checks that the given list contains all sets of four tropes $\{\mathsf{T}_a, \mathsf{T}_b, \mathsf{T}_c, \mathsf{T}_d\}$ whose vertices are nodes.
\end{proof}
One easily checks the following:
\begin{lemma}
Under the isomorphism in Table~\ref{tab:16_6_configuration}, the G\"opel and Rosenhain tetrahedra are bijectively mapped to the G\"opel groups and their translates in Lemma~\ref{lem:Goepel} and the Rosenhain groups and their translates in Lemma~\ref{lem:Rosenhain}, respectively.
\end{lemma} 
\qed
\par By explicit computation one also checks:
\begin{lemma}
\label{lem:quadratics}
There are 30 quadratic relations involving eight tropes. The relations can be written in the form
\begin{equation}
\label{QuadraticEquation}
 \mu\nu\rho \mathsf{T}_a \mathsf{T}_{a'} + \gamma\delta\mu  \mathsf{T}_b \mathsf{T}_{b'} + \beta\delta \nu \mathsf{T}_c \mathsf{T}_{c'} + \beta\gamma\rho \mathsf{T}_d \mathsf{T}_{d'} =0 ,
\end{equation}
where $\beta, \gamma, \delta, \mu, \nu, \gamma \in \mathbb{C}[\lambda_1,\lambda_2,\lambda_3]$ with $\mu+\nu+\rho=0$ and $\beta\mu+\gamma\nu+\delta\rho=0$, and $\{\mathsf{T}_a, \mathsf{T}_b, \mathsf{T}_c, \mathsf{T}_d\}$ and $\{\mathsf{T}_{a'}, \mathsf{T}_{b'}, \mathsf{T}_{c'}, \mathsf{T}_{d'}\}$ are two disjoint Rosenhain tetrahedra. In particular, all quadratic relations are obtained for $\lbrace i, j, k, l, m  \rbrace= \lbrace 1, \dots, 5\rbrace$ as follows:
\begin{equation}
\label{tab:RosenhainT}
\scalemath{0.85}{
\begin{array}{r|c|llll|llll}
\#	& \mathsf{R}'s 					& \mathsf{T}_a & \mathsf{T}_b & \mathsf{T}_c &\mathsf{T}_d 
								& \mathsf{T}_{a'} & \mathsf{T}_{b'} & \mathsf{T}_{c'} & \mathsf{T}_{d'} \\
\hline
10 	& \mathsf{R}^{(3)}, \mathsf{R}^{(4)} 	& \mathsf{T}_{6}&\mathsf{T}_i& \mathsf{T}_j& \mathsf{T}_{ij6}
								& \mathsf{T}_{lm6}& \mathsf{T}_{jk6}& \mathsf{T}_{ik6}& \mathsf{T}_{k}\\
%	&\mathsf{R}^{(1)}, \mathsf{R}^{(2)} 	& \mathsf{T}_{lm6}& \mathsf{T}_i& \mathsf{T}_j& \mathsf{T}_k
%								& \mathsf{T}_6& \mathsf{T}_{jk6}& \mathsf{T}_{ik6}& \mathsf{T}_{ij6}\\
	&& \multicolumn{8}{|l}{\mu=\lambda_j -\lambda_k, \nu=\lambda_k-\lambda_i, \rho=\lambda_i -\lambda_j}\\
	&& \multicolumn{8}{|l}{\beta=\gamma=\delta=1}\\
\hline
10 	&\mathsf{R}^{(4)}, \mathsf{R}^{(4)} 	& \mathsf{T}_i& \mathsf{T}_{jm6}& \mathsf{T}_{ik6}& \mathsf{T}_{il6}
								& \mathsf{T}_j& \mathsf{T}_{im6}& \mathsf{T}_{jk6}& \mathsf{T}_{jl6}\\
%	&\mathsf{R}^{(5)}, \mathsf{R}^{(5)} 	& \mathsf{T}_i& \mathsf{T}_{jm6}& \mathsf{T}_{jk6}& \mathsf{T}_{jl6}
%								& \mathsf{T}_j& \mathsf{T}_{im6}& \mathsf{T}_{ik6}& \mathsf{T}_{il6}\\
	&& \multicolumn{8}{|l}{\mu=\lambda_l-\lambda_k, \nu=\lambda_m-\lambda_l, \rho=\lambda_k-\lambda_m}\\						
	&& \multicolumn{8}{|l}{\beta=\gamma=\delta=1}\\
\hline
5 	&\mathsf{R}^{(1)}, \mathsf{R}^{(5)} 	& \mathsf{T}_{lm6}& \mathsf{T}_i& \mathsf{T}_j& \mathsf{T}_k
								& \mathsf{T}_l& \mathsf{T}_{im6}& \mathsf{T}_{jm6}	& \mathsf{T}_{km6}\\
	&& \multicolumn{8}{|l}{\mu=\lambda_j-\lambda_k, \nu=\lambda_k-\lambda_i, \rho=\lambda_i-\lambda_j}\\[0.2em]
	&& \multicolumn{8}{|l}{\beta =\lambda_i-\lambda_l, \gamma = \lambda_j-\lambda_l, \delta = \lambda_k-\lambda_l}\\
\hline
5 	&\mathsf{R}^{(2)}, \mathsf{R}^{(5)} 	& \mathsf{T}_6& \mathsf{T}_{jk6}& \mathsf{T}_{ik6}& \mathsf{T}_{ij6}
								& \mathsf{T}_m& \mathsf{T}_{il6}& \mathsf{T}_{jl6}	& \mathsf{T}_{kl6}\\
	&& \multicolumn{8}{|l}{\mu=(\lambda_j-\lambda_k)(\lambda_i-\lambda_l),\nu=(\lambda_k-\lambda_i)(\lambda_j-\lambda_l),\rho=(\lambda_i-\lambda_j)(\lambda_k-\lambda_l)}\\
	&& \multicolumn{8}{|l}{\beta=\gamma=\delta=1}\\
	\hline
\end{array}}
\end{equation}
\end{lemma}
\begin{proof}
The proof follows from direct computation using the explicit equations for the tropes in Table~\ref{tab:tropes}.
\end{proof}
\begin{lemma}
\label{lem:quadraticsEE}
\begin{enumerate}
\item[]
\item Every set of eight tropes given in Lemma~\ref{lem:quadratics} can be decomposed into three different pairs of disjoint G\"opel tetrahedra. 
\item Every G\"opel tetrahedron is contained in three sets of eight tropes. 
\item Every set of eight tropes given in Lemma~\ref{lem:quadratics} can be decomposed into four different pairs of disjoint Rosenhain tetrahedra. Eight nodes arise as the vertices of each such pair of Rosenhain tetrahedra. Moreover, the set of eight nodes is independent of the chosen decomposition.
\item Every Rosenhain tetrahedron is contained in three sets of eight tropes.
\end{enumerate}
\end{lemma}
\begin{proof}
The proof follows from direct computation.
\end{proof}
On a singular Kummer variety, sets of eight distinct nodes are called \emph{even eights}.  After minimal resolution, an even eight is a set of eight disjoint smooth rational curves on the smooth Kummer surface whose divisors add up to an even element in the N\'eron-Severi group. Nikulin proved \cite{MR0429917} that on a Kummer surface there are 30 even eights. If we fix a node, say $p_0$, the 15 even eights not containing $p_0$ are enumerated by nodes $p_{ij}$ as follows
\[
 \Delta_{ij} = \{ p_{1i}, \dots , \widehat{p_{ij}} , \dots , p_{i6} , p_{1j} , \dots , \widehat{p_{ij}} , \dots , p_{j6} \} \;,
 \]
where $p_{11}=0$ and a hat indicates a node that is not part of the even eight; see Mehran \cite{MR2804549}. On the other hand, sets of eight distinct nodes arose  in Lemma~\ref{lem:quadratics}.(3) as the vertices of the two Rosenhain tetrahedra forming each set of eight tropes in Lemma~\ref{lem:quadratics}. We have the following:
\begin{proposition}
\label{prop:EvenEights}
The 30 sets of eight tropes satisfying a quadratic relation in Lemma~\ref{lem:quadratics} are in one-to-one correspondence with the 30 even eights where we denote an even eight not containing $p_0$ by $\Delta_{ij}$ and its complement by $\Delta^{\complement}_{ij}$:
\begin{equation*}
\scalemath{0.85}{
\begin{array}{c|l||c|l}
\Delta_{12} 	& \{ \mathsf{T}_1, \mathsf{T}_{136}, \mathsf{T}_{146}, \mathsf{T}_{156}, \mathsf{T}_{2}, \mathsf{T}_{236}, \mathsf{T}_{246}, \mathsf{T}_{256} \}
& 	\Delta^{\complement}_{12} & \{\mathsf{T}_{126}, \mathsf{T}_3, \mathsf{T}_{346}, \mathsf{T}_{356}, \mathsf{T}_{4}, \mathsf{T}_{456}, \mathsf{T}_{5}, \mathsf{T}_{6}\}\\[0.2em]
\Delta_{13} 	& \{ \mathsf{T}_1, \mathsf{T}_{126}, \mathsf{T}_{146}, \mathsf{T}_{156}, \mathsf{T}_{236}, \mathsf{T}_3, \mathsf{T}_{346}, \mathsf{T}_{356} \}
& 	\Delta^{\complement}_{13} & \{\mathsf{T}_{136}, \mathsf{T}_2, \mathsf{T}_{246}, \mathsf{T}_{256}, \mathsf{T}_{4}, \mathsf{T}_{456}, \mathsf{T}_{5}, \mathsf{T}_{6}\}\\[0.2em]
\Delta_{14} 	&\{ \mathsf{T}_1, \mathsf{T}_{126}, \mathsf{T}_{136}, \mathsf{T}_{156}, \mathsf{T}_{246}, \mathsf{T}_{346}, \mathsf{T}_{4}, \mathsf{T}_{456}\}
& 	\Delta^{\complement}_{14} &  \{ \mathsf{T}_{146}, \mathsf{T}_{2}, \mathsf{T}_{236}, \mathsf{T}_{256}, \mathsf{T}_{3}, \mathsf{T}_{356}, \mathsf{T}_{5}, \mathsf{T}_6\}
\\[0.2em]
\Delta_{15} 	& \{ \mathsf{T}_1, \mathsf{T}_{126}, \mathsf{T}_{136}, \mathsf{T}_{146}, \mathsf{T}_{256}, \mathsf{T}_{356}, \mathsf{T}_{456}, \mathsf{T}_5\}
& 	  \Delta^{\complement}_{15} & \{ \mathsf{T}_{156}, \mathsf{T}_{2}, \mathsf{T}_{236}, \mathsf{T}_{246}, \mathsf{T}_{3}, \mathsf{T}_{346}, \mathsf{T}_{4}, \mathsf{T}_6\}
\\[0.2em]
\Delta_{16} 	& \{ \mathsf{T}_1, \mathsf{T}_{236}, \mathsf{T}_{246}, \mathsf{T}_{256}, \mathsf{T}_{346}, \mathsf{T}_{356}, \mathsf{T}_{456}, \mathsf{T}_{6} \}
& 	\Delta^{\complement}_{16} & \{\mathsf{T}_{126}, \mathsf{T}_{136}, \mathsf{T}_{146}, \mathsf{T}_{156}, \mathsf{T}_{2}, \mathsf{T}_{3}, \mathsf{T}_{4}, \mathsf{T}_{5}\}\\[0.2em]
\Delta_{23} 	& \{ \mathsf{T}_{126}, \mathsf{T}_{136}, \mathsf{T}_{2}, \mathsf{T}_{246}, \mathsf{T}_{256}, \mathsf{T}_{3}, \mathsf{T}_{346}, \mathsf{T}_{356} \}
& 	\Delta^{\complement}_{23} & \{\mathsf{T}_{1}, \mathsf{T}_{146}, \mathsf{T}_{156}, \mathsf{T}_{236}, \mathsf{T}_{4}, \mathsf{T}_{456}, \mathsf{T}_{5}, \mathsf{T}_{6}\}
\\[0.2em]
\Delta_{24} 	& \{ \mathsf{T}_{126}, \mathsf{T}_{146}, \mathsf{T}_{2}, \mathsf{T}_{236}, \mathsf{T}_{256}, \mathsf{T}_{346}, \mathsf{T}_{4}, \mathsf{T}_{456} \}
& 	\Delta^{\complement}_{24} & \{\mathsf{T}_{1}, \mathsf{T}_{136}, \mathsf{T}_{156}, \mathsf{T}_{246}, \mathsf{T}_{3}, \mathsf{T}_{356}, \mathsf{T}_{5}, \mathsf{T}_{6}\}\\[0.2em]
\Delta_{25} 	& \{ \mathsf{T}_{126}, \mathsf{T}_{156}, \mathsf{T}_{2}, \mathsf{T}_{236}, \mathsf{T}_{246}, \mathsf{T}_{356}, \mathsf{T}_{456}, \mathsf{T}_{5} \} 
& 	\Delta^{\complement}_{25} &\{\mathsf{T}_{1}, \mathsf{T}_{136}, \mathsf{T}_{146}, \mathsf{T}_{256}, \mathsf{T}_{3}, \mathsf{T}_{346}, \mathsf{T}_{4}, \mathsf{T}_{6}\}
\\[0.2em]
\Delta_{26} 	& \{ \mathsf{T}_{136}, \mathsf{T}_{146}, \mathsf{T}_{156}, \mathsf{T}_{2}, \mathsf{T}_{346}, \mathsf{T}_{356}, \mathsf{T}_{456}, \mathsf{T}_{6} \} 
& 	\Delta^{\complement}_{26} &\{\mathsf{T}_{1}, \mathsf{T}_{126}, \mathsf{T}_{236}, \mathsf{T}_{246}, \mathsf{T}_{256}, \mathsf{T}_{3}, \mathsf{T}_{4}, \mathsf{T}_{5}\}
\\[0.2em]
\Delta_{34} 	& \{ \mathsf{T}_{136}, \mathsf{T}_{146}, \mathsf{T}_{236}, \mathsf{T}_{246}, \mathsf{T}_{3}, \mathsf{T}_{356}, \mathsf{T}_{4}, \mathsf{T}_{456} \}
& 	\Delta^{\complement}_{34} & \{\mathsf{T}_{1}, \mathsf{T}_{126}, \mathsf{T}_{156}, \mathsf{T}_{2}, \mathsf{T}_{256}, \mathsf{T}_{346}, \mathsf{T}_{5}, \mathsf{T}_{6}\}\\[0.2em]
\Delta_{35} 	& \{ \mathsf{T}_{136}, \mathsf{T}_{156}, \mathsf{T}_{236}, \mathsf{T}_{256}, \mathsf{T}_{3}, \mathsf{T}_{346}, \mathsf{T}_{456}, \mathsf{T}_{5} \}
& 	\Delta^{\complement}_{35} & \{\mathsf{T}_{1}, \mathsf{T}_{126}, \mathsf{T}_{146}, \mathsf{T}_{2}, \mathsf{T}_{246}, \mathsf{T}_{356}, \mathsf{T}_{4}, \mathsf{T}_{6}\}\\[0.2em]
\Delta_{36} 	& \{ \mathsf{T}_{126}, \mathsf{T}_{146}, \mathsf{T}_{156}, \mathsf{T}_{246}, \mathsf{T}_{256}, \mathsf{T}_{3}, \mathsf{T}_{456}, \mathsf{T}_{6} \}
& 	\Delta^{\complement}_{36} & \{\mathsf{T}_{1}, \mathsf{T}_{136}, \mathsf{T}_{2}, \mathsf{T}_{236}, \mathsf{T}_{346}, \mathsf{T}_{356}, \mathsf{T}_{4}, \mathsf{T}_{5}\}
\\[0.2em]
\Delta_{45} 	&\{ \mathsf{T}_{146}, \mathsf{T}_{156}, \mathsf{T}_{246}, \mathsf{T}_{256}, \mathsf{T}_{346}, \mathsf{T}_{356}, \mathsf{T}_{4}, \mathsf{T}_{5} \} 
& 	\Delta^{\complement}_{55} & \{\mathsf{T}_{1}, \mathsf{T}_{126}, \mathsf{T}_{136}, \mathsf{T}_{2}, \mathsf{T}_{236}, \mathsf{T}_{3}, \mathsf{T}_{456}, \mathsf{T}_{6}\}
\\[0.2em]
\Delta_{46} 	& \{ \mathsf{T}_{126}, \mathsf{T}_{136}, \mathsf{T}_{156}, \mathsf{T}_{236}, \mathsf{T}_{256}, \mathsf{T}_{356}, \mathsf{T}_{4}, \mathsf{T}_{6} \}
& 	\Delta^{\complement}_{46} & \{\mathsf{T}_{1}, \mathsf{T}_{146}, \mathsf{T}_{2}, \mathsf{T}_{246}, \mathsf{T}_{3}, \mathsf{T}_{346}, \mathsf{T}_{456}, \mathsf{T}_{5}\}\\[0.2em]
\Delta_{56} 	& \{ \mathsf{T}_{126}, \mathsf{T}_{136}, \mathsf{T}_{146}, \mathsf{T}_{236}, \mathsf{T}_{246}, \mathsf{T}_{346}, \mathsf{T}_{5}, \mathsf{T}_{6} \}
& 	\Delta^{\complement}_{56} & \{\mathsf{T}_{1}, \mathsf{T}_{156}, \mathsf{T}_{2}, \mathsf{T}_{256}, \mathsf{T}_{3}, \mathsf{T}_{356}, \mathsf{T}_{4}, \mathsf{T}_{456}\}\\[0.2em]
\end{array}}
\end{equation*}
In particular, the sets of eight tropes corresponding to even eights $\Delta_{ij}$ contain 6 even tropes $\mathsf{T}_{kl6}$ and 2 odd tropes $\mathsf{T}_{m}$ (second and fourth case in Equation~(\ref{tab:RosenhainT})).
\end{proposition}
\begin{proof}
Sets of eight distinct nodes arise  in Lemma~\ref{lem:quadratics}.(3) as the vertices of the two Rosenhain tetrahedra forming each set of eight tropes. Thus, we have two ways of constructing even eights. One checks that this defines a one-to-one correspondence between even eights and sets of eight tropes satisfying a quadratic relation. This one-to-one correspondence maps the even eights $\Delta_{14}, \Delta_{15}, \Delta_{23}, \Delta_{26}, \Delta_{36}, \Delta_{45}$ precisely to the sets of eight tropes obtained using the duality map between Theta divisors and Theta characteristics in Remark~\ref{rem:isos}. For the remaining even eights one has to take the complement. The proof then follows from a direct computation and Lemmas~\ref{lem:quadratics} and~\ref{lem:quadraticsEE}
\end{proof}
One easily checks the following:
\begin{lemma}
\label{lem:quadraticsEEb}
\begin{enumerate}
\item[]
\item Every G\"opel tetrahedron with all even tropes is contained in three sets of eight tropes corresponding to an even eight $\Delta_{ij}$.
\item Every Rosenhain tetrahedron with one odd and three even tropes is contained in two sets of eight tropes corresponding to an even eight $\Delta_{ij}$.
\end{enumerate}
\end{lemma}
\begin{proof}
The proof follows from direct computation.
\end{proof}

\begin{lemma}
The 15 sets of eight tropes corresponding to even eights $\Delta_{ij}$
\end{lemma}
\begin{corollary}
\label{cor:LR}
For every Rosenhain or G\"opel tetrahedron, the remaining twelve tropes are linear functions of the tropes contained in the tetrahedron with coefficients in $\mathbb{C}(\lambda_1,\lambda_2,\lambda_3)$.
\end{corollary}
\begin{proof}
For every Rosenhain or G\"opel tetrahedron, the projective coordinates $[z_1:z_2:z_3:z_4] \in \mathbb{P}^3$ are linear functions of the tropes contained in the tetrahedron with coefficients that are rational functions in $\mathbb{C}[\lambda_1,\lambda_2,\lambda_3]$. But all tropes are linear functions of the coordinates $[z_1:z_2:z_3:z_4] \in \mathbb{P}^3$ and were given in Table~\ref{tab:tropes}.
\end{proof}
\subsection{The irrational Kummer normal form}
The following lemma shows that every G\"opel tetrahedron can be completed into one of the sets of eight tropes given by Lemma~\ref{lem:quadratics}:
\begin{corollary}
\label{cor:Goepel}
Every G\"opel tetrahedron arises as union of two two-tuples chosen from 
\[\Big\{ \{ \mathsf{T}_a ,\mathsf{T}_{a'} \} , \{ \mathsf{T}_b, \mathsf{T}_{b'}\},  \{ \mathsf{T}_c, \mathsf{T}_{c'}\}, \{  \mathsf{T}_d, \mathsf{T}_{d'}\}\Big\},
\]
where $\{\mathsf{T}_a, \mathsf{T}_b, \mathsf{T}_c, \mathsf{T}_d\}$ and $\{\mathsf{T}_{a'}, \mathsf{T}_{b'}, \mathsf{T}_{c'}, \mathsf{T}_{d'}\}$ are two disjoint Rosenhain tetrahedra given in Lemma~\ref{lem:quadratics}. Moreover, every G\"opel tetrahedron is contained in exactly three different sets of eight tropes (out of 30) in Lemma~\ref{lem:quadratics}. Conversely, every set of eight tropes given in Lemma~\ref{lem:quadratics} can be decomposed into pairs of disjoint G\"opel tetrahedra in three different ways. 
\end{corollary}
\qed
\par Having already established quadratic relations between these sets of eight tropes, we can re-write the equation of a singular Kummer surface. We have the following:
\begin {lemma}
For every set of eight tropes given in Lemma~\ref{lem:quadratics}, the equation
\begin{equation}
\label{KummerNormalForm}
\begin{split}
 \left(\mu^2 \mathsf{T}_b \mathsf{T}_{b'} + \nu^2 \mathsf{T}_c \mathsf{T}_{c'} - \rho^2 \mathsf{T}_d \mathsf{T}_{d'}\right)^2 & = 4 \mu^2 \nu^2 \mathsf{T}_b \mathsf{T}_{b'}\mathsf{T}_c \mathsf{T}_{c'}
 \end{split}
\end{equation}
is equivalent to the Cassels-Flynn quartic in Equation~(\ref{kummer}).
\end{lemma}
\begin{proof}
The proof follows from direct computation using the explicit equations for the tropes in Table~\ref{tab:tropes}.
\end{proof}
\begin{remark}
Because of Equation~(\ref{QuadraticEquation}), $\mu +\nu + \rho=0$, and $\beta=\gamma=\delta=1$ or $\beta=\delta-\nu, \gamma=\delta+\mu$, Equation~(\ref{KummerNormalForm}) is identical to any of the following equations:
\begin{equation}
\label{AllKummer}
\begin{split}
 \left(\mu^2 \mathsf{T}_b \mathsf{T}_{b'} + \nu^2 \mathsf{T}_c \mathsf{T}_{c'} -\rho^2 \mathsf{T}_d \mathsf{T}_{d'}\right)^2 & = 4 \mu^2 \nu^2 \mathsf{T}_b \mathsf{T}_{b'}\mathsf{T}_c \mathsf{T}_{c'}\\
 \left(\mu^2 \mathsf{T}_b \mathsf{T}_{b'} - \nu^2 \mathsf{T}_c \mathsf{T}_{c'} +\rho^2 \mathsf{T}_d \mathsf{T}_{d'}\right)^2 & = 4 \mu^2 \rho^2 \mathsf{T}_b \mathsf{T}_{b'}T_d \mathsf{T}_{d'}, \\
  \left(\mu^2 \mathsf{T}_b \mathsf{T}_{b'} - \nu^2\mathsf{T}_c \mathsf{T}_{c'} -\rho^2 \mathsf{T}_d \mathsf{T}_{d'}\right)^2 & = 4 \nu^2 \rho^2 \mathsf{T}_c \mathsf{T}_{c'}\mathsf{T}_d \mathsf{T}_{d'}, \\
 \left(\rho^2 \mathsf{T}_a \mathsf{T}_{a'} + \gamma^2\mathsf{T}_b \mathsf{T}_{b'} -\beta^2 \mathsf{T}_c \mathsf{T}_{c'}\right)^2 & = 4 \gamma^2\rho^2\mathsf{T}_a \mathsf{T}_{a'}\mathsf{T}_b \mathsf{T}_{b'},\\ 
 \left(\nu^2 \mathsf{T}_a \mathsf{T}_{a'} - \delta^2\mathsf{T}_b \mathsf{T}_{b'} +\beta^2 \mathsf{T}_d \mathsf{T}_{d'}\right)^2 & = 4 \beta^2\nu^2\mathsf{T}_a \mathsf{T}_{a'}\mathsf{T}_d \mathsf{T}_{d'},\\ 
   \left(\mu^2 \mathsf{T}_a \mathsf{T}_{a'} +\delta^2 \mathsf{T}_c \mathsf{T}_{c'} -\gamma^2  \mathsf{T}_d \mathsf{T}_{d'} \right)^2 & = 4 \delta^2 \mu^2 \mathsf{T}_a \mathsf{T}_{a'}\mathsf{T}_c\mathsf{T}_{c'}.
\end{split}
\end{equation}
\end{remark}
We have the following:
\begin{proposition}
\label{prop:KummerNF}
For every G\"opel tetrahedron $\{\mathsf{T}_a, \mathsf{T}_{a'}, \mathsf{T}_b, \mathsf{T}_{b'}\}$, the Cassels-Flynn quartic in Equation~(\ref{kummer}) is equivalent to an equation of the form
\[
\left(\rho^2 \mathsf{T}_a \mathsf{T}_{a'} + \gamma^2\mathsf{T}_b \mathsf{T}_{b'} -\beta^2 \mathsf{T}_c \mathsf{T}_{c'}\right)^2  = 4 \gamma^2\rho^2\mathsf{T}_a \mathsf{T}_{a'}\mathsf{T}_b \mathsf{T}_{b'},
\]
where $\beta, \gamma, \rho$ are polynomials in $\mathbb{C}[\lambda_1,\lambda_2,\lambda_3]$, and $\mathsf{T}_c, \mathsf{T}_{c'}$ are two tropes that are linear functions of $\mathsf{T}_a, \mathsf{T}_{a'}, \mathsf{T}_b, \mathsf{T}_{b'}$ with coefficients in $\mathbb{C}(\lambda_1,\lambda_2,\lambda_3)$.
\end{proposition}
\begin{proof}
By Corollary~\ref{cor:Goepel}, each G\"opel tetrahedron on a Kummer quartic is obtained from two pairs of Rosenhain tetrahedra in Lemma~\ref{lem:quadratics} such that Equations~(\ref{AllKummer}) are equivalent to the Cassels-Flynn quartic in Equation~(\ref{kummer}). By Corollary~\ref{cor:LR}, the remaining twelve tropes are linear functions of the tropes contained in the G\"opel tetrahedron with coefficients that are rational functions in $\mathbb{C}[\lambda_1,\lambda_2,\lambda_3]$. 
\end{proof}
\par We give the following example:
\begin{example}
If we choose the two Rosenhain tetrahedra 
\[ \{\mathsf{T}_{lm6}, \mathsf{T}_i, \mathsf{T}_j, \mathsf{T}_k\} \quad \text{and} \quad \{\mathsf{T}_6, \mathsf{T}_{jk6}, \mathsf{T}_{ik6}, \mathsf{T}_{ij6}\} \]
in Lemma~\ref{lem:LinearRelation}, then Equation~(\ref{KummerNormalForm}) is equivalent to
\begin{equation}
 \varphi^2 = 4 \, (\lambda_i-\lambda_k)^2 (\lambda_j-\lambda_k)^2 \, \mathsf{T}_i\, \mathsf{T}_j \, \mathsf{T}_{ik6} \mathsf{T}_{jk6} \;,
\end{equation}
with 
$$
 \varphi= (\lambda_j-\lambda_k)^2 \, \mathsf{T}_i \, \mathsf{T}_{jk6} + (\lambda_i-\lambda_k)^2 \, \mathsf{T}_j \, \mathsf{T}_{ik6}  - (\lambda_i-\lambda_j)^2 \, \mathsf{T}_k \, \mathsf{T}_{ij6} \;.
$$
$\mathsf{T}_k$ and $\mathsf{T}_{ij6}$ are linear functions of the G\"opel tetrahedron $\{\mathsf{T}_i, \mathsf{T}_j, \mathsf{T}_{ik6}, \mathsf{T}_{jk6}\}$ given by
\begin{equation}
\begin{array}{ccc}
 \mathsf{T}_{k} &=& \frac{(\lambda_k-\lambda_m)(\lambda_k-\lambda_l)(\lambda_j-\lambda_k)  \mathsf{T}_{i}}{(\lambda_i-\lambda_j)(\lambda_i\lambda_j-\lambda_i\lambda_k-\lambda_j\lambda_k+\lambda_k\lambda_l)}
 + \frac{(\lambda_k-\lambda_m)(\lambda_k-\lambda_l)(\lambda_i-\lambda_k)  \mathsf{T}_{j}}{(\lambda_j-\lambda_i)(\lambda_i\lambda_j-\lambda_i\lambda_k-\lambda_j\lambda_k+\lambda_k\lambda_l)}\\[0.6em]
 &+& \frac{(\lambda_i-\lambda_k)(\lambda_j-\lambda_k)  \mathsf{T}_{ik6}}{(\lambda_i-\lambda_j)(\lambda_i\lambda_j-\lambda_i\lambda_k-\lambda_j\lambda_k+\lambda_k\lambda_l)}
 + \frac{(\lambda_i-\lambda_k)(\lambda_j-\lambda_k)  \mathsf{T}_{jk6}}{(\lambda_j-\lambda_i)(\lambda_i\lambda_j-\lambda_i\lambda_k-\lambda_j\lambda_k+\lambda_k\lambda_l)},\\[0.6em]
 \mathsf{T}_{ij6} &=&  \frac{(\lambda_j-\lambda_m)(\lambda_j-\lambda_l)(\lambda_j-\lambda_k)(\lambda_i-\lambda_k)  \mathsf{T}_{i}}{(\lambda_j-\lambda_i)(\lambda_i\lambda_j-\lambda_i\lambda_k-\lambda_j\lambda_k+\lambda_k\lambda_l)}
 + \frac{(\lambda_i-\lambda_m)(\lambda_i-\lambda_l)(\lambda_i-\lambda_k)(\lambda_j-\lambda_k)  \mathsf{T}_{j}}{(\lambda_i-\lambda_j)(\lambda_i\lambda_j-\lambda_i\lambda_k-\lambda_j\lambda_k+\lambda_k\lambda_l)}\\[0.6em]
 &+& \frac{(\lambda_j-\lambda_m)(\lambda_j-\lambda_l) (\lambda_i-\lambda_k) \mathsf{T}_{ik6}}{(\lambda_j-\lambda_i)(\lambda_i\lambda_j-\lambda_i\lambda_k-\lambda_j\lambda_k+\lambda_k\lambda_l)}
 + \frac{(\lambda_i-\lambda_m)(\lambda_i-\lambda_l) (\lambda_j-\lambda_k) \mathsf{T}_{jk6}}{(\lambda_i-\lambda_j)(\lambda_i\lambda_j-\lambda_i\lambda_k-\lambda_j\lambda_k+\lambda_k\lambda_l)}.
\end{array} 
\end{equation}
\end{example} 
Kummer~\cite{MR1579281} introduced the notion of formal square roots of Equations~(\ref{AllKummer}) which he called \emph{irrational Kummer normal form}. Irrational Kummer normal forms of a singular Kummer surface also appeared in a slightly different form in \cite{MR1097176}. We make the following:
\begin{remark}
For every G\"opel tetrahedron, four formal square roots giving rise to Equations~(\ref{AllKummer}) are as follows:
\begin{equation}
\label{AllKummer2}
\begin{split}
 \mu \sqrt{\mathsf{T}_b \mathsf{T}_{b'}} + \nu \sqrt{\mathsf{T}_c \mathsf{T}_{c'}} +\rho \sqrt{\mathsf{T}_d \mathsf{T}_{d'}} & = 0,\\
 \rho \sqrt{\mathsf{T}_a \mathsf{T}_{a'}} + \gamma \sqrt{\mathsf{T}_b \mathsf{T}_{b'}} - \beta \sqrt{\mathsf{T}_c \mathsf{T}_{c'}} & = 0,\\ 
 \nu \sqrt{\mathsf{T}_a \mathsf{T}_{a'}} - \delta \sqrt{\mathsf{T}_b \mathsf{T}_{b'}} + \beta \sqrt{\mathsf{T}_d \mathsf{T}_{d'}} & = 0,\\ 
 \mu \sqrt{\mathsf{T}_a \mathsf{T}_{a'}} + \delta \sqrt{\mathsf{T}_c \mathsf{T}_{c'}} - \gamma \sqrt{\mathsf{T}_d \mathsf{T}_{d'}} & = 0.
\end{split}
\end{equation}
A particular compatible choice for the signs of the coefficients in Equations~(\ref{AllKummer2}) has to be made such that two equations are linear combinations of the other two equations using $\mu+\nu+\rho=0$ and $\beta\mu+\gamma\nu+\delta\rho=0$. 
\end{remark}
If we replace the formal expressions $\sqrt{\mathsf{T}_b \mathsf{T}_{b'}}$ in Equations~(\ref{AllKummer}) by the well-defined sections $\mathsf{t}_{b,b'}$ introduced in Equations~(\ref{def:t_ijk}), we obtain the following:
\begin{proposition}
The irrational Kummer normal forms coincide with the Mumford Theta relations in Proposition~\ref{prop:MMF}. In particular, Equations~(\ref{Kummer:topesLRc}) or, equivalently, Equations~(\ref{Eqn:MumfordBimonomial}), generate the same ideal as the following 60 equations given by
\begin{equation}
 \mu \, \mathsf{t}_{b,b'}  + \nu \, \mathsf{t}_{c,c'} +\rho \, \mathsf{t}_{d,d'} = 0,\quad
 \rho \, \mathsf{t}_{a,a'}+ \gamma \, \mathsf{t}_{b,b'}- \beta \, \mathsf{t}_{c,c'}  = 0,
\end{equation}
where $\{a,b,c,d,a',b',c',d'\}$ and $\beta, \gamma, \mu, \nu, \gamma$ run over all cases in Lemma~\ref{lem:quadratics}.
\end{proposition}
\begin{proof}
The proof follows from direct computation using the explicit equations for the tropes in terms of Theta functions in Equations~(\ref{def:t_i}) and Equations~(\ref{def:t_ijk}).
The ideal generated by the irrational Kummer normal forms then coincides with the one generated by the Mumford Theta relations in Proposition~\ref{prop:MMF}.
\end{proof}
\subsection{The G\"opel and G\"opel-Hudson normal forms}
The following definition is due to Kummer~\cite{MR1579281} and Borchardt \cite{MR1579732}:
\begin{definition}
A G\"opel quartic is the surface in $\mathbb{P}^3(P,Q,R,S)$ given by
\begin{equation}
\label{Kummer-Quartic}
    \Phi^2 - 4 \, \delta^2 SPQR  =0\;,
\end{equation}
with 
\begin{equation}
\Phi= P^2 + Q^2 + R^2 +S^2 - \alpha \, \big( PS + QR \big)  - \beta \,  \big( PQ + RS \big) - \gamma \,  \big(PR+QS\big) \;,
\end{equation}
where $\alpha, \beta, \gamma, \delta \in \mathbb{C}$ such that
\begin{equation}
\label{paramG}
 \delta^2 = \alpha^2 + \beta^2 + \gamma^2 + \alpha \beta \gamma -4 \;.
\end{equation}
\end{definition}
We have the following:
\begin{proposition}
\label{lem:KumNF}
Every G\"opel tetrahedron determines an isomorphism between the Cassels-Flynn quartic in Equation~(\ref{kummer}) and a G\"opel quartic in Equation~(\ref{Kummer-Quartic}). In particular, for generic parameters $(\alpha, \beta, \gamma, \delta)$ satisfying Equation~(\ref{paramG}), the G\"opel quartic in Equation~(\ref{Kummer-Quartic}) is isomorphic to the singular Kummer variety associated with a principally polarized abelian variety.
\end{proposition}
\begin{proof}
First, one realizes every G\"opel tetrahedron $\{\mathsf{T}_a, \mathsf{T}_{a'}, \mathsf{T}_b, \mathsf{T}_{b'}\}$ using two disjoint Rosenhain tetrahedra given in Lemma~\ref{lem:quadratics}. Using Proposition~\ref{prop:KummerNF}, the Cassels-Flynn quartic in Equation~(\ref{kummer}) is identical to the equation
\begin{equation}
\label{eqn:test}
\left(\rho^2 \mathsf{T}_a \mathsf{T}_{a'} + \gamma^2\mathsf{T}_b \mathsf{T}_{b'} -\beta^2 \mathsf{T}_c \mathsf{T}_{c'}\right)^2  - 4 \gamma^2\rho^2\mathsf{T}_a \mathsf{T}_{a'}\mathsf{T}_b \mathsf{T}_{b'} =0,
\end{equation}
with $\beta, \gamma, \rho \in \mathbb{C}[\lambda_1,\lambda_2,\lambda_3]$. Using Corollary~\ref{cor:LR}, the two additional tropes $\mathsf{T}_c, \mathsf{T}_{c'}$ are linear functions of $\{\mathsf{T}_a, \mathsf{T}_{a'}, \mathsf{T}_b, \mathsf{T}_{b'}\}$ with coefficients in $\mathbb{C}(\lambda_1,\lambda_2,\lambda_3)$. One substitutes
\[
   [\mathsf{T}_a, \mathsf{T}_{a'}, \mathsf{T}_b, \mathsf{T}_{b'}]=[\, c_1 P\, : \, c_2 Q \, : \, c_3 R \, : \, c_4 S\, ],
\]
into Equation~(\ref{eqn:test}) and solves for the coefficients $c_1, \dots, c_4$ such that Equation~(\ref{eqn:test}) coincides with Equation~(\ref{Kummer-Quartic}). 
\end{proof}
We give another description of a singular Kummer variety associated with a principally polarized abelian variety based on results in \cite{MR2062673}. Let $\mathcal{L}$ be the ample line symmetric bundle of an abelian surface $\mathbf{A}$ defining its  principal polarization and consider the rational map $\varphi_{\mathcal{L}^2}: \mathbf{A} \to \mathbb{P}^3$ associated with the line bundle $\mathcal{L}^2$. Its image $\varphi_{\mathcal{L}^2}(\mathbf{A})$  is a quartic surface in $\mathbb{P}^3$ which in projective coordinates $[w:x:y:z]$ can be written as
\begin{gather}
\label{Eq:QuarticSurfaces12}
    0 = \xi_0 \, (w^4+x^4+y^4+z^4)  + \xi_4 \, w x y z \qquad \\
\nonumber
    +\xi_1 \, \big(w^2 z^2+x^2 y^2\big)  +\xi_2 \, \big(w^2 x^2+y^2 z^2\big)  +\xi_3 \, \big(w^2 y^2+x^2 z^2\big) ,
\end{gather}
with $[\xi_0:\xi_1:\xi_2:\xi_3:\xi_4] \in \mathbb{P}^4$. Any general member of the family~(\ref{Eq:QuarticSurfaces12}) is smooth. As soon as the surface is singular at a general point, it must have sixteen singular nodal points because of its symmetry. The discriminant turns out to be a homogeneous polynomial of degree eighteen in the parameters $[\xi_0:\xi_1:\xi_2:\xi_3:\xi_4] \in \mathbb{P}^4$ and was determined in \cite[Sec.~7.7 (3)]{MR2062673}.   Thus, the Kummer surfaces form an open set among the hypersurfaces in Equation~(\ref{Eq:QuarticSurfaces12}) with parameters $[\xi_0:\xi_1:\xi_2:\xi_3:\xi_4] \in \mathbb{P}^4$, namely the ones that make the only irreducible factor of degree three in the discriminant vanish, i.e.,
\begin{equation}
 \xi_0 \, \big( 16 \xi_0^2 - 4 \xi_1^2-4 \xi_2^2 - 4 \xi_3^3+ \xi_4^2\big) + 4 \, \xi_1 \xi_2 \xi_3 =0.
\end{equation}
Setting $\xi_0=1$ and using the affine moduli $\xi_1=-A$, $\xi_2=-B$, $\xi_3=-C$, $\xi_4=2 D$, we obtain the normal form of a nodal quartic surface. We give the following:
\begin{definition}
A G\"opel-Hudson quartic is the surface in $\mathbb{P}^3(w,x,y,z)$ given by
\begin{gather}
\label{Goepel-Quartic}
  0 = w^4+x^4+y^4+z^4  + 2  D  w x y z \\
\nonumber
    - A  \big(w^2 z^2+x^2 y^2\big)  - B \big(w^2 x^2+y^2 z^2\big)  - C  \big(w^2 y^2+x^2 z^2\big)  \;, 
\end{gather}
where $A, B, C, D \in \mathbb{C}$ such that
\begin{equation}
\label{paramGH}
 D^2 = A^2 + B^2 + C^2 + ABC - 4\;.
\end{equation}
\end{definition}
\begin{remark}
The G\"opel-Hudson quartic in Equation~(\ref{Goepel-Quartic}) is invariant under the transformations changing signs of two coordinates, generated by 
\[
 [w:x:y:z] \to [-w:-x:y:z], \quad [w:x:y:z] \to [-w:x:-y:z],
\] 
and under the permutations of variables, generated by $[w:x:y:z] \to [x:w:z:y]$ and $[w:x:y:z] \to [y:z:w:x]$. 
\end{remark}
We have the following:
\begin{lemma}
\label{lem:GH}
The G\"opel-Hudson quartic in Equation~(\ref{Goepel-Quartic}) is isomorphic to the G\"opel quartic in Equation~(\ref{Kummer-Quartic}). In particular,  for generic parameters $(A,B,C,D)$ satisfying Equation~(\ref{paramGH}), the G\"opel quartic in Equation~(\ref{Kummer-Quartic}) is isomorphic to the singular Kummer variety associated with a principally polarized abelian surface.
\end{lemma}
\begin{proof}
We first introduce complex numbers $[w_0:x_0:y_0:z_0] \in \mathbb{P}^3$ such that
\begin{gather}
\nonumber
  A   = \frac{w_0^4 - x_0^4-y_0^4+z_0^4}{w_0^2 z_0^2 - x_0^2 y_0^2}, \quad
  B  = \frac{w_0^4 + x_0^4-y_0^4-z_0^4}{w_0^2 x_0^2 - y_0^2 z_0^2}, \quad
  C  = \frac{w_0^4 - x_0^4+y_0^4-z_0^4}{w_0^2 y_0^2 - x_0^2 z_0^2}, \\
\label{Eq:paramABCD}
   D  = \frac{w_0 x_0 y_0 z_0 \prod_{\epsilon, \epsilon' \in \lbrace \pm1 \rbrace} (w_0^2 + \epsilon x_0^2+ \epsilon' y_0^2+ \epsilon \epsilon' z_0^2)}{(w_0^2 z_0^2 - x_0^2 y_0^2)
   (w_0^2 x_0^2 - y_0^2 z_0^2)(w_0^2 y_0^2 - x_0^2 z_0^2)}, 
 \end{gather}
and
\begin{gather}
\nonumber
  \alpha  = 2 \, \frac{w_0^2 \, y_0^2 + x_0^2 \, z_0^2}{w_0^2 \, y_0^2 - x_0^2 \, z_0^2}, \quad
  \beta  = 2 \, \frac{w_0^2 \, x_0^2 + y_0^2 \, z_0^2}{w_0^2 \, x_0^2 - y_0^2 \, z_0^2}, \quad
  \gamma  =  2\, \frac{w_0^2 \, z_0^2 + x_0^2 \, y_0^2}{w_0^2 \, z_0^2 - x_0^2 \, y_0^2},\\
\label{KummerParameter_c}
  \delta^2 = 16 \, \frac{w_0^4 x_0^4 y_0^4 z_0^4 \, \prod_{\epsilon, \epsilon' \in \lbrace \pm1 \rbrace} (w_0^2 + \epsilon x_0^2+ \epsilon' y_0^2+ \epsilon \epsilon' z_0^2)}{\big(w_0^2 \, y_0^2 - x_0^2 \, z_0^2\big)^2 \, \big(w_0^2 \, x_0^2 - y_0^2 \, z_0^2\big)^2 \, \big(w_0^2 \, z_0^2 - x_0^2 \, y_0^2\big)^2}.
  \end{gather}
Then, the linear transformation given by
\begin{equation}
\label{eq:LinearTransfo}
\begin{split}
 P & =  w_0 \, w + x_0 \, x + y_0 \, y+ z_0 \, z \;,\\
 Q & =  w_0 \, w + x_0 \, x - y_0 \, y- z_0 \, z \;,\\
 R & =  w_0 \, w  - x_0 \, x - y_0 \, y+ z_0 \, z \;,\\
 S & =  w_0 \, w - x_0 \, x + y_0 \, y- z_0 \, z \;,
\end{split}
\end{equation} 
transforms Equation~(\ref{Kummer-Quartic}) into Equation~(\ref{Goepel-Quartic}). 
\end{proof}
\begin{lemma}
\label{lem:16nodes}
Using the same notation as in Equation~(\ref{Eq:paramABCD}), the sixteen nodes of the G\"opel-Hudson quartic~(\ref{Goepel-Quartic}) are the points $[w:x:y:z]$ given by
\begin{gather*}
[w_0:x_0:y_0:z_0], [-w_0:-x_0:y_0:z_0], [-w_0:x_0:-y_0:z_0], [-w_0:x_0:y_0:-z_0],\\
[x_0:w_0:z_0:y_0], [-x_0:-w_0:z_0:y_0], [-x_0:w_0:-z_0:y_0], [-x_0:w_0:z_0:-y_0],\\
[y_0:z_0:w_0:x_0], [-y_0:-z_0:w_0:x_0], [-y_0:z_0:-w_0:x_0], [-y_0:z_0:w_0:-x_0],\\
[z_0:y_0:x_0:w_0], [-z_0:-y_0:x_0:w_0], [-z_0:y_0:-x_0:w_0], [-z_0:y_0:x_0:-w_0].
\end{gather*}
In particular, for generic parameters $(A,B,C,D)$, no node is contained in the coordinate planes $w=0$, $x=0$, $y=0$, or $z=0$.
\end{lemma}
We also set
\begin{gather}
\label{KummerParameter_thetas}
p_0^2 = w_0^2 + x_0^2+ y_0^2+z_0^2, \quad q_0^2 = w_0^2 + x_0^2- y_0^2-z_0^2,\\
\nonumber
r_0^2 = w_0^2 - x_0^2+ y_0^2-z_0^2, \quad s_0^2 = w_0^2 - x_0^2- y_0^2+z_0^2.
\end{gather}
We have the following:
\begin{lemma}
\label{lem:16nodesG}
The nodes of the G\"opel quartic~(\ref{Kummer-Quartic}) are the points $[P:Q:R:S]$ given by
\begin{gather*}
[p_0^2:q_0^2:r_0^2:s_0^2], [q_0^2:p_0^2:s_0^2:r_0^2], [r_0^2:s_0^2:r_0^2:q_0^2], [s_0^2:r_0^2:q_0^2:p_0^2], \\
[w_0x_0+y_0z_0	:w_0x_0-y_0z_0	:0				:0],				[w_0x_0-y_0z_0	:w_0x_0+y_0z_0	:0				:0], \\
[0				:0				:w_0x_0+y_0z_0	:w_0x_0-y_0z_0], 	[0				:0				:w_0x_0-y_0z_0	:w_0x_0+y_0z_0],\\
[w_0z_0+x_0y_0	:0				:w_0z_0-x_0y_0	:0],				[0				:w_0z_0+x_0y_0	:0				:w_0z_0-x_0y_0], \\
[w_0z_0-x_0y_0	:0				:w_0z_0+x_0y_0	:0],				[0				:w_0z_0-x_0y_0	:0				:w_0z_0+x_0y_0], \\
[w_0y_0+x_0z_0	:0				:0				:w_0y_0-x_0z_0],	[0				:w_0y_0+x_0z_0	:w_0y_0-x_0z_0	:0], \\
[0				:w_0y_0-x_0z_0	:w_0y_0+x_0z_0	:0],				[w_0y_0-x_0z_0	:0				:0				:w_0y_0+x_0z_0].
\end{gather*}
\end{lemma}
\subsection{$(2,2)$-polarized Kummer surfaces}
We use the G\"opel and G\"opel-Hudson quartics to provide an explicit model for $(2,2)$-isogenous Kummer surfaces.
\begin{remark}
The transformation $[w:x:y:z] \to [-w:x:y:z]$ is an isomorphism between the G\"opel-Hudson quartic with moduli parameters $(A, B, C, D)$ and the one with parameters $(A, B, C, -D)$. Moreover, the two quartics coincide exactly along the coordinate planes $w=0$, $x=0$, $y=0$, or $z=0$.
\end{remark}
\par The map $\pi: \mathbb{P}(w,x,y,z) \to \mathbb{P}(P,Q,R,S)$ with $P=w^2, \dots, S=z^2$ is $8:1$ outside the coordinate planes. The map $\pi$ induces a covering of a reducible octic surface in $\mathbb{P}^3$ given by
\begin{equation}
\begin{split}
\left(\Psi - 2 D wxyz\right) \, \left(\Psi + 2 D wxyz\right) = 0 \;,
 \end{split} 
\end{equation}
with
\begin{gather*}
\Psi = w^4 + x^4 + y^4 + z^4 - A \big( w^2 y^2 + y^2 z^2 \big) 
- B\big( w^2z^2 + x^2y^2 \big) - C  \big(w^2x^2+y^2z^2\big),
\end{gather*}
onto the G\"opel quartic in Equation~(\ref{Kummer-Quartic}) with $\alpha=A$, $\beta=B$, $C=\gamma$, and $D=\delta$. Because of Lemma~\ref{lem:GH} and Lemma~\ref{lem:KumNF}, we can assume that the G\"opel-Hudson quartic and G\"opel quartic are the singular Kummer varieties associated with two principally polarized abelian varieties, say $\mathbf{A}$ and $\mathbf{A}'$, respectively. We have the following:
\begin{proposition}
\label{prop:RatMap}
The map $\pi: \mathbb{P}(w,x,y,z) \to \mathbb{P}(P,Q,R,S)$ with $P=w^2, \dots, S=z^2$ restricted to the G\"opel-Hudson quartic onto the G\"opel quartic with $\alpha=A$, $\beta=B$, $C=\gamma$, and $D=\delta$, is induced by a $(2,2)$-isogeny $\psi: \mathbf{A} \to \mathbf{A}'\cong \hat{\mathbf{A}}$.
\end{proposition}
\begin{proof}
The rational map $\pi$ is 4:1 from the G\"opel-Hudson quartic onto the G\"opel quartic. The map $\pi: \mathbb{P}(w,x,y,z) \to \mathbb{P}(P,Q,R,S)$ maps the sixteen nodes on the G\"opel-Hudson quartic in Lemma~\ref{lem:16nodes} to four nodes on the G\"opel quartic. It follows from \cite{MR2804549} that the map $\pi$ between Kummer varieties associated with two principally polarized abelian varieties $\mathbf{A}$ and $\mathbf{A}'$ is induced by an isogeny $p: \mathbf{A}' \to \mathbf{A}$ of degree four. The  kernel of $p$ is a two-dimensional isotropic subspace $\mathsf{K}$ of $\mathbf{A}[2]$ such that $\mathbf{A}'=\mathbf{A}/\mathsf{K}$.
\end{proof}
\subsubsection{An explicit model using Theta functions}
\label{computation1}
We provide an explicit model for the G\"opel quartic in terms of Theta functions.  We have the following:
\begin{proposition}
\label{prop:ThetaImageG}
The surface in $\mathbb{P}^3$ given by Equation~(\ref{Kummer-Quartic}) is isomorphic to the singular Kummer variety $\mathbf{A}_{\tau}/\langle -\mathbb{I} \rangle$ where the coordinates are given by
\begin{equation} 
\label{KummerVariables}
 [P:Q:R:S]=\left[\theta_1(z)^2:\theta_2(z)^2:\theta_3(z)^2:\theta_4(z)^2\right] \;,
\end{equation}
and the moduli parameters are
\begin{gather}
\nonumber
 \alpha =  \frac{\theta_1^4-\theta_2^4-\theta_3^4+\theta_4^4}{\theta_1^2\theta_4^2-\theta_2^2\theta_3^2}, \quad
 \beta  =   \frac{\theta_1^4+\theta_2^4-\theta_3^4-\theta_4^4}{\theta_1^2\theta_2^2-\theta_3^2\theta_4^2}, \quad
 \gamma  = \frac{\theta_1^4-\theta_2^4+\theta_3^4-\theta_4^4}{\theta_1^2\theta_3^2-\theta_2^2\theta_4^2},\\
\label{KummerParameter}
 \delta = \frac{ \theta_1 \theta_2 \theta_3 \theta_4 \prod_{\epsilon, \epsilon' \in \lbrace \pm1 \rbrace} (\theta_1^2 + \epsilon \theta_2^2+ \epsilon' \theta_3^2+ \epsilon \epsilon' \theta_3^2)}{\big(\theta_1^2 \, \theta_2^2 - \theta_3^2 \, \theta_4^2\big) \, \big(\theta_1^2 \, \theta_3^2 - \theta_2^2 \, \theta_4^2\big) \, \big(\theta_1^2 \, \theta_4^2 - \theta_2^2 \, \theta_3^2\big)}.
\end{gather}
\end{proposition}
\begin{proof}
One checks that the Rosenhain tetrahedron $\{\mathsf{T}_{256}, \mathsf{T}_{136}, \mathsf{T}_{356},\mathsf{T}_{126}\}$, with tropes given by the Theta function $\theta_1(z)^2,\theta_2(z)^2,\theta_3(z)^2,\theta_4(z)^2$ is contained in three sets of eight tropes corresponding to the even eights $\Delta_{15}$, $\Delta_{46}$, $\Delta_{23}$ in Proposition~\ref{prop:EvenEights}. One can choose any of these sets of eight tropes to proceed. For example, n Lemma~\ref{lem:LinearRelation} we can choose the Rosenhain tetrahedra $\{ \mathsf{T}_{256}, \mathsf{T}_{126}, \mathsf{T}_{236},\mathsf{T}_{4}\}$ and $\{\mathsf{T}_{136}, \mathsf{T}_{356}, \mathsf{T}_{156},\mathsf{T}_{6}\}$. Then, Equation~(\ref{KummerNormalForm}) is equivalent to
\begin{equation}
\label{eqn:test2}
 \varphi^2 - 4 \, (\lambda_2-\lambda_1)^2 (\lambda_3-\lambda_1)^2 (\lambda_2-1)^2 (\lambda_3-1)^2   \mathsf{T}_{126} \mathsf{T}_{136}\mathsf{T}_{256}  \mathsf{T}_{356} \;,
\end{equation}
with 
\begin{gather*}
 \varphi = (\lambda_3-\lambda_1)^2(\lambda_2-1)^2 \mathsf{T}_{136} \mathsf{T}_{256} + (\lambda_2-\lambda_1)^2(\lambda_3-\lambda_1)^2  \mathsf{T}_{126}  \mathsf{T}_{236} \\
 - (\lambda_3-\lambda_2)^2 (\lambda_1-1)^2 \mathsf{T}_{156} \mathsf{T}_{236}.
\end{gather*}
In Proposition~\ref{prop:KummerNF} we then choose the coordinates determined by the G\"opel tetrahedron $\{\mathsf{T}_{256}, \mathsf{T}_{136}, \mathsf{T}_{356}, \mathsf{T}_{126}\}$. We choose coefficients $c_1, \dots, c_4$ with
\[
   [\mathsf{T}_{256}, \mathsf{T}_{136}, \mathsf{T}_{356}, \mathsf{T}_{126}]=[\, c_1 P\, : \, c_2 Q \, : \, c_3 R \, : \, c_4 S\, ],
\]
such that Equation~(\ref{eqn:test2}) coincides with Equation~(\ref{Kummer-Quartic}).  It turns out that the coefficients $c_1, \dots, c_4$ are rational functions in the squares of (even) Theta constants $\theta_1^2, \dots, \theta_{10}^2$, and the coordinates are then given by Theta functions with non-vanishing elliptic variables $[P:Q:R:S]=[\theta_1(z)^2:\theta_2(z)^2:\theta_3(z)^2:\theta_4(z)^2]$. In particular, setting
\begin{equation}
\label{Transfo_Irr2Kummer}
 \big\lbrack \mathsf{T}_{256}, \mathsf{T}_{136}, \mathsf{T}_{356}, \mathsf{T}_{126} \big\rbrack 
 =  \big\lbrack  \theta_2^2 \theta_3^2 \theta_4^2 \,P :  \theta_1^2 \theta_3^2 \theta_4^2 \,Q :  \theta_1^2 \theta_2^2 \theta_4^2 \, R:  \theta_1^2 \theta_2^2 \theta_3^2 \, S\big\rbrack
\end{equation}
changes Equation~(\ref{QuadraticEquation}) into Equation~(\ref{Kummer-Quartic}). Using Equations~(\ref{Eq:degree2doublingR}), the transformation~(\ref{Transfo_Irr2Kummer}) descends to $\mathcal{A}_2(2,4)$ with
\begin{equation}
\begin{split}
\label{KummerParameter_B}
 \alpha =  2\, \frac{\Theta_1^2 \Theta_3^2+\Theta_2^2\Theta_4^2}{\Theta_1^2 \Theta_3^2-\Theta_2^2\Theta_4^2}, \quad
 \beta  =   2\, \frac{\Theta_1^2 \Theta_2^2+\Theta_3^2\Theta_4^2}{\Theta_1^2 \Theta_2^2-\Theta_3^2\Theta_4^2}, \quad
 \gamma  =   2\, \frac{\Theta_1^2 \Theta_4^2+\Theta_2^2\Theta_3^2}{\Theta_1^2 \Theta_4^2-\Theta_2^2\Theta_3^2}.
\end{split}
\end{equation}
Using Equations~(\ref{Eq:degree2doubling}) the moduli parameter can be re-written 
\begin{equation}
\label{KummerParameter_b}
\begin{split}
 \alpha =  \frac{\theta_1^4-\theta_2^4-\theta_3^4+\theta_4^4}{\theta_1^2\theta_4^2-\theta_2^2\theta_3^2}, \quad
 \beta  =   \frac{\theta_1^4+\theta_2^4-\theta_3^4-\theta_4^4}{\theta_1^2\theta_2^2-\theta_3^2\theta_4^2}, \quad
 \gamma  = \frac{\theta_1^4-\theta_2^4+\theta_3^4-\theta_4^4}{\theta_1^2\theta_3^2-\theta_2^2\theta_4^2}.
\end{split}
\end{equation}
such that over $\mathcal{A}_2(4,8)$ the modulus $\delta$ is given as
\begin{equation*}
  \delta = \frac{ \theta_1 \theta_2 \theta_3 \theta_4 \prod_{\epsilon, \epsilon' \in \lbrace \pm1 \rbrace} (\theta_1^2 + \epsilon \theta_2^2+ \epsilon' \theta_3^2+ \epsilon \epsilon' \theta_3^2)}{\big(\theta_1^2 \, \theta_2^2 - \theta_3^2 \, \theta_4^2\big) \, \big(\theta_1^2 \, \theta_3^2 - \theta_2^2 \, \theta_4^2\big) \, \big(\theta_1^2 \, \theta_4^2 - \theta_2^2 \, \theta_3^2\big)}.
\end{equation*}
\end{proof}
The Rosenhain roots of the $(2,2)$-isogenous genus-two curve $\hat{\mathcal{C}}$ in Equation~(\ref{Eq:Rosenhain2}) can be considered coordinates of an isogenous moduli space that we denote by $\hat{\mathcal{A}}_2(2)$. We have the following:
\begin{lemma}
\label{lem:ModuliG}
The moduli $\alpha, \beta, \gamma, \delta$ are rational functions over $\hat{\mathcal{A}}_2(2)$.
\end{lemma}
\begin{proof}
Equations~(\ref{KummerParameter}) descend to rational functions on $\hat{\mathcal{A}}_2(2)$ given by
\begin{gather}
\nonumber
 \alpha =  2 \, \frac{\Lambda_1+1}{\Lambda_1-1},\;
 \beta =  2 \, \frac{\Lambda_1\Lambda_2+\Lambda_1\Lambda_3-2\Lambda_2\Lambda_3-2\Lambda_1+\Lambda_2+\Lambda_3}{(\Lambda_2-\Lambda_3)(\Lambda_1-1)}, \;
 \gamma = 2 \, \frac{\Lambda_3+\Lambda_2}{\Lambda_3-\Lambda_2}, \\
 \label{KummerParameter2}
\delta = \frac{4(\Lambda_1-\Lambda_2\Lambda_3)}{(\Lambda_1-1)(\Lambda_3-\Lambda_2)}.
\end{gather}
\end{proof}
\par We also provide an explicit model for the G\"opel-Hudson quartic in terms of Theta functions.  In terms of Theta functions, the relation between Equation~(\ref{Goepel-Quartic}) and Equation~(\ref{Kummer-Quartic}) was first determined by Borchardt \cite{MR1579732}.  We have the following:
\begin{proposition}
\label{prop:ThetaImageGH}
The surface in $\mathbb{P}^3$ given by Equation~(\ref{Goepel-Quartic}) is isomorphic to the singular Kummer variety $\mathbf{A}_{\tau}/\langle -\mathbb{I} \rangle$ where the coordinates are given by
\begin{equation} 
 [w:x:y:z]=[ \Theta_1(2  z): \Theta_2(2  z): \Theta_3(2  z) : \Theta_4(2  z) ] \;,
\end{equation}
and the moduli parameters are
\begin{gather}
  \nonumber
  A   = \frac{\Theta_1^4 - \Theta_2^4-\Theta_3^4+\Theta_4^4}{\Theta_1^2 \Theta_4^2 - \Theta_2^2 \Theta_3^2}, \quad
  B  = \frac{\Theta_1^4 + \Theta_2^4-\Theta_3^4-\Theta_4^4}{\Theta_1^2 \Theta_2^2 - \Theta_3^2 \Theta_4^2}, \quad
  C  = \frac{\Theta_1^4 - \Theta_2^4+\Theta_3^4-\Theta_4^4}{\Theta_1^2 \Theta_3^2 - \Theta_2^2 \Theta_4^2}, \\
  \label{KummerParameter3}
  D  = \frac{\Theta_1 \Theta_2 \Theta_3 \Theta_4 \prod_{\epsilon, \epsilon' \in \lbrace \pm1 \rbrace} (\Theta_1^2 + \epsilon \Theta_2^2+ \epsilon' \Theta_3^2+ \epsilon \epsilon' \Theta_4^2)}
  {(\Theta_1^2 \Theta_2^2 - \Theta_3^2 \Theta_4^2)(\Theta_1^2 \Theta_3^2 - \Theta_2^2 \Theta_4^2)(\Theta_1^2 \Theta_4^2 - \Theta_2^2 \Theta_3^2)}.
\end{gather}
\end{proposition}
\begin{proof}
Comparing Equations~(\ref{KummerParameter_B}) and (\ref{KummerParameter_c}), we find a solution in terms of Theta function given by
\begin{equation}
 [w_0:x_0:y_0:z_0]=[\Theta_1:\Theta_2:\Theta_3:\Theta_4] \;.
 \end{equation}
Equations~(\ref{Eq:Degree2doubling_z}) are equivalent to
\begin{equation}
 \label{2isogTHETA}
\begin{split}
  \theta_1^2(z)  & = \Theta_1 \, \Theta_1(2z) + \Theta_2 \, \Theta_2(2z) + \Theta_3 \, \Theta_3(2z) + \Theta_4 \, \Theta_4(2z) ,\\
 \theta_2^2(z)  & = \Theta_1 \, \Theta_1(2z) + \Theta_2 \, \Theta_2(2z) - \Theta_3 \, \Theta_3(2z) - \Theta_4 \, \Theta_4(2z) ,\\
  \theta_3^2(z)  & =  \Theta_1 \, \Theta_1(2z) - \Theta_2 \, \Theta_2(2z) - \Theta_3 \, \Theta_3(2z) + \Theta_4 \, \Theta_4(2z) ,\\
   \theta_4^2(z)  & =  \Theta_1 \, \Theta_1(2z) - \Theta_2 \, \Theta_2(2z) + \Theta_3 \, \Theta_3(2z) - \Theta_4 \, \Theta_4(2z) .
  \end{split} 
\end{equation}
Comparing Equations~(\ref{eq:LinearTransfo}) with Equations~(\ref{2isogTHETA}), the coordinates can be expressed in terms of Theta functions with non-vanishing elliptic arguments as
$$[w:x:y:z]=[ \Theta_1(2  z): \Theta_2(2  z): \Theta_3(2  z) : \Theta_4(2  z) ] .$$
\end{proof}
It follows:
\begin{lemma}
\label{lem:ModuliGH}
The moduli parameter $A, B, C, D$ are rational functions over $\mathcal{A}_2(2)$.
\end{lemma}
\begin{proof}
Equations~(\ref{KummerParameter3}) descend to rational functions on $\mathcal{A}_2(2)$ given by
\begin{gather}
\nonumber
 A =  2 \, \frac{\lambda_1+1}{\lambda_1-1}, \quad
 B =  2 \, \frac{\lambda_1\lambda_2+\lambda_1\lambda_3-2\lambda_2\lambda_3-2\lambda_1+\lambda_2+\lambda_3}{(\lambda_2-\lambda_3)(\lambda_1-1)},  \quad
 C  = 2 \, \frac{\lambda_3+\lambda_2}{\lambda_3-\lambda_2}, \\ 
 \label{KummerParameter4}
 D = 4 \, \frac{\lambda_1-\lambda_2 \lambda_3}{(\lambda_2 - \lambda_3) (\lambda_1-1)} \;.
\end{gather}
\end{proof}
We have proved the following:
\begin{theorem}
\begin{enumerate}
\item[]
\item The singular Kummer variety associated with the Jacobian $\operatorname{Jac}(\mathcal{C})$ of the genus-two curve~$\mathcal{C}$ in Rosenhain normal form given by Equation~(\ref{Eq:Rosenhain}) is the image of $[\Theta_{1}(2z): \Theta_{2}(2z): \Theta_{3}(2z): \Theta_{4}(2z)]$ in $\mathbb{P}^3$ which satisfies Equation~(\ref{Goepel-Quartic}) with moduli parameters given by Equations~(\ref{KummerParameter4}). 
\item The singular Kummer variety associated with the Jacobian $\operatorname{Jac}(\hat{\mathcal{C}})$ of the $(2,2)$-Isogenous curve~$\hat{\mathcal{C}}$ (cf.~Lemma~\ref{lem:R2isog}) in Rosenhain normal form~(\ref{Eq:Rosenhain2}) is the image of $[\theta^2_{1}(z): \theta^2_{2}(z): \theta^2_{3}(z): \theta^2_{4}(z)]$ in $\mathbb{P}^3$ which satisfies Equation~(\ref{Kummer-Quartic}) with moduli parameters given by Equations~(\ref{KummerParameter2}).
\end{enumerate}
\end{theorem}
\begin{proof}
The proof follows from Propositions~\ref{prop:RatMap}, \ref{prop:ThetaImageG}, \ref{prop:ThetaImageGH} and Lemmas~\ref{lem:ModuliG}, \ref{lem:ModuliGH}.
\end{proof}
\subsection{The Rosenhain normal form}
In~\cite[Prop.~\!10.3.2]{MR2062673} another normal form for a nodal quartic surfaces was established. We make the following definition:
\begin{definition}
\label{def:BL}
A Rosenhain quartic is the surface in $\mathbb{P}^3(Y_0, \dots, Y_3)$ given by
\begin{equation}
\label{Birkenhake-Lange-Quartic}
\begin{split}
a^2  \big(Y_0^2 Y_1^2 + Y_2^2 Y_3^2\big) + b^2 \big(Y_0^2 Y_2^2 + Y_1^2 Y_3^2\big) + c^2  \big(Y_0^2 Y_3^2 + Y_1^2 Y_2^2\big)\\
+ 2ab  \big(Y_0 Y_1 - Y_2 Y_3\big) \big(Y_0 Y_2 + Y_1 Y_3\big)  - 2ac \big(Y_0 Y_1 + Y_2 Y_3\big)  \big(Y_0 Y_3 + Y_1 Y_2\big) \\
+ 2 b c \big(Y_0 Y_2 - Y_1 Y_3\big)  \big(Y_0 Y_3 - Y_1 Y_2\big) + d^2  Y_0Y_1Y_2Y_3 =0 ,
\end{split}
\end{equation}
with $[a:b:c:d] \in \mathbb{P}^3$.
\end{definition}
The following was proved in~\cite[Prop.~10.3.2]{MR2062673}:
\begin{proposition}
\label{lem:BL}
For generic parameters $[a:b:c:d] \in \mathbb{P}^3$, the Rosenhain quartic in Equation~(\ref{Birkenhake-Lange-Quartic}) is isomorphic to the singular Kummer variety associated with a principally polarized abelian variety.
\end{proposition}
We also have the following:
\begin{lemma}
The Rosenhain quartic in Equation~(\ref{Birkenhake-Lange-Quartic}) is isomorphic to the G\"opel-Hudson quartic in Equation~(\ref{Goepel-Quartic}). 
\end{lemma}
\begin{proof}
Using the same notation as in the proof of Lemma~\ref{lem:KumNF}, we set
\begin{equation}
 \begin{split}
 a &= 4 (w_0^2 z_0^2-x_0^2 y_0^2)(w_0 y_0+x_0 z_0),\\
 b &=(w_0^2+x_0^2-y_0^2-z_0^2)(w_0^2-x_0^2+y_0^2-z_0^2)(w_0x_0-y_0z_0) ,\\
 c & =(w_0^2-x_0^2-y_0^2+z_0^2)(w_0^2+x_0^2+y_0^2+z_0^2)(w_0x_0+y_0z_0) ,
 \end{split}
 \end{equation}
and a polynomial expression for $d^2$ of degree twelve that we do not write out explicitly. Then, the linear transformation given by
\begin{equation}
\label{eq:LinearTransfo2}
\begin{split}
 Y_0 & =  w_0 \, w + x_0 \, x + y_0 \, y+ z_0 \, z \;,\\
 Y_1 & =  w_0 \, w  + x_0 \, x - y_0 \, y- z_0 \, z \;,\\
 Y_2 & =  z_0 \, w+ y_0 \, x + x_0 \, y+ w_0 \, z \;,\\
 Y_3 & =  z_0 \, w + y_0 \, x - x_0 \, y- w_0 \, z \;,
\end{split}
\end{equation} 
transforms Equation~(\ref{Birkenhake-Lange-Quartic}) into Equation~(\ref{Goepel-Quartic}). 
\end{proof}
We have the stronger result:
\begin{proposition}
\label{lem:KumRNF}
Every Rosenhain tetrahedron determines an isomorphism between the Cassels-Flynn quartic in Equation~(\ref{kummer}) and the Rosenhain quartic in Equation~(\ref{Birkenhake-Lange-Quartic}).
\end{proposition}
\begin{proof}
Given a Rosenhain tetrahedron $\{\mathsf{T}_a, \mathsf{T}_b, \mathsf{T}_c, \mathsf{T}_d\}$, we use Proposition~\ref{prop:KummerNF} to write the Cassels-Flynn quartic in Equation~(\ref{kummer}) in the equivalent form
\begin{equation}
\label{eqn:test3}
\left(\rho^2 \mathsf{T}_a \mathsf{T}_{a'} + \gamma^2\mathsf{T}_b \mathsf{T}_{b'} -\beta^2 \mathsf{T}_c \mathsf{T}_{c'}\right)^2  - 4 \gamma^2\rho^2\mathsf{T}_a \mathsf{T}_{a'}\mathsf{T}_b \mathsf{T}_{b'} =0,
\end{equation}
with $\beta, \gamma, \rho \in \mathbb{C}[\lambda_1,\lambda_2,\lambda_3]$. The two additional tropes $\mathsf{T}_c \mathsf{T}_{c'}$ are linear functions of $\mathsf{T}_a, \mathsf{T}_{a'}, \mathsf{T}_b, \mathsf{T}_{b'}$ with coefficients in $\mathbb{C}(\lambda_1,\lambda_2,\lambda_3)$. One substitutes
\[
   [\mathsf{T}_a, \mathsf{T}_b, \mathsf{T}_c, \mathsf{T}_d]=[\, c_0 Y_0\, : \, c_1 Y_1 \, : \, c_2 Y_2 \, : \, c_3 Y_3\, ],
\]
into Equation~(\ref{eqn:test3}) and solves for the coefficients $c_0, \dots, c_3$ such that Equation~(\ref{Birkenhake-Lange-Quartic}) coincides with Equation~(\ref{Kummer-Quartic}). 
\end{proof}
\begin{remark}
Equation~(\ref{Birkenhake-Lange-Quartic}) is unchanged by the Cremona transformation 
\[
[ Y_0: Y_1: Y_2: Y_3] \mapsto [Y_1Y_2 Y_3: Y_0Y_2Y_3: Y_0Y_1Y_3: Y_0Y_1Y_2] \;.
\]
\end{remark}
We also have the following:
\begin{theorem}
\label{prop:ThetaImageR}
The surface in $\mathbb{P}^3$ given by Equation~(\ref{Birkenhake-Lange-Quartic}) is isomorphic to the singular Kummer variety $\mathbf{A}_{\tau}/\langle -\mathbb{I} \rangle$ where the coordinates are given by
\begin{equation} 
\label{KummerVariablesR}
 [Y_0:Y_1:Y_2:Y_3]=\left[\theta_1(z)^2:\theta_2(z)^2:\theta_7(z)^2:\theta_{12}(z)^2\right] \;,
\end{equation}
and the moduli parameters are
\begin{gather*}
a= \left( 2\,\Theta_1\Theta_4-2\,\Theta_2\Theta_3  \right)  \left( 2\,\Theta_1\Theta_4+2\,\Theta_2\Theta_3 \right)  \left( 2\,\Theta_1\Theta_3+2\,\Theta_2\Theta_4 \right), \\
b= \left( \Theta_1^2+\Theta_2^2-\Theta_3^2-\Theta_4^2 \right) \left( \Theta_1^2-\Theta_2^2+\Theta_3^2-\Theta_4^2 \right)  \left( 2\,\Theta_1\Theta_2-2\,\Theta_3\Theta_4 \right), \\
c= \left( \Theta_1^2-\Theta_2^2-\Theta_3^2+\Theta_4^2 \right) \left( \Theta_1^2+\Theta_2^2+\Theta_3^2+\Theta_4^2 \right)  \left( 2\,\Theta_1\Theta_2+2\,\Theta_3\Theta_4\right), \\
d^2 =256\,\Theta_1\Theta_2\Theta_4\Theta_3 \left( \Theta_1^2\Theta_4^2-\Theta_2^2\Theta_3^2 \right)  \left( \Theta_1^4-\Theta_2^4-\Theta_3^4+\Theta_4^4 \right) \\
+ 8 \left( \Theta_1^2+\Theta_4^2 \right) \left( \Theta_2^2+\Theta_3^2 \right) \left( \Theta_1^2+\Theta_2^2+\Theta_3^2+\Theta_4^2 \right)^2\left( \Theta_1^2-\Theta_2^2-\Theta_3^2+\Theta_4^2 \right)^2\\
+8 \left(\Theta_1^2-\Theta_4^2 \right) \left( \Theta_2^2-\Theta_3^2 \right)   \left( \Theta_1^2+\Theta_2^2-\Theta_3^2-\Theta_4^2 \right)^2 \left( \Theta_1^2-\Theta_2^2+\Theta_3^2-\Theta_4^2 \right)^2\\
-32\left( \Theta_1^2\Theta_2^2+\Theta_3^2\Theta_4^2 \right)  \left( \Theta_1^2+\Theta_2^2+\Theta_3^2+\Theta_4^2 \right)  \left( \Theta_1^2-\Theta_2^2+\Theta_3^2-\Theta_4^2 \right) \\
\times  \left( \Theta_1^2-\Theta_2^2-\Theta_3^2+\Theta_4^2 \right)  \left( \Theta_1^2+\Theta_2^2-\Theta_3^2-\Theta_4^2 \right).
\end{gather*}
\end{theorem}
\begin{proof}
All tropes in Proposition~\ref{prop:KummerNF} are determined in terms of the Rosenhain tetrahedron $\{\mathsf{T}_{256}, \mathsf{T}_{136}, \mathsf{T}_{246}, \mathsf{T}_{2}\}$, with tropes given by $\theta_1(z)^2,\theta_2(z)^2,\theta_7(z)^2,\theta_{12}(z)^2$. One checks that the Rosenhain tetrahedron is contained in three sets of eight tropes corresponding to the even eights $\Delta_{12}$, $\Delta_{13}$ in Proposition~\ref{prop:EvenEights}. The proof then proceeds analogous to the proof of Proposition~\ref{prop:ThetaImageG}. 
\end{proof}
\section{Non-principally polarized Kummer surfaces}
In this section we consider abelian surfaces $\mathbf{B}$ with a polarization of type $(d_1,d_2)$ given by an ample symmetric line bundle $\mathcal{N}$. We recall that by the Riemann-Roch theorem we have $\chi(\mathcal{N})=(\mathcal{N}^2)/2=d_1d_2$. It follows from \cite[Prop.~4.5.2]{MR2062673} that $\mathcal{N}$ is ample if and only if $h^i(\mathbf{B},\mathcal{N})=0$ for $i=1,2$ and $(\mathcal{N}^2)>0$. Therefore, the line bundle defines an associated rational map $\varphi_{\mathcal{N}}: \mathbf{B} \to \mathbb{P}^{d_1d_2-1}$. It was proven in \cite[Prop.~4.1.6, Lemma 10.1.2]{MR2062673} that the linear system $|\mathcal{N}|$ is base point free for $d_1=2$ and for $d_1=1$ and $d_2 \ge 3$, and it has exactly four base points if $(d_1,d_2)=(1,2)$. Moreover, every polarization is induced by a principal polarization $\mathcal{L}$ on an abelian surface $\mathbf{A}$, that is, there is an isogeny $\phi: \mathbf{B} \to \mathbf{A}$ such that $\phi^* \mathcal{L}\cong \mathcal{N}$  \cite[Prop.~4.1.2]{MR2062673}.

\subsection{$(1,4)$-polarized Kummer surfaces}
Let us now consider a generic abelian surface $\mathbf{B}$ with a $(1,4)$-polarization.  In~\cite{MR2062673} the following octic surface in $\mathbb{P}^3$ was considered:
\begin{definition}
A Birkenhake-Lange octic is the surface in $\mathbb{P}^3(Z_0, \dots,Z_3)$ given by
\begin{equation}
\label{Birkenhake-Lange-Octic}
\begin{split}
a^2  \big(Z_0^4 Z_1^4 + Z_2^4 Z_3^4\big) + b^2  \big(Z_0^4 Z_2^4 + Z_1^4 Z_3^4\big) + c^2 \big(Z_0^4 Z_3^4 + Z_1^4 Z_2^4\big)\\
+  2ab \big(Z_0^2 Z_1^2 - Z_2^2 Z_3^2\big) \big(Z_0^2 Z_2^2 + Z_1^2 Z_3^2\big) - 2 a c  \big(Z_0^2 Z_1^2 + Z_2^2 Z_3^2\big) \big(Z_0^2 Z_3^2 + Z_1^2 Z_2^2\big) \\
+ 2 b c \big(Z_0^2 Z_2^2 - Z_1^2 Z_3^2\big) \big(Z_0^2 Z_3^2 - Z_1^2 Z_2^2\big) + d^2 Z_0^2Z_1^2Z_2^2Z_3^2 =0 ,
\end{split}
\end{equation}
with $[a:b:c:d] \in \mathbb{P}^3$.
\end{definition}
We have the following lemma:
\begin{lemma}
The map $\mathbb{P}(Z_0,Z_1,Z_2,Z_3) \to \mathbb{P}(Y_0,Y_1,Y_2,Y_3)$ with $Y_i=Z_i^2$ induces a covering 
of the octic surface in Equation~(\ref{Birkenhake-Lange-Octic}) onto the quartic in Equation~(\ref{Birkenhake-Lange-Quartic}) which is $8:1$ outside the coordinate planes
$Z_i=0$. 
\end{lemma}
\begin{proof}
Along the coordinate planes the covering is $4:1$, and the coordinate points $P_0=[1:0:0:0]$, $P_1=[0:1:0:0]$, $P_2=[0:0:1:0]$, $P_3=[0:0:0:1]$ are of multiplicity four in the image. The octic surface has double curves along the coordinate planes. The coordinate planes $Y_i=0$ form a Rosenhain tetrahedron.  Therefore, the coordinate plane passes through six nodes on the Kummer quartic. For example, $Y_3=0$ passes through $P_0, P_1, P_2$ and three more nodes $P'_0, P'_1, P'_2$. Therefore, the preimage contains $3+4\cdot 3$ nodes. 
\end{proof}
\par We consider the rational map $\varphi_{\mathcal{N}}: \mathbf{B} \to \mathbb{P}^3$ associated with the line bundle $\mathcal{N}$. The map $\varphi_{\mathcal{N}}$ cannot be an embedding, i.e., diffeomorphic onto its image. However, it is generically birational onto its image  \cite{MR2062673}.  The following was proved in \cite[Sec.10.5]{MR2062673}:
\begin{corollary}{\cite[Prop.~10.5.7]{MR2062673}}
\label{cor:1-4-polarization}
If $\mathcal{N}$ is the ample symmetric line bundle on an abelian surface $\mathbf{B}$ defining a polarization of type $(1,4)$ such that the induced canonical map $\varphi_{\mathcal{N}}: \mathbf{B} \to \mathbb{P}^3$ is birational, then the coordinates of $\mathbb{P}^3$ can be chosen such that  $\varphi_{\mathcal{N}}(\mathbf{B})$ is given by the Birkenhake-Lange octic in Equation~(\ref{Birkenhake-Lange-Octic}).
\end{corollary}
\qed
\par We again have the following:
\begin{corollary}
\label{prop:ThetaImageRO}
The surface in $\mathbb{P}^3$ given by Equation~(\ref{Birkenhake-Lange-Octic}) is birational to the singular Kummer variety $\mathbf{B}/\langle -\mathbb{I} \rangle$ where the coordinates are given by
\begin{equation} 
\label{KummerVariables2}
 [Z_0:Z_1:Z_2:Z_3]=\left[\theta_{1}(z):\theta_{2}(z):\theta_7(z):\theta_{12}(z)\right].
\end{equation}
\end{corollary}
\qed
\subsection{$(2,4)$-polarized and $(1,2)$-polarized Kummer surfaces}
Let us now consider the generic abelian surface $\mathbf{B}$ with a $(1,2)$-polarization. Barth studied abelian surfaces with $(1,2)$-polarization in~\cite{MR946234}. An excellent summary of Barth's construction was given by Garbagnati in \cite{Garbagnati08, MR3010125}. Kummer surfaces with $(1,2)$-polarization were also discussed \cite{MR2306633, MR2729013}. Barth studied a projective model of the surface $\operatorname{Kum}(\mathbf{B})$ as intersection of three quadrics in $\mathbb{P}^5$, giving also the explicit equations of these quadrics. We will show how these conics arise as Mumford identities of Theta functions. We make the following:
\begin{definition}
A Barth surface is the surface in $\mathbb{P}^7(w,\dots,z,X_1,\dots,X_4)$ given as the complete intersection of the 6 quadrics 
\begin{equation}
\label{Barth_a}
\begin{split}
 2 p_0q_0 (X_1^2+X_2^2) & = (p^2_0 +q ^2_0) (w^2+x^2)-(p^2_0-q^2_0)(y^2+z^2),\\
 2 r_0s_0 (X_1^2-X_2^2) & = (r^2_0 +s^2_0) (w^2-x^2)+(r^2_0-s^2_0)(y^2-z^2),\\
 4 u_0 t_0 X_1 X_2  & =2(t^2_0+u^2_0) wx - 2(t^2_0-u^2_0) yz,
\end{split}
\end{equation}
and
\begin{equation}
\label{Barth_b}
\begin{split}
2 p_0q_0 (X_3^2+X_4^2) & = (-p^2_0 +q^2_0) (w^2+x^2)+(p^2_0+q^2_0)(y^2+z^2),\\
2 r_0s_0 (X_3^2-X_4^2) & = (-r^2_0 +s^2_0) (w^2-x^2)-(r^2_0+s^2_0)(y^2-z^2),\\
4 u_0 t_0 X_3 X_4  & =2(t^2_0-u^2_0) wx - 2(t^2_0+u^2_0) yz,
\end{split}
\end{equation}
where $[p_0: \dots :u_0] \in \mathbb{P}^5$ such that
\begin{equation}
 t_0^2 u_0^2  = p_0^2 q_0^2- r_0^2 s_0^2 \, , \qquad
 t_0^4 + u_0^4  = p_0^4 + q_0^4- r_0^4 - s_0^4 \,.
\end{equation}
\end{definition}
We start by considering the polarization given by the divisor $2\mathcal{N}$ on $\mathbf{B}$  defining a polarization of type $(2, 4)$. Barth proved the following in \cite[Prop.~2.1]{MR946234}:
\begin{proposition} 
The divisor $2\mathcal{N}$ is very ample on $\mathbf{B}$ and the linear system $|2\mathcal{N}|$ embeds $\mathbf{B}$ as a smooth surface of degree 16 in $\mathbb{P}^7$.
\end{proposition}
\begin{remark}
It was proved in \cite[Prop.~4.6]{MR946234} that a generic set of moduli parameters satisfies
\begin{gather*}
 (p_0^2s_0^2-q_0^2r_0^2) (p_0^2r_0^2-q_0^2s_0^2)(p_0^2u_0^2-q_0^2t_0^2) (p_0^2t_0^2-q_0^2u_0^2)\\
 \times  (r_0^2u_0^2-s_0^2t_0^2) (r_0^2t_0^2-s_0^2u_0^2) \not = 0.
\end{gather*}
\end{remark}
We have the following:
\begin{corollary}
For generic parameters $[p_0:\dots:u_0] \in \mathbb{P}^5$, the six quadrics defining the Barth surface in Equations~(\ref{Barth_a})~and~(\ref{Barth_b}) generate the ideal of a smooth irreducible surface of degree 16 in $\mathbb{P}^7$ isomorphic to $\mathbf{B}$ with $(2, 4)$-polarization given by $2\mathcal{N}$.
\end{corollary}
\begin{proof}
The quadratic equations for the conics defining $\mathbf{B}$ in $\mathbb{P}^7$ were determined by Barth in \cite[Eq.~(2.10)]{MR946234}. 
The result then follows when taking suitable linear combinations of these equations using the variable transformation
\begin{equation}
 \mathbf{x}_1=w, \; \mathbf{x}_2=x, \; \mathbf{x}_3=X_1, \; \mathbf{x}_4=X_2, \; \mathbf{x}_5=y, \; \mathbf{x}_6=z,\; \mathbf{x}_7=i X_4, \; \mathbf{x}_8=-iX_3,
\end{equation} 
and identifying the moduli parameters according to
\begin{equation}
 \lambda_1=p_0, \; \mu_1=q_0, \; \lambda_2=s_0, \; \mu_2=r_0, \; \lambda_3=t_0, \; \mu_3=u_0.
\end{equation}
\end{proof}
With respect to the action of the involution $-\mathbb{I}$ on $\mathbf{B}$, the space $H^0(\mathbf{B},2\mathcal{N})$ with $h^0(\mathbf{B}, 2\mathcal{N}) = 8$ splits into the direct sum $H^0(\mathbf{B},2\mathcal{N}) = H^0(\mathbf{B},2\mathcal{N})_+ \oplus H^0(\mathbf{B},2\mathcal{N})_-$ of eigenspaces of dimensions $h^0(\mathbf{B},2\mathcal{N})_+=6$ and $h^0(\mathbf{B},2\mathcal{N})_-= 2$, respectively. In particular, it is possible to choose the coordinates of $\mathbb{P}^7$ such that $\{X_3=X_4=0\}$ is the subspace of $\mathbb{P}^7$ invariant under the action of $-\mathbb{I}$. We will denote the subspace by $\mathbb{P}^5_+ \cong \mathbb{P} H^0(\mathbf{B},2\mathcal{N})_+$, and the anti-invariant subspace  $\{w=0, \dots, X_1 =0, X_2=0\}$ by $\mathbb{P}^1_- \cong \mathbb{P} H^0(\mathbf{B},2\mathcal{N})_-$; see \cite[Sec.~2]{MR946234}. If we consider the projection
\[
  \Pi: \; \mathbb{P}^7 \to \mathbb{P}^5_+, \quad  [w: x: y: z :X_1:X_2:X_3 : X_4] \mapsto [w: x:y:z :X_1:X_2],
\]
with center $\mathbb{P}^1_-$, it was proved in \cite[Sec.~4.3]{MR946234} that $\mathbf{B} \cap \mathbb{P}^1_- = \emptyset$. Hence, the projection $\Pi$ is well defined and induces a double cover from $\mathbf{B}$ onto its image. The following was proved in \cite[Prop.~4.6]{MR946234}:
\begin{corollary}
\label{cor:KumB}
The three quadrics in Equations~(\ref{Barth_a}) generate the ideal of an irreducible surface of degree 8 in $\mathbb{P}^5_+$ with 16 normal singularities isomorphic to the singular Kummer variety $\mathbf{B}/\langle -\mathbb{I} \rangle$.
\end{corollary}
Using the same argument as before, the projection
\begin{equation}
\label{p35}
 \pi : \mathbb{P}^5_+ \to \mathbb{P}^3, \qquad  [w: x: y: z :X_1:X_2] \mapsto [w:x:y:z],
\end{equation}
is well-defined and induces a double cover of $\mathbf{B}/\langle -\mathbb{I} \rangle$ onto its image. We have the following:
\begin{proposition}
\label{prop:BGH}
The image of the projection $\pi$ is isomorphic to the singular Kummer variety $\mathbf{}/\langle -\mathbb{I} \rangle$ associated with a principally polarized abelian variety $\mathbf{A}$.
\end{proposition}
\begin{proof}
The statement follows by eliminating $X_1, X_2$ from Equations~(\ref{Barth_a}) and recovering a G\"opel-Hudson quartic in Equation~(\ref{Goepel-Quartic}) with parameters
\begin{gather}
\nonumber
A = \frac{2(p_0^2r_0^2+q_0^2s_0^2)}{p_0^2r_0^2-q_0^2s_0^2}, \; B = \frac{2(p_0^2q_0^2+r_0^2s_0^2)}{p_0^2q_0^2-r_0^2s_0^2},  
C= \frac{2(p_0^2s_0^2+q_0^2r_0^2)}{p_0^2s_0^2-q_0^2r_0^2}, \\
\label{param:Goepel-Quartic2}
D = \frac{4p_0^2q_0^2r_0^2s_0^2(t_0^2-u_0^2)(t_0^2+u_0^2)}{(p_0^2r_0^2-q_0^2s_0^2)(p_0^2q_0^2-r_0^2s_0^2)(p_0^2s_0^2-q_0^2r_0^2)}.
\end{gather}
These moduli parameters equal the ones in Equations~(\ref{Eq:paramABCD}) when using Equations~(\ref{KummerParameter_thetas}) and~(\ref{KummerParameter_thetas}).
\end{proof}
\begin{remark}
Other projections are obtained by eliminating either $w, x$ or $y, z$ instead. These images are isomorphic G\"opel-Hudson quartics and are related by the action of a projective automorphism mapping one even G\"opel tetrahedron to another one with two tropes in common. 
\end{remark}
We now describe the role of the even eight in the construction above:  Mehran proved in \cite{MR2804549} that the rational double cover of a smooth Kummer surface $\operatorname{Kum}(\mathbf{A})$ with principal polarization branched along an even eight is a Kummer surface with $(1,2)$-polarization. In terms of the singular Kummer variety $\mathbf{B}/\langle -\mathbb{I} \rangle$ and its image -- which is a singular Kummer variety $\mathbf{A}/\langle -\mathbb{I} \rangle$ associated with a principally polarized abelian variety -- this can be interpreted as follows: the 16 singular points of $\mathbf{B}/\langle -\mathbb{I} \rangle$ are mapped to 8 singular points on $\mathbf{A}/\langle -\mathbb{I} \rangle$ that form the complement of the even eight. 
\par The Kummer variety $\mathbf{B}/\langle -\mathbb{I} \rangle$ is the complete intersection of  three quadrics in Equations~(\ref{Barth_a}) which we denote by $\mathsf{Q}_1, \mathsf{Q}_2, \mathsf{Q}_3$. The Kummer variety is contained in each quadric, or equivalently, in the hypernet of quadrics $\alpha_1 \mathsf{Q}_1 + \alpha_2 \mathsf{Q}_2 +\alpha_3 \mathsf{Q}_3$ for complex numbers $\alpha_1, \alpha_2, \alpha_3 \in \mathbb{C}$. We have the following:
\begin{lemma}
\label{lem:Ks}
For generic moduli parameters there are exactly four special quadrics $\mathsf{K}_1, \dots, \mathsf{K}_4$ in the hypernet where the rank equals three. The four special quadrics are given by
\begin{equation}
\begin{split}
 \mathsf{K_1}:  c_{1,1} \big( w_0 x - x_0 w\big)^2 + c_{1,2}   \big( y_0 z - z_0 y\big)^2 + c_{1,3}   \big( X_1^{(0)} X_2 -  X_2^{(0)} X_1\big)^2 =0, \\
 \mathsf{K_2}:  c_{2,1} \big( w_0 x + x_0 w\big)^2 + c_{2,2}   \big( y_0 z+ z_0 y\big)^2 + c_{2,3}   \big( X_1^{(0)} X_2 +  X_2^{(0)} X_1\big)^2 =0, \\
 \mathsf{K_3}:  c_{3,1} \big( w_0 w - x_0 x\big)^2 + c_{3,2}   \big( y_0 y - z_0 z\big)^2 + c_{3,3}   \big( X_1^{(0)} X_1 -  X_2^{(0)} X_2\big)^2 =0, \\
 \mathsf{K_4}:  c_{4,1} \big( w_0 w + x_0 x\big)^2 + c_{4,2}   \big( y_0 y + z_0 z\big)^2 + c_{4,3}   \big( X_1^{(0)} X_1 +  X_2^{(0)} X_2\big)^2 =0,
\end{split} 
\end{equation}
where $c_{i,j} \in \mathbb{C}^*$, the parameters $w_0, x_0, y_0, z_0$ are related to $p_0, q_0, r_0, s_0$ by Equations~(\ref{KummerParameter_thetas}) and
\begin{equation}
 X_1^{(0) \, 2} X_2^{(0) \, 2}   = w_0^2 x_0^2- y_0^2 z_0^2 \, , \qquad
 X_1^{(0) \, 4}  + X_1^{(0) \, 4}   = w_0^4 + x_0^4- y_0^4 - z_0^4 \,.
\end{equation}
\end{lemma}
\begin{proof}
The fact that there are exactly four sets of parameters, and thus four special quadrics $\mathsf{K}_1, \dots, \mathsf{K}_4$ in the hypernet where the rank equals three was proven in \cite[Sec.~4]{MR946234}; see also a similar computation outlined in \cite[Sec.~2.4.2]{Garbagnati08}.  The four set of parameters where the rank of the quadric in the hypernet equals three are given by 
\begin{gather*}
 \alpha_1  = \pm \frac{\sqrt{(s_0^2u_0^2-r_0^2 t_0^2)(s_0^2 t_0^2-r_0^2u_0^2)}}{p_0^2s_0^2t_0^2},\;
 \alpha_2  = \pm \frac{\sqrt{(p_0^2u_0^2-q_0^2 t_0^2)(p_0^2 t_0^2-q_0^2u_0^2)}}{p_0^2s_0^2t_0^2},\;\\
 \alpha_3  = \pm \frac{\sqrt{(p_0^2r_0^2-q_0^2 s_0^2)(q_0^2 r_0^2-p_0^2s_0^2)}}{p_0^2s_0^2t_0^2}.
\end{gather*}
Plugging these values into $\alpha_1 \mathsf{Q}_1 + \alpha_2 \mathsf{Q}_2 +\alpha_3 \mathsf{Q}_3$, we obtain, after a tedious computation, the four quadrics.
\end{proof}
The singular locus of each quadric $\mathsf{K}_i$ in Lemma~\ref{lem:Ks} is a plane $\mathsf{S}_i$ for $1\le i \le 4$, given by
\begin{equation}
\label{eq:SingPlanes}
\begin{split}
 \mathsf{S}_1:& \quad w_0 x = x_0 w, \quad y_0 z = z_0 y, \quad X_1^{(0)} X_2 =  X_2^{(0)} X_1\;,\\
 \mathsf{S}_2:& \quad w_0 x = -x_0 w, \quad y_0 z = -z_0 y, \quad X_1^{(0)} X_2 =  -X_2^{(0)} X_1\;,\\
 \mathsf{S}_3:& \quad x_0 x = w_0 w, \quad y_0 y = z_0 z, \quad X_1^{(0)} X_1 =  X_2^{(0)} X_2\;,\\
 \mathsf{S}_4:& \quad x_0 x = -w_0 w, \quad y_0 y = -z_0 z, \quad X_1^{(0)} X_1 =  -X_2^{(0)} X_2\;.
\end{split}
\end{equation}
On each plane $\mathsf{S}_i$ with $1\le i \le 4$, the other cones cut out, for generic moduli parameters, pencils of conics with four distinct base points \cite[Lemma.~4.1]{MR946234}. These four points on the four different planes constitute precisely the 16 singular points $\mathbf{B}/\langle -\mathbb{I} \rangle$. We have the following:
\begin{proposition}
The even eight of the projection $\pi$ in Equation~(\ref{p35}) consists of the eight nodes $[w:x:y:z]$ on the G\"opel-Hudson quartic~(\ref{Goepel-Quartic}) given by
\begin{gather*}
[y_0:z_0:w_0:x_0], [-y_0:-z_0:w_0:x_0], [-y_0:z_0:-w_0:x_0], [-y_0:z_0:w_0:-x_0],\\
[z_0:y_0:x_0:w_0], [-z_0:-y_0:x_0:w_0], [-z_0:y_0:-x_0:w_0], [-z_0:y_0:x_0:-w_0].
\end{gather*}
\end{proposition}
\begin{proof}
The image under the projection $\pi$ is a singular Kummer variety $\mathbf{A}/\langle -\mathbb{I} \rangle$ associated with a principally polarized abelian variety realized as G\"opel-Hudson quartic. We computed the 16 singular points of a G\"opel-Hudson quartic in Lemma~\ref{lem:16nodes}.  The 16 singular points $\mathbf{B}/\langle -\mathbb{I} \rangle$ are located on the four planes $\mathsf{S}_i$ for $1\le i \le 4$. Under the projection $\pi$ the ratios of the first four coordinates as determined by the equations of the planes must be preserved. We computed the 16 singular points of a G\"opel-Hudson quartic in Lemma~\ref{lem:16nodes}. The image of the 16 singular points $\mathbf{B}/\langle -\mathbb{I} \rangle$ under the double cover $\pi$ are the eight nodes on $\mathbf{A}/\langle -\mathbb{I} \rangle$ given by
\begin{gather*}
[w_0:x_0:y_0:z_0], [-w_0:-x_0:y_0:z_0], [-w_0:x_0:-y_0:z_0], [-w_0:x_0:y_0:-z_0],\\
[x_0:w_0:z_0:y_0], [-x_0:-w_0:z_0:y_0], [-x_0:w_0:-z_0:y_0], [-x_0:w_0:z_0:-y_0].
\end{gather*}
The even eight consists of the complimentary set of nodes.
\end{proof}
\par We provide an explicit model for the Barth surface in terms of Theta functions.  We have the following:
\begin{theorem}
\label{prop:ThetaImageB}
The surface $\mathbf{B}$ in $\mathbb{P}^7$ given by Equations~(\ref{Barth_a})~and~(\ref{Barth_b}) with moduli parameters given by
\begin{equation} 
 [p_0:q_0:r_0:s_0:t_0:u_0]=\left[\theta_1:\theta_2:\theta_3:\theta_4: \theta_8:\theta_{10}\right] \;.
\end{equation}
is isomorphic to the image of six even and two odd Theta functions given by 
\begin{gather*} 
 [w:x:y:z:X_1:X_2:X_3:X_4]\\=[ \Theta_1(2  z): \Theta_2(2  z): \Theta_3(2  z) : \Theta_4(2  z): \Theta_8(2  z): \Theta_{10}(2  z): \Theta_{13}(2  z) : \Theta_{16}(2  z) ].
\end{gather*}
In particular, the ideal of the surface $\mathbf{B}$ in $\mathbb{P}^7$ is generated by the Mumford relations for the Theta function $\Theta_1(2  z), \dots, \Theta_{16}(2  z)$. Moreover,  the ideal of the singular Kummer variety $\mathbf{B}/\langle -\mathbb{I} \rangle$ is generated by the Mumford relations for the Theta function $\Theta_1(2  z), \dots, \Theta_4(2z),   \Theta_8(2z),  \Theta_{10}(2z)$. The map $\pi: \mathbf{B}/\langle -\mathbb{I} \rangle \to \mathbf{A}_\tau/\langle -\mathbb{I} \rangle$ is a rational double cover onto the singular Kummer variety associated with the principally polarized abelian variety $\mathbf{A}_\tau$.
\end{theorem}
\begin{proof}
We use a set of eight tropes consisting of the Rosenhain tetrahedra 
\[ 
 \{ \mathsf{T}_{256}, \mathsf{T}_{126}, \mathsf{T}_{236},\mathsf{T}_{4}\}\quad \text{and} \quad \{\mathsf{T}_{136}, \mathsf{T}_{356}, \mathsf{T}_{156},\mathsf{T}_{6}\}.
 \]
 The corresponding Theta functions are 
 \[ 
  \{ \Theta_1(2z), \Theta_4(2z), \Theta_{10}(2z), \Theta_{16}(2z)\} \quad \text{and} \quad \{ \Theta_2(2z), \Theta_3(2z), \Theta_{8}(2z), \Theta_{13}(2z)\}.
 \]
 The bi-monomial Mumford relations in Proposition~\ref{prop:MMF}, after $\tau$ is replaced by $2\tau$, include the equations
 \begin{equation}
 \begin{split}
      \Theta_8 \Theta_{10}  \Theta_8(2z) \Theta_{10}(2z)  &= \Theta_1 \Theta_2  \Theta_1(2z) \Theta_2(2z) - \Theta_3\Theta_4 \Theta_3(2z) \Theta_4(2z), \;\\
   \Theta_8 \Theta_{10}  \Theta_{13}(2z) \Theta_{16}(2z) &=  \Theta_3 \Theta_4  \Theta_1(2z) \Theta_2(2z) - \Theta_1\Theta_2 \Theta_3(2z) \Theta_4(2z).
\end{split}
\end{equation}
Using the Frobenius identities~(\ref{Eq:FrobeniusIdentities}) we obtain
 \begin{equation}
 \begin{split}
   2 \theta_8 \theta_{10}  \Theta_8(2z) \Theta_{10}(2z)     &=  (\theta_8^2+\theta_{10}^2) \Theta_1(2z) \Theta_2(2z) - (\theta_8^2-\theta_{10}^2)  \Theta_3(2z) \Theta_4(2z), \;\\
   2 \theta_8 \theta_{10}  \Theta_{13}(2z) \Theta_{16}(2z) &=  (\theta_8^2-\theta_{10}^2)  \Theta_1(2z) \Theta_2(2z) - (\theta_8^2+\theta_{10}^2) \Theta_3(2z) \Theta_4(2z). \;
\end{split}
\end{equation}
Similarly, the quadratic Mumford relations in Proposition~\ref{prop:MF} combined with Equations~(\ref{Eq:degree2doubling}) yield
\begin{small}
\begin{equation*}
 \begin{split}
2 \theta_1 \theta_2  (\Theta_8^2(2z)+ \Theta_{10}^2(2z))= (\theta_1^2+\theta_2^2) (\Theta_1^2(2z)+\Theta_2^2(2z)) - (\theta_1^2-\theta_2^2) (\Theta_3^2(2z)+\Theta_4^2(2z)),\\
2 \theta_3 \theta_4  (\Theta_8^2(2z)- \Theta_{10}^2(2z))= (\theta_3^2+\theta_4^2) (\Theta_1^2(2z)-\Theta_2^2(2z)) + (\theta_3^2-\theta_4^2) (\Theta_3^2(2z)-\Theta_4^2(2z)),\\
2 \theta_1 \theta_2  (\Theta_{13}^2(2z)+ \Theta_{16}^2(2z))= (-\theta_1^2+\theta_2^2) (\Theta_1^2(2z)+\Theta_2^2(2z)) + (\theta_1^2+\theta_2^2) (\Theta_3^2(2z)+\Theta_4^2(2z)),\\
2 \theta_3 \theta_4  (\Theta_{13}^2(2z)- \Theta_{16}^2(2z))= (-\theta_3^2+\theta_4^2) (\Theta_1^2(2z)-\Theta_2^2(2z)) - (\theta_3^2+\theta_4^2) (\Theta_3^2(2z)-\Theta_4^2(2z)).
\end{split}
\end{equation*}
\end{small}
These are precisely Equations~(\ref{Barth_a})~and~(\ref{Barth_b}) for moduli parameters given by
\begin{equation} 
 [p_0:q_0:r_0:s_0:t_0:u_0]=\left[\theta_1:\theta_2:\theta_3:\theta_4: \theta_8:\theta_{10}\right] \;.
\end{equation}
and variables given by
\begin{gather*} 
 [w:x:y:z:X_1:X_2:X_3:X_4]\\=[ \Theta_1(2  z): \Theta_2(2  z): \Theta_3(2  z) : \Theta_4(2  z): \Theta_8(2  z): \Theta_{10}(2  z): \Theta_{13}(2  z) :  \Theta_{16}(2  z) ] \;.
\end{gather*}
By eliminating $X_1, X_2$ one obtains the G\"opel-Hudson quartic in Equation~(\ref{Goepel-Quartic}) with moduli parameters matching those in Equation~(\ref{KummerParameter3}).
\end{proof}
In Proposition~\ref{prop:ThetaImageB}, we used a set of eight tropes consisting of the Rosenhain tetrahedra $\{ \mathsf{T}_{256}, \mathsf{T}_{126}, \mathsf{T}_{236},\mathsf{T}_{4}\}$ and $\{\mathsf{T}_{136}, \mathsf{T}_{356}, \mathsf{T}_{156},\mathsf{T}_{6}\}$. The corresponding Theta functions are $\{ \Theta_1(2z), \Theta_4(2z), \Theta_{10}(2z), \Theta_{16}(2z)\}$ and $\{ \Theta_2(2z), \Theta_3(2z), \Theta_{8}(2z), \Theta_{13}(2z)\}$. Only the Theta functions $\Theta_{13}(2z)$ and $\Theta_{16}(2z)$ are odd, the remaining ones are even which allowed us to identify the sub-spaces $\mathbb{P}^5_+$ and $\mathbb{P}^1_-$. This computation can be generalized.  We have the following:
\begin{proposition}
\label{prop:KumB_IsoClasses}
Every set of eight tropes that satisfies a quadratic relation and corresponds to an even eight, given in Proposition~\ref{prop:EvenEights}, determines an isomorphism between the complete intersection of the three quadrics in Equations~(\ref{Barth_a}) in $\mathbb{P}^5_+$ and a singular Kummer variety $\mathbf{B}/\langle -\mathbb{I} \rangle$ with $(1, 2)$-polarization. 
\end{proposition}
\begin{proof}
The construction of the Kummer variety $\mathbf{B}/\langle -\mathbb{I} \rangle$ with $(1,2)$-polarization as projection from the Barth surface $\mathbf{B}$ in Corollary~\ref{cor:KumB} requires a splitting  of the space $H^0(\mathbf{B},2\mathcal{N})$ with $h^0(\mathbf{B}, 2\mathcal{N}) = 8$ into the direct sum $H^0(\mathbf{B},2\mathcal{N}) = H^0(\mathbf{B},2\mathcal{N})_+ \oplus H^0(\mathbf{B},2\mathcal{N})_-$ of $\pm1$ eigenspaces of dimensions $h^0(\mathbf{B},2\mathcal{N})_+=6$ and $h^0(\mathbf{B},2\mathcal{N})_-= 2$.  Among the 30 sets of eight tropes given in Proposition~\ref{prop:EvenEights}, 15 sets correspond to even eights $\Delta_{ij}$ not containing the node $p_0$, the other 15 sets do contain the node $p_0$. The sets of eight tropes corresponding to even eights $\Delta_{ij}$ contain 6 even tropes of the form $\mathsf{T}_{kl6}$ and 2 odd tropes of the form $\mathsf{T}_{m}$, whereas the other 15 sets contain 4 even and 4 odd tropes. Therefore, for the sets of eight tropes corresponding to an even eight $\Delta_{ij}$  we obtain an abelian surface $\mathbf{B}_{ij}$ with $(2, 4)$-polarization and, by eliminating the odd coordinates $X_3$ and $X_4$, a singular Kummer variety $\mathbf{B}_{ij}/\langle -\mathbb{I} \rangle$ in Corollary~\ref{cor:KumB}.
\end{proof}
Mehran proved \cite[Cor.~4.7]{MR2804549} that there are exactly 15 distinct isomorphism classes of smooth Kummer surfaces $\operatorname{Kum}(\mathbf{B}_{ij})$ with $(1, 2)$-polarization for $1\le i < j < 6$ that are rational double covers of the smooth Kummer surface $\operatorname{Kum}(\mathbf{A})$ associated with a particular principally polarized abelian variety $\mathbf{A}$.  It was also proved that each such double cover is induced by a two-isogeny $\phi_{ij}: \mathbf{B}_{ij} \to \mathbf{A}$ of abelian surfaces \cite[Prop~2.3]{MR2804549}. We have the following:
\begin{theorem}
\label{thm:KumB_IsoClasses}
The 15 Barth surfaces obtained from the 15 sets of eight tropes that satisfy a quadratic relation and correspond to an even eight, given in Proposition~\ref{prop:EvenEights}, 
realize all distinct 15 isomorphism classes of singular Kummer varieties with $(1, 2)$-polarization covering a fixed smooth Kummer surface $\operatorname{Kum}(\mathbf{A})$ with a principal polarization.
\end{theorem}
\begin{proof}
Each sets of eight tropes that satisfies a quadratic relation and corresponds to an even eight, given in Proposition~\ref{prop:EvenEights}, determines a Barth surface using Proposition~\ref{prop:KumB_IsoClasses}. We first show that for any two different sets of eight tropes corresponding to even eights $\Delta_{ij}$ and $\Delta_{i'j'}$, the images of the singular Kummer varieties $\mathbf{B}_{ij}/\langle -\mathbb{I} \rangle$ and $\mathbf{B}_{i'j'}/\langle -\mathbb{I} \rangle$, respectively, under their respective projections $\pi$ and $\pi'$ are isomorphic to the same singular Kummer variety $\mathbf{A}/\langle -\mathbb{I} \rangle$ associated with a principally polarized abelian variety $\mathbf{A}$: each set of eight tropes corresponding to an even eight $\Delta_{ij}$ contains three different G\"opel tetrahedra that pairwise have two tropes in common. Each G\"opel tetrahedron defines a projection from $\mathbf{B}_{ij}/\langle -\mathbb{I} \rangle$ onto a G\"opel-Hudson quartic in Proposition~\ref{prop:BGH} by eliminating the complimentary pair of variables. However, all resulting quartics are isomorphic and related by the action of a projective automorphisms obtained by the composition of isomorphisms obtained in Proposition~\ref{prop:KummerNF}. Because of Lemma~\ref{lem:quadraticsEEb}, each of the three G\"opel tetrahedra also appear as projection of two other singular Kummer varieties $\mathbf{B}_{i'j'}/\langle -\mathbb{I} \rangle$. Therefore, the images obtained as projections of all the different singular Kummer varieties $\mathbf{B}_{ij}/\langle -\mathbb{I} \rangle$ are isomorphic. They realize the same singular Kummer variety $\mathbf{A}/\langle -\mathbb{I} \rangle$ associated with a principally polarized abelian variety $\mathbf{A}$ and moduli parameter given by Equations~(\ref{param:Goepel-Quartic2}).
\par  It was proved in \cite[Prop.~4.2]{MR2804549} that all distinct 15 isomorphism classes $\operatorname{Kum}(\mathbf{B}_{ij})$ are obtained by taking a double cover branched along the fifteen different even eights $\Delta_{ij}$ on the smooth Kummer surface $\operatorname{Kum}(\mathbf{A})$. One only has to consider the even eights not containing $p_0$ is because one obtains the exact same Kummer surface whether one takes the double cover branched along an even eight or its complement \cite[Prop.~2.3]{MR2804549}.  After blowing down the exceptional divisors, the double covers are equivalent to 15 morphisms $p_{ij}: \mathbf{B}_{ij}/\langle -\mathbb{I} \rangle \to \mathbf{A}/\langle -\mathbb{I} \rangle$ which map the 16 nodes of $\mathbf{B}_{ij}$ to the eight nodes of $\mathbf{A}/\langle -\mathbb{I} \rangle$ contained in the complement of $\Delta_{ij}$. Proposition~\ref{prop:KumB_IsoClasses} proves that the projection $p$ in Equation~(\ref{p35}) when applied to all cases in Proposition~\ref{prop:KumB_IsoClasses} realizes all such morphisms.
\end{proof}
\bibliography{CM16} 

\providecommand{\bysame}{\leavevmode\hbox to3em{\hrulefill}\thinspace}
\providecommand{\MR}{\relax\ifhmode\unskip\space\fi MR }
\providecommand{\MRhref}[2]{%
  \href{http://www.ams.org/mathscinet-getitem?mr=#1}{#2}
}
\providecommand{\href}[2]{#2}
\begin{thebibliography}{10}

\bibitem{MR946234}
Wolf Barth, \emph{Abelian surfaces with {$(1,2)$}-polarization}, Algebraic
  geometry, {S}endai, 1985, Adv. Stud. Pure Math., vol.~10, North-Holland,
  Amsterdam, 1987, pp.~41--84. \MR{946234}

\bibitem{MR2729013}
F.~Bastianelli, G.~P. Pirola, and L.~Stoppino, \emph{Galois closure and
  {L}agrangian varieties}, Adv. Math. \textbf{225} (2010), no.~6, 3463--3501.
  \MR{2729013}

\bibitem{MR2062673}
Christina Birkenhake and Herbert Lange, \emph{Complex abelian varieties},
  second ed., Grundlehren der Mathematischen Wissenschaften [Fundamental
  Principles of Mathematical Sciences], vol. 302, Springer-Verlag, Berlin,
  2004. \MR{2062673}

\bibitem{MR1579732}
C.~W. Borchardt, \emph{\"uber die {D}arstellung der {K}ummerschen {F}l\"ache
  vierter {O}rdnung mit sechzehn {K}notenpunkten durch die {G}\"opelsche
  biquadratische {R}elation zwischen vier {T}hetafunctionen mit zwei
  {V}ariabeln}, J. Reine Angew. Math. \textbf{83} (1877), 234--244.
  \MR{1579732}

\bibitem{MR970659}
Jean-Beno\^it Bost and Jean-Fran\c{c}ois Mestre, \emph{Moyenne
  arithm{\'e}tico-g{\'e}om{\'e}trique et p{\'e}riodes des courbes de genre
  {$1$} et {$2$}}, Gaz. Math. (1988), no.~38, 36--64. \MR{970659}

\bibitem{MR1406090}
J.~W.~S. Cassels and E.~V. Flynn, \emph{Prolegomena to a middlebrow arithmetic
  of curves of genus {$2$}}, London Mathematical Society Lecture Note Series,
  vol. 230, Cambridge University Press, Cambridge, 1996. \MR{1406090}

\bibitem{2018arXiv180607460C}
A.~{Clingher}, A.~{Malmendier}, and T.~{Shaska}, \emph{{Six line configurations
  and string dualities}}, Comm. Math. Phys. (2019).

\bibitem{MR2369941}
Adrian Clingher and Charles~F. Doran, \emph{Modular invariants for lattice
  polarized {$K3$} surfaces}, Michigan Math. J. \textbf{55} (2007), no.~2,
  355--393. \MR{2369941}

\bibitem{MR2935386}
\bysame, \emph{Lattice polarized {K}3 surfaces and {S}iegel modular forms},
  Adv. Math. \textbf{231} (2012), no.~1, 172--212. \MR{2935386}

\bibitem{MR0296076}
Pierre Deligne, \emph{La conjecture de {W}eil pour les surfaces {$K3$}},
  Invent. Math. \textbf{15} (1972), 206--226. \MR{0296076}

\bibitem{MR2457735}
I.~Dolgachev and D.~Lehavi, \emph{On isogenous principally polarized abelian
  surfaces}, Curves and abelian varieties, Contemp. Math., vol. 465, Amer.
  Math. Soc., Providence, RI, 2008, pp.~51--69. \MR{2457735}

\bibitem{MR2112585}
Igor~V. Dolgachev, \emph{Abstract configurations in algebraic geometry}, The
  {F}ano {C}onference, Univ. Torino, Turin, 2004, pp.~423--462. \MR{2112585}

\bibitem{MR2964027}
\bysame, \emph{Classical algebraic geometry}, Cambridge University Press,
  Cambridge, 2012, A modern view. \MR{2964027}

\bibitem{MR1041476}
Eugene~Victor Flynn, \emph{The {J}acobian and formal group of a curve of genus
  {$2$} over an arbitrary ground field}, Math. Proc. Cambridge Philos. Soc.
  \textbf{107} (1990), no.~3, 425--441. \MR{1041476}

\bibitem{Garbagnati08}
Alice Garbagnati, \emph{Symplectic automorphisms on k3 surfaces}, Ph.D. thesis,
  Dipartimento di matematica, Universit{\`a} degli Studi di Milano, 2008.

\bibitem{MR3010125}
Alice Garbagnati and Alessandra Sarti, \emph{On symplectic and non-symplectic
  automorphisms of {K}3 surfaces}, Rev. Mat. Iberoam. \textbf{29} (2013),
  no.~1, 135--162. \MR{3010125}

\bibitem{MR2372155}
P.~Gaudry, \emph{Fast genus 2 arithmetic based on theta functions}, J. Math.
  Cryptol. \textbf{1} (2007), no.~3, 243--265. \MR{2372155}

\bibitem{MR1182682}
Maria~R. Gonzalez-Dorrego, \emph{{$(16,6)$} configurations and geometry of
  {K}ummer surfaces in {${\bf P}^3$}}, Mem. Amer. Math. Soc. \textbf{107}
  (1994), no.~512, vi+101. \MR{1182682}

\bibitem{MR1097176}
R.~W. H.~T. Hudson, \emph{Kummer's quartic surface}, Cambridge Mathematical
  Library, Cambridge University Press, Cambridge, 1990, With a foreword by W.
  Barth, Revised reprint of the 1905 original. \MR{1097176}

\bibitem{MR0141643}
Jun-Ichi Igusa, \emph{On {S}iegel modular forms of genus two}, Amer. J. Math.
  \textbf{84} (1962), 175--200. \MR{0141643}

\bibitem{MR0168805}
\bysame, \emph{On {S}iegel modular forms genus two. {II}}, Amer. J. Math.
  \textbf{86} (1964), 392--412. \MR{0168805}

\bibitem{MR3238326}
Kenji Koike, \emph{On {J}acobian {K}ummer surfaces}, J. Math. Soc. Japan
  \textbf{66} (2014), no.~3, 997--1016. \MR{3238326}

\bibitem{MR3263663}
Abhinav Kumar, \emph{Elliptic fibrations on a generic {J}acobian {K}ummer
  surface}, J. Algebraic Geom. \textbf{23} (2014), no.~4, 599--667.
  \MR{3263663}

\bibitem{MR1579281}
E.~E. Kummer, \emph{\"{U}ber die {F}l\"achen vierten {G}rades, auf welchen
  {S}chaaren von {K}egelschnitten liegen}, J. Reine Angew. Math. \textbf{64}
  (1865), 66--76. \MR{1579281}

\bibitem{MR3712162}
A.~Malmendier and T.~Shaska, \emph{The {S}atake sextic in {F}-theory}, J. Geom.
  Phys. \textbf{120} (2017), 290--305. \MR{3712162}

\bibitem{MR3366121}
Andreas Malmendier and David~R. Morrison, \emph{K3 surfaces, modular forms, and
  non-geometric heterotic compactifications}, Lett. Math. Phys. \textbf{105}
  (2015), no.~8, 1085--1118. \MR{3366121}

\bibitem{MR2306633}
Afsaneh Mehran, \emph{Double covers of {K}ummer surfaces}, Manuscripta Math.
  \textbf{123} (2007), no.~2, 205--235. \MR{2306633}

\bibitem{MR2804549}
\bysame, \emph{Kummer surfaces associated to {$(1,2)$}-polarized abelian
  surfaces}, Nagoya Math. J. \textbf{202} (2011), 127--143. \MR{2804549}

\bibitem{MR2352717}
David Mumford, \emph{Tata lectures on theta. {I}}, Modern Birkh{\"a}user
  Classics, Birkh{\"a}user Boston, Inc., Boston, MA, 2007, With the
  collaboration of C. Musili, M. Nori, E. Previato and M. Stillman, Reprint of
  the 1983 edition. \MR{2352717}

\bibitem{MR2514037}
\bysame, \emph{Abelian varieties}, Tata Institute of Fundamental Research
  Studies in Mathematics, vol.~5, Published for the Tata Institute of
  Fundamental Research, Bombay; by Hindustan Book Agency, New Delhi, 2008, With
  appendices by C. P. Ramanujam and Yuri Manin, Corrected reprint of the second
  (1974) edition. \MR{2514037}

\bibitem{MR0429917}
V.~V. Nikulin, \emph{Kummer surfaces}, Izv. Akad. Nauk SSSR Ser. Mat.
  \textbf{39} (1975), no.~2, 278--293, 471. \MR{0429917}

\bibitem{MR2367218}
E.~Previato, T.~Shaska, and G.~S. Wijesiri, \emph{Thetanulls of cyclic curves
  of small genus}, Albanian J. Math. \textbf{1} (2007), no.~4, 253--270.
  \MR{2367218}

\bibitem{MR1578135}
Fried.~Jul. Richelot, \emph{De transformatione integralium {A}belianorum primi
  ordinis commentatio. {C}aput secundum. {D}e computatione integralium
  {A}belianorum primi ordinis}, J. Reine Angew. Math. \textbf{16} (1837),
  285--341. \MR{1578135}

\bibitem{MR0419459}
E.~R. Shafarevitch, \emph{Le th{\'e}or{\`e}me de {T}orelli pour les surfaces
  alg{\'e}briques de type {$K3$}},  (1971), 413--417. \MR{0419459}

\bibitem{MR2296439}
Tetsuji Shioda, \emph{Classical {K}ummer surfaces and {M}ordell-{W}eil
  lattices}, Algebraic geometry, Contemp. Math., vol. 422, Amer. Math. Soc.,
  Providence, RI, 2007, pp.~213--221. \MR{2296439}

\end{thebibliography}
\bibliographystyle{amsplain}
\newpage
\begin{appendix}
\section{Mumford relations I}
\label{app:MumfordRelations}
The three-term relations between bi-monomial combinations of Theta functions $\xi_{i,j} = \theta_i(z)\theta_j(z)$ for $1\le i < j \le 16$ in Proposition~\ref{prop:MMF} generate an ideal that is generated by the following 60 equations:
\begin{equation}
\label{Eqn:MumfordBimonomial}
\scalemath{0.8}{
\begin{array}{lcllcl}
%1
\theta_1\theta_2\, \xi_{{3,4}}-\theta_3\theta_4\, \xi_{{1,2}}+\theta_8\theta_{10}\, \xi_{{13,16}} &=&0, 	& 
\theta_1\theta_2\, \xi_{{13,16}}-\theta_3\theta_4\, \xi_{{8,10}}+\theta_8\theta_{10}\, \xi_{{3,4}}&=&0,	\\
%2
\theta_1\theta_3\, \xi_{{7,9}}+\theta_2\theta_4\, \xi_{{11,12}}-\theta_7\theta_9\, \xi_{{1,3}} &=&0, 	& 
\theta_1\theta_3\, \xi_{{11,12}}+\theta_2\theta_4\, \xi_{{7,9}}-\theta_7\theta_9\, \xi_{{2,4}}&=&0,	\\
%3
\theta_1\theta_4\, \xi_{{5,6}}+\theta_2\theta_3\, \xi_{{14,15}}-\theta_5\theta_6\, \xi_{{1,4}}&=&0, 	& 
\theta_1\theta_4\, \xi_{{14,15}}+\theta_2\theta_3\, \xi_{{5,6}}-\theta_5\theta_6\, \xi_{{2,3}}&=&0,	\\
%4
\theta_1\theta_5\, \xi_{{7,8}}+\theta_4\theta_6\, \xi_{{11,13}}-\theta_7\theta_8\, \xi_{{1,5}}&=&0, 	& 
\theta_1\theta_5\, \xi_{{11,13}}+\theta_4\theta_6\, \xi_{{7,8}}-\theta_7\theta_8\, \xi_{{4,6}}&=&0,	\\
%5
\theta_1\theta_6\, \xi_{{4,5}}-\theta_4\theta_5\, \xi_{{1,6}}-\theta_9\theta_{10}\, \xi_{{12,16}}&=&0, 	& 
\theta_1\theta_6\, \xi_{{9,10}}-\theta_4\theta_5\, \xi_{{12,16}}-\theta_9\theta_{10}\, \xi_{{1,6}}&=&0,	\\
%6
\theta_1\theta_7\, \xi_{{5,8}}+\theta_3\theta_9\, \xi_{{14,16}}-\theta_5\theta_8\, \xi_{{1,7}}&=&0, 	& 
\theta_1\theta_7\, \xi_{{14,16}}+\theta_3\theta_9\, \xi_{{5,8}}-\theta_5\theta_8\, \xi_{{3,9}}&=&0,	\\
%7
\theta_1\theta_8\, \xi_{{5,7}}+\theta_2\theta_{10}\, \xi_{{12,15}}-\theta_5\theta_7\, \xi_{{1,8}}&=&0, 	& 
\theta_1\theta_8\, \xi_{{12,15}}+\theta_2\theta_{10}\, \xi_{{5,7}}-\theta_5\theta_7\, \xi_{{2,10}}&=&0,	\\
%8
\theta_1\theta_9\, \xi_{{3,7}}-\theta_3\theta_7\, \xi_{{1,9}}+\theta_6\theta_{10}\, \xi_{{13,15}}&=&0, 	& 
\theta_1\theta_9\, \xi_{{6,10}}-\theta_3\theta_7\, \xi_{{13,15}}-\theta_6\theta_{10}\, \xi_{{1,9}}&=&0,	\\
%9
\theta_1\theta_{10}\, \xi_{{6,9}}+\theta_2\theta_8\, \xi_{{11,14}}-\theta_6\theta_9\, \xi_{{1,10}}&=&0, 	& 
\theta_1\theta_{10}\, \xi_{{11,14}}+\theta_2\theta_8\, \xi_{{6,9}}-\theta_6\theta_9\, \xi_{{2,8}}&=&0,	\\
%10
\theta_2\theta_9\, \xi_{{10,14}}+\theta_4\theta_7\, \xi_{{5,13}}-\theta_6\theta_8\, \xi_{{1,11}}&=&0, 	& 
\theta_2\theta_9\, \xi_{{5,13}}+\theta_4\theta_7\, \xi_{{10,14}}-\theta_6\theta_8\, \xi_{{3,12}}&=&0,	\\
%11
\theta_2\theta_7\, \xi_{{8,15}}-\theta_4\theta_9\, \xi_{{6,16}}-\theta_5\theta_{10}\, \xi_{{1,12}}&=&0, 	& 
\theta_2\theta_7\, \xi_{{6,16}}-\theta_4\theta_9\, \xi_{{8,15}}+\theta_5\theta_{10}\, \xi_{{3,11}}&=&0,	\\
%12
\theta_3\theta_{10}\, \xi_{{2,16}}+\theta_4\theta_8\, \xi_{{1,13}}-\theta_6\theta_7\, \xi_{{5,11}}&=&0, 	& 
\theta_3\theta_{10}\, \xi_{{1,13}}+\theta_4\theta_8\, \xi_{{2,16}}-\theta_6\theta_7\, \xi_{{9,15}}&=&0,	\\
%13
\theta_2\theta_6\, \xi_{{4,15}}-\theta_3\theta_5\, \xi_{{1,14}}+\theta_8\theta_9\, \xi_{{7,16}}&=&0, 	& 
\theta_2\theta_6\, \xi_{{10,11}}+\theta_3\theta_5\, \xi_{{7,16}}-\theta_8\theta_9\, \xi_{{1,14}}&=&0,	\\
%14
\theta_2\theta_5\, \xi_{{4,14}}-\theta_3\theta_6\, \xi_{{1,15}}+\theta_7\theta_{10}\, \xi_{{9,13}}&=&0, 	& 
\theta_2\theta_5\, \xi_{{9,13}}-\theta_3\theta_6\, \xi_{{8,12}}+\theta_7\theta_{10}\, \xi_{{4,14}}&=&0,	\\
%15
\theta_3\theta_8\, \xi_{{2,13}}+\theta_4\theta_{10}\, \xi_{{1,16}}-\theta_5\theta_9\, \xi_{{6,12}}&=&0, 	& 
\theta_3\theta_8\, \xi_{{1,16}}+\theta_4\theta_{10}\, \xi_{{2,13}}-\theta_5\theta_9\, \xi_{{7,14}}&=&0,	\\
%16
\theta_2\theta_5\, \xi_{{3,6}}-\theta_3\theta_6\, \xi_{{2,5}}-\theta_7\theta_{10}\, \xi_{{11,16}}&=&0, 	& 
\theta_2\theta_5\, \xi_{{11,16}}-\theta_3\theta_6\, \xi_{{7,10}}+\theta_7\theta_{10}\, \xi_{{3,6}}&=&0,	\\
%17
\theta_2\theta_6\, \xi_{{8,9}}+\theta_3\theta_5\, \xi_{{12,13}}-\theta_8\theta_9\, \xi_{{2,6}}&=&0, 	& 
\theta_2\theta_6\, \xi_{{12,13}}+\theta_3\theta_5\, \xi_{{8,9}}-\theta_8\theta_9\, \xi_{{3,5}}&=&0,	\\
%18
\theta_2\theta_7\, \xi_{{4,9}}-\theta_4\theta_9\, \xi_{{2,7}}-\theta_5\theta_{10}\, \xi_{{13,14}}&=&0, 	& 
\theta_2\theta_7\, \xi_{{13,14}}-\theta_4\theta_9\, \xi_{{5,10}}+\theta_5\theta_{10}\, \xi_{{4,9}}&=&0,	\\
%19
\theta_2\theta_9\, \xi_{{4,7}}-\theta_4\theta_7\, \xi_{{2,9}}+\theta_6\theta_8\, \xi_{{15,16}}&=&0, 	& 
\theta_2\theta_9\, \xi_{{15,16}}+\theta_4\theta_7\, \xi_{{6,8}}-\theta_6\theta_8\, \xi_{{4,7}}&=&0,	\\
%20
\theta_1\theta_9\, \xi_{{8,14}}-\theta_3\theta_7\, \xi_{{5,16}}-\theta_6\theta_{10}\, \xi_{{2,11}}&=&0, 	& 
\theta_1\theta_9\, \xi_{{5,16}}-\theta_3\theta_7\, \xi_{{8,14}}+\theta_6\theta_{10}\, \xi_{{4,12}}&=&0,	\\
%21
\theta_1\theta_7\, \xi_{{10,15}}+\theta_3\theta_9\, \xi_{{6,13}}-\theta_5\theta_8\, \xi_{{2,12}}&=&0, 	& 
\theta_1\theta_7\, \xi_{{6,13}}+\theta_3\theta_9\, \xi_{{10,15}}-\theta_5\theta_8\, \xi_{{4,11}}&=&0,	\\
%22
\theta_1\theta_6\, \xi_{{3,15}}-\theta_4\theta_5\, \xi_{{2,14}}-\theta_9\theta_{10}\, \xi_{{7,13}}&=&0, 	& 
\theta_1\theta_6\, \xi_{{8,11}}-\theta_4\theta_5\, \xi_{{7,13}}-\theta_9\theta_{10}\, \xi_{{2,14}}&=&0,	\\
%23
\theta_1\theta_5\, \xi_{{3,14}}-\theta_4\theta_6\, \xi_{{2,15}}-\theta_7\theta_8\, \xi_{{9,16}}&=&0, 	& 
\theta_1\theta_5\, \xi_{{9,16}}+\theta_4\theta_6\, \xi_{{10,12}}-\theta_7\theta_8\, \xi_{{3,14}}&=&0,	\\
%24
\theta_3\theta_8\, \xi_{{5,9}}+\theta_4\theta_{10}\, \xi_{{11,15}}-\theta_5\theta_9\, \xi_{{3,8}}&=&0, 	& 
\theta_3\theta_8\, \xi_{{11,15}}+\theta_4\theta_{10}\, \xi_{{5,9}}-\theta_5\theta_9\, \xi_{{4,10}}&=&0,	\\
%25
\theta_3\theta_{10}\, \xi_{{6,7}}+\theta_4\theta_8\, \xi_{{12,14}}-\theta_6\theta_7\, \xi_{{3,10}}&=&0, 	& 
\theta_3\theta_{10}\, \xi_{{12,14}}+\theta_4\theta_8\, \xi_{{6,7}}-\theta_6\theta_7\, \xi_{{4,8}}&=&0,	\\
%26
\theta_1\theta_{10}\, \xi_{{4,16}}+\theta_2\theta_8\, \xi_{{3,13}}-\theta_6\theta_9\, \xi_{{5,12}}&=&0, 	& 
\theta_1\theta_{10}\, \xi_{{3,13}}+\theta_2\theta_8\, \xi_{{4,16}}-\theta_6\theta_9\, \xi_{{7,15}}&=&0,	\\
%27
\theta_1\theta_8\, \xi_{{4,13}}+\theta_2\theta_{10}\, \xi_{{3,16}}-\theta_5\theta_7\, \xi_{{6,11}}&=&0, 	& 
\theta_1\theta_8\, \xi_{{3,16}}+\theta_2\theta_{10}\, \xi_{{4,13}}-\theta_5\theta_7\, \xi_{{9,14}}&=&0,	\\
%28
\theta_1\theta_3\, \xi_{{8,16}}+\theta_2\theta_4\, \xi_{{10,13}}-\theta_7\theta_9\, \xi_{{5,14}}&=&0, 	& 
\theta_1\theta_3\, \xi_{{10,13}}+\theta_2\theta_4\, \xi_{{8,16}}-\theta_7\theta_9\, \xi_{{6,15}}&=&0,	\\
%29
\theta_1\theta_2\, \xi_{{5,15}}-\theta_3\theta_4\, \xi_{{6,14}}-\theta_8\theta_{10}\, \xi_{{7,12}}&=&0, 	& 
\theta_1\theta_2\, \xi_{{7,12}}-\theta_3\theta_4\, \xi_{{9,11}}-\theta_8\theta_{10}\, \xi_{{5,15}}&=&0,	\\
%30
\theta_1\theta_4\, \xi_{{8,13}}+\theta_2\theta_3\, \xi_{{10,16}}-\theta_5\theta_6\, \xi_{{7,11}}&=&0, 	& 
\theta_1\theta_4\, \xi_{{10,16}}+\theta_2\theta_3\, \xi_{{8,13}}-\theta_5\theta_6\, \xi_{{9,12}}&=&0.
\end{array}}
\end{equation}
\newpage
\section{Mumford relations II}
\label{app:MumfordRelationsTropes}
In terms of the variables $\mathsf{t}_{a,b}$ used in Proposition~\ref{prop:MMFT}, the ideal generated by the Mumford relations in Equations~(\ref{Eqn:MumfordBimonomial}) coincides with the ideal generated by the following equations whose coefficients are determined by the Rosenhain parameters of the genus-two curve $\mathcal{C}$ in Equation~(\ref{Eq:Rosenhain}):
\begin{equation}
\label{Kummer:topesLRc}
\scalemath{0.7}{
\begin{array}{lcllcl}
\mathsf{t}_{1,2}-\mathsf{t}_{156,256}+\mathsf{t}_{146,246} 								&=&0, 	& 
\lambda_3\mathsf{t}_{1,2}-\mathsf{t}_{136,236}+\mathsf{t}_{146,246} 						&=&0,	\\
\mathsf{t}_{1,3}-\mathsf{t}_{156,356}+\mathsf{t}_{146,346}								&=&0, 	& 
\lambda_2\mathsf{t}_{1,3}-\mathsf{t}_{126,236}+\mathsf{t}_{146,346}						&=&0,	\\
\left( \lambda_2-1 \right) \mathsf{t}_{1,4}-\mathsf{t}_{126,246}+\mathsf{t}_{156,456}				&=&0, 	& 
\left( \lambda_2-\lambda_3 \right) \mathsf{t}_{1,4}-\mathsf{t}_{126,246}+\mathsf{t}_{136,346}		&=&0,	\\
\lambda_2\mathsf{t}_{1,5}-\mathsf{t}_{126,256}+\mathsf{t}_{146,456}						&=&0, 	& 
\lambda_3\mathsf{t}_{1,5}-\mathsf{t}_{136,356}+\mathsf{t}_{146,456}				&=&0,	\\
\left( \lambda_3-\lambda_2 \right) \mathsf{t}_{1,6}+\mathsf{t}_{256,346}-\mathsf{t}_{246,356}		&=&0, 	& 
\lambda_2 \left( \lambda_3-1 \right) \mathsf{t}_{1,6}+\mathsf{t}_{256,346}-\mathsf{t}_{236,456}		&=&0,	\\
\mathsf{t}_{2,3}-\mathsf{t}_{256,356}+\mathsf{t}_{246,346}								&=&0, 	& 
\lambda_1\mathsf{t}_{2,3}-\mathsf{t}_{126,136}+\mathsf{t}_{246,346}						&=&0,	\\
\left( \lambda_1-1 \right) \mathsf{t}_{2,4}-\mathsf{t}_{126,146}+\mathsf{t}_{256,456}				&=&0, 	& 
\left( \lambda_3-1 \right) \mathsf{t}_{2,4}+\mathsf{t}_{256,456}-\mathsf{t}_{236,346}				&=&0,	\\
\lambda_1\mathsf{t}_{2,5}-\mathsf{t}_{126,156}+\mathsf{t}_{246,456}						&=&0, 	& 
\lambda_3\mathsf{t}_{2,5}-\mathsf{t}_{236,356}+\mathsf{t}_{246,456}						&=&0,	\\
\left( \lambda_1-\lambda_3 \right) \mathsf{t}_{2,6}+\mathsf{t}_{146,356}-\mathsf{t}_{156,346}		&=&0, 	& 
\lambda_1 \left( \lambda_3-1 \right) \mathsf{t}_{2,6}+\mathsf{t}_{156,346}-\mathsf{t}_{136,456}		&=&0,	\\
\left( \lambda_1-1 \right) \mathsf{t}_{3,4}-\mathsf{t}_{136,146}+\mathsf{t}_{356,456}				&=&0, 	&
\left( \lambda_2-1 \right) \mathsf{t}_{3,4}+\mathsf{t}_{356,456}-\mathsf{t}_{236,246}				&=&0,	\\
\lambda_1\mathsf{t}_{3,5}-\mathsf{t}_{136,156}+\mathsf{t}_{346,456}						&=&0, 	&
\lambda_2\mathsf{t}_{3,5}-\mathsf{t}_{236,256}+\mathsf{t}_{346,456}						&=&0,	\\
\left( \lambda_1-\lambda_2 \right) \mathsf{t}_{3,6}+\mathsf{t}_{146,256}-\mathsf{t}_{156,246}		&=&0, 	&
\lambda_1 \left( \lambda_2-1 \right) \mathsf{t}_{3,6}+\mathsf{t}_{156,246}-\mathsf{t}_{126,456}		&=&0,	\\
\left( \lambda_1-\lambda_2 \right) \mathsf{t}_{4,5}-\mathsf{t}_{146,156}+\mathsf{t}_{246,256}		&=&0, 	&
\left( \lambda_1-\lambda_3 \right) \mathsf{t}_{4,5}-\mathsf{t}_{146,156}+\mathsf{t}_{346,356}		&=&0,	\\
\left( \lambda_2-1 \right)  \left( \lambda_1-\lambda_3 \right) \mathsf{t}_{4,6}+\mathsf{t}_{156,236}-\mathsf{t}_{126,356}&=&0, 	&
\left( \lambda_2-\lambda_3 \right)  \left( \lambda_1-1 \right) \mathsf{t}_{4,6}+\mathsf{t}_{136,256}-\mathsf{t}_{126,356} &=&0,	\\
\lambda_2 \left( \lambda_1-\lambda_3 \right) \mathsf{t}_{5,6}+\mathsf{t}_{146,236}-\mathsf{t}_{126,346}&=&0, 	&
\left( \lambda_2-\lambda_3 \right) \lambda_1\mathsf{t}_{5,6}+\mathsf{t}_{136,246}-\mathsf{t}_{126,346}&=&0,	\\
\mathsf{t}_{1,126}-\lambda_1\mathsf{t}_{5,256}+ \left( \lambda_1-1 \right) \mathsf{t}_{4,246}		&=&0, 	& 
\mathsf{t}_{3,236}-\lambda_3\mathsf{t}_{5,256}+ \left( \lambda_3-1 \right) \mathsf{t}_{4,246}		&=&0,	\\
\mathsf{t}_{1,136}-\lambda_1\mathsf{t}_{5,356}+ \left( \lambda_1-1 \right) \mathsf{t}_{4,346}		&=&0, 	& 
\mathsf{t}_{2,236}-\lambda_2\mathsf{t}_{5,356}+ \left( \lambda_2-1 \right) \mathsf{t}_{4,346}		&=&0,	\\
\left( \lambda_2-1 \right) \mathsf{t}_{1,146}- \left( \lambda_1-1 \right) \mathsf{t}_{2,246}+ \left( \lambda_1-\lambda_2 \right) \mathsf{t}_{5,456} &=&0, 	& 
\left( \lambda_2-\lambda_3 \right) \mathsf{t}_{1,146}+ \left( \lambda_3-\lambda_1 \right) \mathsf{t}_{2,246}+ \left( \lambda_1-\lambda_2 \right) \mathsf{t}_{3,346} &=&0,	\\
\mathsf{t}_{2,126}-\lambda_2\mathsf{t}_{5,156}+ \left( \lambda_2-1 \right) \mathsf{t}_{4,146}		&=&0, 	& 
\mathsf{t}_{3,136}-\lambda_3\mathsf{t}_{5,156}+ \left( \lambda_3-1 \right) \mathsf{t}_{4,146}		&=&0,	\\
\mathsf{t}_{6,126}+\mathsf{t}_{5,346}-\mathsf{t}_{4,356}									&=&0, 	& 
\lambda_3\mathsf{t}_{6,126}+\mathsf{t}_{3,456}-\mathsf{t}_{4,356}							&=&0,	\\
\mathsf{t}_{6,136}+\mathsf{t}_{5,246}-\mathsf{t}_{4,256}									&=&0, 	& 
\lambda_2\mathsf{t}_{6,136}+\mathsf{t}_{2,456}-\mathsf{t}_{4,256}							&=&0,	\\
\left( \lambda_2-1 \right) \mathsf{t}_{6,146}+\mathsf{t}_{2,356}-\mathsf{t}_{5,236}				&=&0, 	&
\left( \lambda_3-1 \right) \mathsf{t}_{6,146}+\mathsf{t}_{3,256}-\mathsf{t}_{5,236}				&=&0,	\\
\mathsf{t}_{6,236}+\mathsf{t}_{5,146}-\mathsf{t}_{4,156}									&=&0, 	&
\lambda_1\mathsf{t}_{6,236}+\mathsf{t}_{1,456}-\mathsf{t}_{4,156}							&=&0,	\\
\left( \lambda_1-1 \right) \mathsf{t}_{6,246}+\mathsf{t}_{1,356}-\mathsf{t}_{5,136}				&=&0, 	&
\left( \lambda_3-1 \right) \mathsf{t}_{6,246}+\mathsf{t}_{3,156}-\mathsf{t}_{5,136}				&=&0,	\\
\left( \lambda_1-1 \right) \mathsf{t}_{6,346}+\mathsf{t}_{1,256}-\mathsf{t}_{5,126}				&=&0, 	&
\left( \lambda_2-1 \right) \mathsf{t}_{6,346}+\mathsf{t}_{2,156}-\mathsf{t}_{5,126}				&=&0,	\\
\lambda_2\mathsf{t}_{6,156}+\mathsf{t}_{2,346}-\mathsf{t}_{4,236}							&=&0, 	&
\lambda_3\mathsf{t}_{6,156}+\mathsf{t}_{3,246}-\mathsf{t}_{4,236}							&=&0,	\\
\lambda_1\mathsf{t}_{6,256}+\mathsf{t}_{1,346}-\mathsf{t}_{4,136}							&=&0, 	&
\lambda_3\mathsf{t}_{6,256}+\mathsf{t}_{3,146}-\mathsf{t}_{4,136}							&=&0,	\\
\lambda_1\mathsf{t}_{6,356}+\mathsf{t}_{1,246}-\mathsf{t}_{4,126}							&=&0, 	&
\lambda_2\mathsf{t}_{6,356}+\mathsf{t}_{2,146}-\mathsf{t}_{4,126}							&=&0,	\\
\left( \lambda_1-\lambda_2 \right) \mathsf{t}_{6,456}+\mathsf{t}_{1,236}-\mathsf{t}_{2,136}		&=&0, 	&
\left( \lambda_1-\lambda_3 \right) \mathsf{t}_{6,456}+\mathsf{t}_{1,236}-\mathsf{t}_{3,126}		&=&0,	\\
\lambda_2\mathsf{t}_{1,156}-\lambda_1\mathsf{t}_{2,256}+ \left( \lambda_1-\lambda_2 \right) \mathsf{t}_{4,456} &=&0, 	&
\lambda_3\mathsf{t}_{1,156}-\lambda_1\mathsf{t}_{3,356}+ \left( \lambda_1-\lambda_3 \right) \mathsf{t}_{4,456} &=&0.	\\
\end{array}}
\end{equation}
\end{appendix}
\end{document}